\documentclass[10pt]{article}
\usepackage{amsmath,amssymb,amsthm}
\usepackage{bm}
\usepackage[margin=1in]{geometry}
\usepackage{graphicx}
\usepackage[font=scriptsize]{subcaption}
\usepackage[font=scriptsize]{caption}
\usepackage{color}
\usepackage{scrextend}
\usepackage{hyperref}
\usepackage[affil-it]{authblk}
\usepackage[title]{appendix}

\newcommand{\RR}{\mathbb{R}}
\newcommand{\CC}{\mathbb{C}}
\newcommand{\ZZ}{\mathbb{Z}}
\newcommand{\OO}[1]{\mathcal{O}\paren{#1}}
\newcommand{\paren}[1]{\left( #1 \right)}
\newcommand{\brak}[1]{\left[ #1 \right]}

\newcommand{\wh}[1]{\widehat{#1}}

\newcommand{\norm}[1]{\left\lVert#1\right\rVert}
\newcommand{\abs}[1]{\left\lvert#1\right\rvert}
\renewcommand{\Re}{\operatorname{Re}}
\renewcommand{\Im}{\operatorname{Im}}
\newcommand{\bea}{\begin{eqnarray}}
\newcommand{\eea}{\end{eqnarray}}
\newcommand{\be}{\begin{equation}}
\newcommand{\ee}{\end{equation}}

\newcommand{\Schrod}{Schr\"{o}dinger\ }
\newcommand{\uhat}{\wh{u}}
\newcommand{\Vuhat}{\wh{(Vu)}}
\newcommand{\uzhat}{\wh{u}_0}
\newcommand{\phimax}{\varphi^{\max}}
\newcommand{\calE}{\mathcal{E}}
\newcommand{\calA}{\mathcal{A}}
\newcommand{\calC}{\mathcal{C}}

\newcommand{\err}{E}
\newcommand{\epsm}{\bm{\epsilon}}
\newcommand{\uzhatk}{(\wh{u_0})_k}
\newcommand{\nr}{n_r}

\newcommand{\NE}{N^{(\calE)}}
\newcommand{\NC}{N^{(\calC)}}
\newcommand{\NT}{n^{(c)}}

\newtheorem{theorem}{Theorem}
\newtheorem{lemma}{Lemma}
\newtheorem{remark}{Remark}
\newtheorem{definition}{Definition}

\title{A high-order integral equation-based solver for the time-dependent \Schrod
equation}

\author[1,2]{Jason Kaye \thanks{Email: jkaye@flatironinstitute.org}}
\author[1]{Alex Barnett}
\author[1,2]{Leslie Greengard}
\affil[1]{Flatiron Institute, Simons Foundation}
\affil[2]{Courant Institute of Mathematical Sciences,
    New York University}

\date{}

\begin{document}

\maketitle

\begin{abstract}

  We introduce a numerical method for the solution of the time-dependent
  \Schrod equation with a smooth potential, based on its reformulation as a Volterra
  integral equation. We present versions of the method both for periodic
  boundary conditions, and for free space problems with compactly supported
  initial data and potential. A spatially uniform
  electric field may be included, making the solver applicable to
  simulations of light-matter interaction.

  The primary computational challenge in using the Volterra formulation
  is the application of a space-time history dependent integral
  operator. This may be accomplished by projecting the solution onto a
  set of Fourier modes, and updating their coefficients from one time
  step to the next by a simple recurrence. In the periodic case, the
  modes are those of the usual Fourier series, and the fast Fourier
  transform (FFT) is used to alternate between physical and frequency
  domain grids. In the free space case, the oscillatory behavior of the
  spectral Green's function leads us to use a set of complex-frequency Fourier
  modes obtained by discretizing a contour deformation of the
  inverse Fourier transform, and we develop a corresponding fast transform based
  on the FFT.

  Our approach is related to pseudo-spectral methods, but applied to an
  integral rather than the usual differential formulation. This has
  several advantages: it avoids the need for artificial boundary
  conditions, admits simple, inexpensive high-order implicit time
  marching schemes, and naturally includes time-dependent
  potentials. We present examples in one and two dimensions showing
  spectral accuracy in space and eighth-order accuracy in time for both
  periodic and free space problems.
  
\end{abstract}

\section{Introduction}

We consider the numerical solution of the non-dimensionalized
$d$-dimensional time-dependent \Schrod
equation (TDSE) with a uniform advective potential, given by
\be
\begin{aligned}
  \label{eq:schrodfree}
    i \partial_t u(x,t) &= - \nabla^2 u(x,t) +
    V(x,t) \, u(x,t) + i A(t) \cdot \nabla u(x,t), 
    \quad  x\in D \subseteq \RR^d, \quad t \in (0,T], \\
    u(x,0) &= u_0(x),\qquad x\in D. \\
\end{aligned}
\ee
Here, $u$ is a complex-valued
wavefunction, $V$ a $C^\infty$-smooth
scalar binding or scattering potential, $A:
[0,T] \to \RR^d$ a $C^\infty$ electromagnetic vector potential,
and $u_0$ a $C^\infty$ initial wavefunction with $\|u_0\|_{L^2(D)}=1$.
The first term on
the right hand side corresponds to the kinetic energy of the system, and
the second to the potential energy. The third term 
is of particular interest in simulations of light-matter
interaction, in which $A$ is often taken to be spatially uniform---%
the so-called dipole approximation \cite{bandrauk13}---and induces a spatially uniform electric field. When $V = 0$ and $A = 0$, we refer to \eqref{eq:schrodfree} as the 
  free particle equation, and when $V = 0$ but $A \neq 0$,
  we refer to it as the free particle equation with advection.

We will consider both the periodic and free space formulations of
\eqref{eq:schrodfree}. In the periodic case, we take 
$D = [-\pi,\pi]^d$, and assume that $u_0$, $V$ and $u$ are spatially periodic on this
domain.
In the free space case, we take $D=\RR^d$, and assume that $u(\cdot,t)$ is in
the Schwartz space for each $t$,
and that $u_0$ and $V$ are compactly supported in the box
$[-1,1]^d$. A purely time-dependent function may be added
to $V$ by making a gauge transformation of $u$.

Note that an equivalent formulation of \eqref{eq:schrodfree} can be obtained by 
removing the gradient term $A(t) \cdot \nabla u(x,t)$ and
adding an unbounded term of the form $E(t) \cdot x$ to 
$V(x,t)$. This is typically
referred to as the {\em length gauge} formulation, and ours above as
the {\em velocity gauge} formulation \cite{bandrauk13}.

The literature on the numerical solution of the TDSE
is extensive, and we refer the reader to 
\cite{leforestier91,blanes00,lubich02_2,kormann08,blanes15,blanes17,bader18,iserles18} 
for good summaries of the state of the art. The papers
\cite{castro04,kidd17,pueyo18} provide careful comparisons of a selection of methods in
the context of time-dependent density functional theory. 
Before describing our approach in detail,
it is worth noting that the dominant framework for existing
numerical methods involves implementing a direct
approximation of the {\em unitary single time step propagator}. More specifically,
assuming first that $A = 0$ and $V = V(x)$ is time-independent, the propagator
is given by the formula
\begin{equation} \label{eq:tindprop}
  u(\cdot,t + \Delta t) = e^{-i \mathcal{H} \Delta t} u(\cdot,t).
\end{equation}
Here $\mathcal{H} = -\nabla^2 + V$ is the constant system Hamiltonian. A typical
method of this type involves discretizing $\mathcal{H}$
and, at each time step, applying the resulting matrix exponential to a
vector by one of many approaches, which include
operator splitting, polynomial approximation of the exponential
by Taylor expansion or Chebyshev interpolation, and Lanczos
iteration \cite{leforestier91}. For the general case with time-dependent $V$
and the electromagnetic field term included, the unitary solution operator in
the length gauge is given by
\begin{equation} \label{eq:tdepprop}
  u(\cdot,t+\Delta t) = \mathcal{T} \paren{e^{-i \int_t^{t+\Delta t}
  \mathcal{H}(s) \,
ds}} u(\cdot,t).
\end{equation}
Here $\mathcal{T}$ is the time-ordering symbol, which is needed to correct for the
lack of commutativity of the Hamiltonian operator $\mathcal{H}(t)$ at different
points in time \cite[Sec. 3.6]{lin19}.
Implementing the propagator in this form
is impractical, and instead it is typical to use a ``Magnus" or ``quasi-Magnus"
expansion to reduce this formula to one of the form \eqref{eq:tindprop}, with a more
complicated time-independent Hamiltonian 
$\mathcal{H}$ \cite{magnus,iserles00,blanes00_2,hochbruck03,iserles18}.

Here, we explore an alternative approach, which we describe first for the free
space case $D=\RR^d$.
If $V = 0$, then the solution of \eqref{eq:schrodfree} is
given by the explicit integral representation
\begin{equation} \label{eq:uhomogsoln}
  u(x,t) = \int_{\RR^d} G(x-y,t,0) u_0(y) \, dy,
\end{equation}
where $G(x,t,s)$ is the Green's function for the free particle
TDSE with advection \cite{ermolaev99,kaye18},
\begin{equation} \label{eq:gfun}
  G(x,t,s) := \frac{\exp\paren{i \abs{x + \varphi(t) - \varphi(s)}^2 /
4 (t-s)}}{(4 \pi i (t-s))^{d/2}}
\end{equation}
with
\be
\varphi(t) := \int_0^t A(s) \, ds~.
\label{phi}
\ee
This Green's function reduces to the ordinary free particle Green's function when $A = 0$.
The formula \eqref{eq:uhomogsoln} may be viewed as a realization of
the formal propagator discussed above in the free particle setting.
However, rather than including the potential energy
term in the propagator, we will treat it as a source term for the
free particle equation.
This leads to the following Volterra-type integral equation, which is
called the Duhamel principle in the mathematics literature 
and the Lippmann--Schwinger equation in physics:
\begin{equation} \label{eq:lippmannschwinger}
  u(x,t) = \int_{\RR^d} G(x-y,t,0) u_0(y) \, dy
  \;-\;
  i \int_0^t \int_{\RR^d} G(x-y,t,s) (Vu)(y,s) \, dy \, ds.
\end{equation}
Here we have used the notation $(V u)(x,t) \equiv V(x,t)
u(x,t)$.
It is straightforward to verify 
that \eqref{eq:lippmannschwinger} satisfies \eqref{eq:schrodfree}.
Note that \eqref{eq:lippmannschwinger} represents $u$ in terms of $u_0$ and its
history over the spatial support of $Vu$, and hence, of $V$. 
A similar formula may be obtained for the periodic problem using the periodic
Green's function.

This integral formulation offers a variety of significant benefits,
to be discussed shortly. However, as written, it does not suggest a practical computational
scheme. In particular, the potential term 
depends on the full spacetime history of the solution, and is therefore prohibitively
expensive to evaluate directly at a large collection of time steps. For
$N$ time steps on a domain discretized by $M$ points, the naive cost---even ignoring the
difficult problem of quadratures for the highly oscillatory kernel $G$---%
is at least of the order $\OO{M^2 N^2}$. 
Moreover, the $\OO{M N}$ memory required to store
the spacetime history of the solution is impractical for large-scale problems. 
Thus, in the absence
of suitable fast and memory-efficient algorithms,
the Volterra integral equation approach has been largely ignored. 
Below, we develop
spectral, fast Fourier transform (FFT)-based algorithms which reduce 
these costs to near-optimal complexity in both the periodic and free space settings.

For low-order accuracy in time, we obtain a method which in many ways resembles
classical pseudo-spectral operator splitting schemes for periodic problems. 
The similarities include spectral accuracy in space,
quasi-optimal cost, and optimal memory requirements.
However, our approach permits the application of simple
high-order accurate multistep marching schemes which require the same number of FFTs per time step as low-order
discretizations.
Furthermore, our method has the same form for time-independent and
time-dependent potentials $V$. By contrast, the construction of high-order
splitting-based schemes is rather involved even for time-independent
potentials, and more so for time-dependent potentials. For
time-independent potentials,
high-order splitting formulas with complex coefficients have
been constructed directly \cite{bandrauk06,castella09,hansenostermann,blanes13}, and
deferred correction procedures can be applied to increase the order of
accuracy of low-order splitting methods 
\cite{bourlioux,christlieb,duarte,hagstromzhou,zhang2018}. In both
cases, the cost per time step increases substantially with the order
of accuracy. For time-dependent potentials, operator splitting and other 
propagator-based methods require high-order Magnus or quasi-Magnus expansions
to handle the time-ordering operator in \eqref{eq:tdepprop}, as
mentioned above. We note that within our framework, multistage
Runge-Kutta-style schemes are also available in addition to the
multistep schemes, but for these
the cost grows with the desired order of accuracy.

Two other general properties of our method are worth noting, both of which follow
from its use of a second-kind Volterra integral equation formulation.
First, because of the $\delta$-function property of the free particle Green's
function, the linear systems generated by simple
implicit time discretizations are diagonal. As a result,
implicit time marching is no more expensive than explicit
marching. By contrast, implicit methods based on
semi-discretizing in space and recasting the PDE as a system of ODEs 
(i.e.\ the {\em method of lines} \cite[Sec.~9.2]{leveque})
typically require
the solution of a sparse linear system at each time step. Second, the method is
insensitive to over-resolution in space, since the spatial grid is only
used to discretize integral operators. Many existing methods, 
like those utilizing polynomial approximations of matrix exponentials, suffer
from stiffness induced by the large spectral range of discretizations of the 
kinetic energy operator \cite{leforestier91,blanes15,kidd17}.

In the free space setting, the integral equation approach overcomes
a more fundamental limitation of standard methods.
In particular,
numerical methods based on direct discretization of the PDE require the solution 
to be represented on a finite computational domain $\Omega$ rather than
the infinite domain $D=\RR^d$. However, it is common for the support of
the wavefunction $u(x,t)$ to radiate beyond the boundary of
any reasonably-sized domain $\Omega$, for instance when simulating the excitation of a particle
from a bound state to a continuum state by an applied field.
In this case, 
care must be taken to avoid spurious boundary reflections.
As a result, a great deal of research
has been devoted to the design of algorithms which permit the 
imposition of conditions on the boundary of $\Omega$,
assumed to enclose the support of $u_0$ and $V$, which mimic radiation
into free space.

By and large, existing approaches to the approximation of radiative boundary conditions for the TDSE fall into
two broad categories. The first consists of methods which modify the
underlying equation near the boundary of $\Omega$ so as to dampen outgoing components of
the solution. These ``absorbing region" methods include the method
of mask functions, complex absorbing potentials,
exterior complex scaling, and perfectly matched layers
\cite{degiovannini15,weinmuller17,antoine17}. They are by far the more
common approach in practical calculations. While these methods are, in
principle, straightforward to combine with existing propagation schemes and
are often effective, they
typically involve parameters whose tuning is problem-dependent,
making them difficult to use in a robust manner. 
In the second category are methods which implement exact 
transparent boundary conditions (TBCs), for
which the associated solution is equal to the
restriction of the free space solution to the computational domain. The exact
conditions come with a mathematical guarantee of correctness, but are
prohibitively expensive to implement without suitable fast, memory-efficient 
algorithms. A variety of such algorithms have been proposed, mostly
for the case in which $A = 0$ and the computational domain is taken to
be an interval in $\RR$ \cite{baskakov91,lubich02,jiang04,schadle06}, a disk
in $\RR^2$ \cite{han04,jiang08}, or a ball in $\RR^3$ \cite{han07}.
Fewer efficient algorithms exist for the more
computationally convenient case of a rectangle in $\RR^2$
\cite{feshchenko11}, a box in
$\RR^3$ \cite{feshchenko13}, or arbitrary domains
\cite{ermolaev99,schadle02,kaye18}. Some work has extended these approaches
to the case in which
$A \neq 0$ \cite{ermolaev99,lorin09,vaibhav15,feshchenko17,kaye18}, but
corresponding fast algorithms are again lacking, particularly for dimensions greater
than one.
Finally, we note that there are methods which make purely
local approximations of exact conditions \cite{engquist77,antoine01,antoine04,antoine08}, and methods
which implement exact, nonlocal TBCs for specific time discretization schemes
\cite{arnold03,arnold12,ji18}. The 
papers \cite{antoine08,antoine17} contain useful introductions to many of the methods
mentioned above, and a more thorough collection of references.

Nevertheless, although significant progress has been made,
the accurate treatment of artificial boundaries remains an
ongoing challenge in large-scale simulations.
Using the formula \eqref{eq:lippmannschwinger}, the issue of artificial boundary
conditions is avoided entirely,
since the spatial integrals can simply be truncated to a box containing the support of
$u_0$ and $V$. 
This benefit has been noted by others
\cite{hoffman90,sharafeddin92}, but has not been exploited previously because 
of the computational obstacles discussed above. 
This was the primary motivation for the present work.

The derivation of our method begins from the 
Fourier domain representation of the equation \eqref{eq:lippmannschwinger}, which
leads to a system of Volterra integral equations coupled
through the potential $V$. These integral equations can be
rewritten in recurrence form, permitting the Fourier representation 
to be advanced analytically for one time step, with a local update.
In the periodic case, $u(x,t)$ is simply represented as a Fourier series, and 
the recurrence relation applies to the 
discrete Fourier coefficients. The spatial coupling
induced by the potential $V$ is computed in the physical domain, in the style
of a pseudo-spectral method, with the FFT used to 
accelerate the mapping between the physical and frequency domains. If
the box $D=[-\pi,\pi]^d$ is discretized by $M$ grid points per dimension,
then the cost per time step is 
$\OO{M^d \log M}$, and the memory requirements are of the order $\OO{M^d}$, 
as in standard pseudo-spectral methods.

In the free space case, the 
classical Fourier integral representation of $u(x,t)$ is so oscillatory that
the corresponding method would require $\OO{M^{2d} (\log M) T}$ work per
time step. We will show that, by a suitable {\em
contour deformation of the Fourier integral}
into the complex plane, we can obtain
a significantly more efficient representation. A recurrence can still be
used to advance the resulting complex-frequency
coefficients, and an FFT-based algorithm allows us to accelerate
the transform between the 
physical domain and these coefficients. If $A =
0$, the asymptotic cost of the resulting method per time step is only slightly larger than
that for the periodic case: it is $\OO{M \log M + \log T}$ per time
step in one dimension, $\OO{M^2 \log M + M \log T + \log^2 T}$ in two
dimensions, and $\OO{M^3 \log M + M^2 \log T + M \log^2 T + \log^3 T}$
in three dimensions. For applied
fields $A(t)$ for which the so-called {\em quiver radius}---the maximum
advective excursion of a free wavepacket---is larger than the domain size, the
cost of the method scales quasi-linearly with the quiver radius in
each dimension as well. Thus, for linearly-polarized fields, the cost grows
by a factor approximately equal to the quiver radius.
The memory requirements are also near-optimal, of the order $\OO{\paren{M + \log T}^d}$. 

We begin by considering the periodic case in Section \ref{sec:periodic}
and show how the integral equation viewpoint leads to simple high-order
time marching methods. There is significant overlap between this method
and that for the free space case, presented in Section \ref{sec:free},
but the context is simpler. In particular, whereas in the periodic case
we use the standard FFT to move between the physical and frequency
domains, in the free space case we require a more specialized fast
algorithm to move between the physical domain and a complex-frequency
domain. This algorithm, based on the FFT, is presented in Section
\ref{sec:fftgamma}. In Section \ref{sec:analysis}, we provide a detailed
analysis of the computational cost associated with our complex-frequency
representation of the solution. Section \ref{sec:results} contains
demonstrations of a high-order accurate implementation of our method for several model problems.

\section{The periodic case} \label{sec:periodic}

We recall that any smooth periodic function $f(x)$ on $[-\pi,\pi]^d$
can be represented as a Fourier series 
\[ f(x) = \sum_{k \in \ZZ^d} \wh{f}_k e^{i k \cdot x},\]
with Fourier coefficients given by the periodic Fourier transform
\begin{equation} \label{eq:ftper}
  \wh{f}_k := \frac{1}{(2 \pi)^d} \int_{[-\pi,\pi]^d} e^{-i k \cdot x} f(x) \, dx
\end{equation}
for $k \in \ZZ^d$. 
Suppose now $u$ satisfies \eqref{eq:schrodfree} with periodic boundary
conditions, and smooth, periodic $u_0$ and $V$. 
We can represent $u$ as a Fourier series:
\[u(x,t) = \sum_{k \in \ZZ^d} \uhat_k(t)
e^{i k \cdot x}.\]
Taking the periodic Fourier transform of the
governing equation, we find that each
$\uhat_k$ satisfies an ordinary differential equation (ODE):
\begin{equation*}
  \begin{aligned}
    i \uhat_k'(t) &= \paren{\|k\|^2 - k \cdot A(t)} \uhat_k(t) +
    \Vuhat_k(t), \quad t \in (0,T], \\
    \uhat_k(0) &= \uzhatk. &
  \end{aligned}
\end{equation*}
$(Vu)(x,t)$ is itself periodic in space and, like
$u(x,t)$, may be represented by a Fourier series. 
Treating $\Vuhat_k$ as an inhomogeneity, we can solve this ODE
by the variation of parameters formula. We obtain
\begin{equation} \label{eq:uhatkinteq}
  \uhat_k(t) = e^{-i \|k\|^2 t + i k \cdot \varphi(t)} \uzhatk - i
  \int_0^t e^{-i \|k\|^2 (t-s) + i k \cdot (\varphi(t) - \varphi(s))}
  \Vuhat_k(s) \, ds.
\end{equation}
Note that
\eqref{eq:uhatkinteq} is simply the Fourier transform of
the periodic version of the Duhamel formula \eqref{eq:lippmannschwinger}. 
It represents $\uhat_k(t)$
in terms of initial data and $Vu$. When $V = 0$, it
is an explicit formula for $\uhat_k(t)$. Otherwise, it is 
impractical for computation, as written, because it couples
$\uhat_k(t)$ to its entire spacetime history. 

\subsection{The periodic marching scheme} \label{sec:perscheme}

We start by observing that \eqref{eq:uhatkinteq} can be 
reformulated as a recurrence in time.

\begin{lemma}[Discrete spectral evolution] \label{lem:discspecevol}
Let $\Delta t > 0$ be a time step size. The evolution formula \eqref{eq:uhatkinteq}
can be written without explicit history dependence in the form
\begin{equation}\label{eq:uhatkrecur}
  \uhat_k(t) = e^{-i \|k\|^2 \Delta t + i k \cdot \paren{\varphi(t) -
  \varphi(t-\Delta t)}} \uhat_k(t-\Delta t) - i \int_{t - \Delta t}^t
  e^{-i \|k\|^2 (t-s) +
  i k \cdot (\varphi(t) - \varphi(s))} \Vuhat_k(s) \, ds.
\end{equation}
Using the two-point trapezoidal rule for the update integral, we obtain
  the following recurrence:
\begin{equation}\label{eq:uhatkrecurdisc}
\uhat_k(t) + i \frac{\Delta t}{2} \Vuhat_k(t) \approx e^{-i \|k\|^2
\Delta t + i k \cdot \paren{\varphi(t) - \varphi(t-\Delta t)}} \paren{
  \uhat_k(t-\Delta t)  - i  \frac{\Delta t}{2} \Vuhat_k(t - \Delta t)}.
\end{equation}
\end{lemma}
\begin{proof}
  The equation
  \eqref{eq:uhatkinteq} may be rewritten as
\begin{multline*}
  \uhat_k(t) = e^{-i \|k\|^2 \Delta t + i k \cdot \paren{\varphi(t) -
  \varphi(t-\Delta t)}} \\ 
    \times \left(e^{-i \|k\|^2 (t-\Delta t) + i k \cdot
     \varphi(t - \Delta t)} \uzhatk
    - i \int_0^{t-\Delta t} e^{-i \|k\|^2 (t-\Delta t-s) + i
    k \cdot (\varphi(t - \Delta t) - \varphi(s))} 
  \Vuhat_k(s) \, ds \right) \\ - i \int_{t - \Delta t}^t e^{-i \|k\|^2 (t-s) +
  i k \cdot (\varphi(t) - \varphi(s))} \Vuhat_k(s) \, ds,
\end{multline*}
  which gives \eqref{eq:uhatkrecur}. Equation \eqref{eq:uhatkrecurdisc}
follows from the quadrature $\int_{t-\Delta t}^t g(s) ds \approx
  \frac{\Delta t}{2}\paren{g(t-\Delta t)+g(t)}$.
\end{proof}
Equation \eqref{eq:uhatkrecur} states that $\uhat_k(t)$ may be represented exactly in terms of its value
$\uhat_k(t-\Delta t)$ at the previous time step
and an update integral which is local in time.
The marching rule \eqref{eq:uhatkrecurdisc} is globally second-order accurate.
A higher-order quadrature rule would yield a higher-order accurate
evolution formula, as discussed in Section \ref{sec:highorder}.

Summing the expression \eqref{eq:uhatkrecurdisc} over all Fourier modes
and dividing by the factor $1 + i \frac{\Delta t}{2} V(x,t)$, we obtain
\begin{equation} \label{eq:utrapk}
  u(x,t) \approx \frac{1}{1 + i \frac{\Delta t}{2} V(x,t)}
  \sum_{k \in \ZZ^d} e^{i k \cdot x} e^{-i \|k\|^2
  \Delta t + i k \cdot \paren{\varphi(t) - \varphi(t-\Delta t)}} \paren{
\uhat_k(t-\Delta t)  - i  \frac{\Delta t}{2} \Vuhat_k(t - \Delta t)}.
\end{equation}
This formula suggests a simple marching scheme, semi-discretized with respect to time.
To obtain
$u(x,t)$ from $u(x,t-\Delta t)$, we transform the quantity
$u(x,t - \Delta t) - i \frac{\Delta t}{2} (Vu)(x,t - \Delta t)$ 
to its Fourier representation,
multiply the $k$th mode by the 
factor indicated in \eqref{eq:utrapk}, sum the resulting Fourier series for each 
$x \in [-\pi,\pi]^d$, and divide the result by $1 + i\frac{\Delta t}{2} V(x,t)$.

To obtain a fully discrete scheme, we need to 
truncate the Fourier series in \eqref{eq:utrapk} and discretize the 
Fourier transform \eqref{eq:ftper}. For simplicity, we
write the formulas for the one-dimensional case. The
$d$-dimensional generalization is straightforward.
Let us denote the frequency truncation parameter by $M$, with $M$ even,
and let 
\[u(x,t) = \sum_{k=-\infty}^{\infty} e^{i k x}
\uhat_k(t) \approx \sum_{k=-M/2}^{M/2-1} e^{i k x}
\uhat_k(t), \]
\[(Vu)(x,t) = \sum_{k=-\infty}^{\infty} e^{i k x}
\Vuhat_k(t) \approx \sum_{k=-M/2}^{M/2-1} e^{i k x}
\Vuhat_k(t).\]
Since $u(x,t)$ and $V(x,t)$ are
smooth and periodic, their Fourier coefficients decay rapidly---
faster than any finite power of $M^{-1}$---and the
truncated representations are said to converge {\em spectrally} or 
{\em superalgebraically}.
Moreover, given M equispaced points $\{x_j\}$ on 
$[-\pi, \pi]$, $x_j = -\pi + 2 \pi j / M$ for $j =
0,\ldots,M-1$, the Fourier transform \eqref{eq:ftper}, used to compute $\uzhatk$ and
$\Vuhat_k(t)$, can be approximated with spectral
accuracy using the periodic trapezoidal rule as
\begin{equation} \label{eq:ftperdisc}
  \wh{f}_k \approx \frac{1}{M} \sum_{j=0}^{M-1} e^{-i 2 \pi j k / M}
  f(x_j),
\end{equation}
for $k = -M/2,\ldots,M/2-1$; see, for example, \cite{boyd2001,trefethen13}
and Remark \ref{rem:ptrerr} below.

Using these approximations in \eqref{eq:uhatkrecurdisc}, we obtain
\begin{equation} \label{eq:utrapkdisc}
  u(x_j,t) \approx \frac{1}{1 + i \frac{\Delta t}{2}
  V(x_j,t)} \sum_{k=-M/2}^{M/2-1} e^{i 2 \pi j k / M} e^{-i \|k\|^2
  \Delta t + i k \cdot \paren{\varphi(t) - \varphi(t-\Delta t)}} \paren{
\uhat_k(t-\Delta t)  - i  \frac{\Delta t}{2} \Vuhat_k(t - \Delta t)}.
\end{equation}
Both the discrete Fourier transform (DFT) in \eqref{eq:ftperdisc} and
the evaluation of the Fourier series at the points $\{ x_j \}$ in
\eqref{eq:utrapkdisc} (the inverse DFT) can be carried out
using the FFT with $\OO{M \log M}$ operations. In the
marching scheme, both the DFT and the inverse DFT are computed
once per time step. Thus, the overall cost of the
fully discrete algorithm is quasi-optimal at $\OO{M \log M}$ work per
time step. Since we only need to store quantities at the current and
previous time steps, the net memory requirement is $\OO{M}$. 

In summary, the Fourier-based marching scheme using the trapezoidal rule
in time is spectrally accurate in space, second-order accurate in time,
quasi-optimal in cost, and optimal in memory. The method in this form therefore
has similar features to a standard pseudo-spectral Strang
splitting method. As discussed in the introduction, the primary
advantage of our approach for the periodic problem is the simplicity of generating higher-order
schemes of various flavors. These are discussed in the next section.

\begin{remark}
  Note that we have used an implicit time discretization for the local update
  integral in \eqref{eq:uhatkrecur}. That is, the trapezoidal rule
  involves the unknown at the new time step.
  By transforming back to the physical domain in \eqref{eq:utrapk}, however, 
  the resulting system is diagonalized, so that inversion is trivial.
  This property is typical of implicit discretizations of Volterra integral
  equations arising from time-dependent parabolic PDEs
  \cite{epperson91,strain94,heatfull}. In particular, see
  \cite{epperson91} for a discussion of this phenomenon from a Green's function
  perspective. 
\end{remark}

\begin{remark} \label{rem:ptrerr}

The number $M$ of frequency modes is chosen to be equal to the number of spatial grid points
not only for simplicity or compatibility with the FFT algorithm, but because the
frequency truncation is intrinsically linked to the grid spacing
required to resolve $u(x,t)$ and $(Vu)(x,t)$ in physical space. 
Indeed, the standard result \cite{trefethen14} on the aliasing error of
the periodic trapezoidal rule \eqref{eq:ftperdisc}
is
\[\wh{f}_k - \frac{1}{M} \sum_{j=0}^{M-1} e^{-i 2 \pi j k / M}
f(x_j) = \sum_{\substack{j=-\infty \\ j \neq 0}}^\infty \wh{f}_{k+jM}~.\]
Thus it suffices to choose the number $M$ of grid points so that the sum
of the Fourier coefficients beyond $\abs{k} = M/2 - 1$ is sufficiently small. 
The rapid decay of the Fourier coefficients for smooth functions is
responsible for the superalgebraic decay mentioned above. In free
space, the relationship between the physical and Fourier domains is
more complicated and their simultaneous discretization will be more
challenging.

\end{remark}

\subsection{Higher-order time discretizations} \label{sec:highorder}

In the preceding section, we discretized the local update time integral in 
\eqref{eq:uhatkrecur} using the two-point trapezoidal rule. We extend this 
now to the broader class of linear multistep schemes, analogous to Adams-type
methods for ODEs \cite[Sec. 5.9]{leveque}.
These lead to high-order marching schemes at a
negligible additional cost.
By way of a brief review, let us consider an update integral like
that in \eqref{eq:uhatkrecur}, which we write more simply for the moment as
\[\int_{t - \Delta t}^t g(s) \, ds.\]
We can approximate $g(s)$ by a polynomial interpolant using its values at
several previous time steps. If the current value $g(t)$ is included in the 
interpolant, the resulting method is said to be implicit; otherwise it is
explicit. More precisely, to generate a
$n$th-order accurate implicit method, we construct a polynomial $p(s)$
of degree at most $n-1$ satisfying the interpolation conditions
\[p(t - j \Delta t) = g(t - j \Delta t),
\qquad
j = 0,\ldots,n-1.
\]
The coefficients of $p(s)$ may be found in terms of the values $\{g(t - j \Delta
t)\}_{j=0}^{n-1}$ by solving a Vandermonde system. Replacing $g(s)$ by
$p(s)$ in the integral and integrating exactly, we find
\[\int_{t - \Delta t}^t g(s) \, ds \approx \int_{t - \Delta t}^t p(s) \,
ds = \Delta t \sum_{j=0}^{n-1} \mu_j g(t - j \Delta t)\]
for some coefficients $\{\mu_j\}_{j=0}^{n-1}$. The coefficients for the
implicit Adams methods up to fifth-order are listed in \cite[Sec.
5.9]{leveque}. The second-order method is the trapezoidal rule used
before, with coefficients $\mu_0 = \mu_1 = 1/2$. Table \ref{tab:adamsimp8} gives the coefficients of
the eighth-order method, which will be used for our numerical experiments in
Section \ref{sec:results}.

\begin{table}[t]
  \renewcommand*{\arraystretch}{1.2}
  \centering
  \begin{tabular}{|c|r|r|r|r|r|r|r|r|}
    \hline
    $j$ & 0 & 1 & 2 & 3 & 4 & 5 & 6 & 7 \\ \hline
    $\mu_j$ & $\frac{5257}{17280}$ & $\frac{139849}{120960}$ &
    $-\frac{4511}{4480}$ & $\frac{123133}{120960}$ &
    $-\frac{88547}{120960}$ & $\frac{1537}{4480}$ &
    $-\frac{11351}{120960}$ & $\frac{275}{24192}$   \\
    \hline
  \end{tabular}
  \caption{Coefficients of the 8th-order implicit Adams method.}
  \label{tab:adamsimp8}
\end{table}

Using this approximation in \eqref{eq:uhatkrecur} yields
\begin{multline*}
  \uhat_k(t) + i \mu_0 \Delta t \Vuhat_k(t) \approx e^{-i \|k\|^2
\Delta t + i k \cdot \paren{\varphi(t) - \varphi(t-\Delta t)}} \uhat_k(t-\Delta t)
  \\ - i \Delta t \sum_{j=1}^{n-1} \mu_j e^{-i \|k\|^2 j \Delta t + i k
  \cdot \paren{\varphi(t) - \varphi(t-j\Delta t)}}  \Vuhat_k(t - j \Delta
  t)
\end{multline*}
  
in place of \eqref{eq:uhatkrecurdisc}, leading to 
\begin{multline} \label{eq:uadamsk}
u(x,t) \approx \frac{1}{1 + i \mu_0 \Delta t V(x,t)}
  \sum_{k
  \in \ZZ^d} e^{i k \cdot x} \Bigg[ e^{-i \|k\|^2
\Delta t + i k \cdot \paren{\varphi(t) - \varphi(t-\Delta t)}}
  \uhat_k(t-\Delta t)  \\
   - i \Delta t \sum_{j=1}^{n-1} \mu_j e^{-i \|k\|^2 j \Delta t + i k
  \cdot \paren{\varphi(t) - \varphi(t-j\Delta t)}}  \Vuhat_k(t - j \Delta
  t)  \Bigg]
\end{multline}
in place of \eqref{eq:utrapk}.

The semi-discrete and fully discrete marching schemes follow from these
formulas in the same manner as before, with the caveat that
\eqref{eq:uadamsk} is only valid for $t \geq (n-1) \Delta t$. As with all multistep
methods, we therefore need an alternative initialization method to obtain
the first $n-2$ time steps with sufficient accuracy that subsequent calculations 
retain the overall $n$th-order accuracy of the scheme.
There are many possible approaches, but a simple method is
iterated Richardson extrapolation \cite[Sec. 3.4.6]{dahlquist08} based on the second-order
trapezoidal rule. As an example, we illustrate the procedure for a
single time step at eighth-order accuracy.

Given that we have completed the simulation up to time $t - \Delta t$,
let $u_0^{(0)}$, $u_0^{(1)}$, $u_0^{(2)}$, and $u_0^{(3)}$ be the
approximations of $u(x,t)$ obtained by the second-order trapezoidal rule
starting from $u(x,t-\Delta t)$
with one step of size $\Delta t$, two steps of size $\Delta t / 2$, four
steps of size $\Delta t / 4$, and eight steps of size $\Delta t
/ 8$, respectively. These may be combined to obtain a
collection of fourth-order accurate approximations $u_1^{(0)}$,
$u_1^{(1)}$, and $u_1^{(2)}$ of $u(x,t)$ by the following formulas:
\[u_1^{(0)} = \frac{2^2 u_0^{(1)} - u_0^{(0)}}{2^2-1}, \qquad u_1^{(1)} = \frac{2^2 u_0^{(2)} -
u_0^{(1)}}{2^2-1}, \qquad u_1^{(2)} = \frac{2^2 u_0^{(3)} - u_0^{(2)}}{2^2-1}.\]
These may be subsequently combined to obtain sixth-order accurate
approximations $u_2^{(0)}$ and $u_2^{(1)}$:
\[u_2^{(0)} = \frac{2^4 u_1^{(1)} - u_1^{(0)}}{2^4-1}, \qquad u_2^{(1)} = \frac{2^4 u_1^{(2)} -
u_1^{(1)}}{2^4-1}.\]
An eighth-order accurate approximation $u_3^{(0)}$ of $u(x,t)$ is then
given by
\[u_3^{(0)} = \frac{2^6 u_2^{(1)} - u_2^{(0)}}{2^6-1}.\]
Note that we were able to skip odd orders in the extrapolation procedure
because the error expansion of the trapezoidal rule contains only even
powers in $\Delta t$. Seven steps of the above procedure must be carried
out to initialize the eighth-order implicit Adams
method. The iterated Richardson extrapolation approach may be
generalized to build a single-step
method of any even order $n$, which can then be used to initialize the $n$th order implicit Adams
method.

The dominant cost of the multistep method is
is that of computing one FFT and one inverse FFT
per time step, just as for the trapezoidal rule-based method, regardless
of the order of accuracy $n$.
One could also derive multistage, Runge-Kutta-style schemes by discretizing the local
update integral using a quadrature rule involving intermediate time points. The
resulting methods would be more expensive, but might have different
stability properties. We have not yet analyzed and compared the various possible schemes.

\section{The free space case} \label{sec:free}

The derivation of the semi-discrete marching scheme for the free space case
is virtually identical to that of the periodic case once the
periodic Fourier series is replaced by the continuous inverse Fourier transform.
The difficulty appears only once we consider the 
fully discrete scheme. Naively discretizing $u(x,t)$ and $\wh{u}(\xi,t)$ on grids
in physical and Fourier space, respectively, results in a highly
inefficient method. A marching scheme preserving the favorable properties of
the periodic algorithm will be obtained by deforming the contour of
integration defining the inverse Fourier transform.

We will require the Fourier transform of the free-space Green's function \eqref{eq:gfun},
which we will refer to as the {\em spectral Green's function}:
\[\wh{G}(\xi,t,s) = e^{-i \| \xi \|^2 (t-s) + i \xi \cdot (\varphi(t) - \varphi(s))}.\]
This function already played an important role in the periodic case. 

Suppose now that $u$ satisfies \eqref{eq:schrodfree} with
$D = \RR^d$ in the Schwartz space,
with the $C^\infty$-smooth functions $u_0, V$ supported in the box 
$[-1,1]^d$. $u(x,t)$ may be represented via the
Fourier transform,
\begin{equation}\label{eq:uiftrep}
  u(x,t) = \frac{1}{(2 \pi)^d} \int_{\RR^d} e^{i \xi \cdot x} \uhat(\xi,t) \, d \xi,
\end{equation}
with the definition
\begin{equation}\label{eq:ftfree}
  \wh{f}(\xi) := \int_{\RR^d} e^{-i \xi \cdot x} f(x) \, dx.
\end{equation}
Analogous to the periodic case, $\uhat(\xi,t)$ satisfies an ODE in time,
\begin{align*}
  \begin{aligned}
    i \partial_t \uhat(\xi,t) &= \paren{\|\xi\|^2 - \xi \cdot A(t)} \uhat(\xi,t) + \wh{(Vu)}(\xi,t), 
    \quad & t \in (0,T], \\
    \uhat(\xi,0) &= \wh{u}_0(\xi), &
  \end{aligned}
\end{align*}
which we again write in integral form as
\begin{equation} \label{eq:uhatAinteq}
  \uhat(\xi,t) = e^{-i \| \xi \|^2 t + i \xi \cdot \varphi(t)} \wh{u}_0(\xi) - i \int_0^t e^{-i 
\| \xi \|^2 (t-s) + i \xi \cdot (\varphi(t) - \varphi(s))} \wh{(Vu)}(\xi,s) \, ds.
\end{equation}
This is the Fourier transform of the
Duhamel formula \eqref{eq:lippmannschwinger}. 

\subsection{The free space marching scheme using the classical Fourier
transform}
\label{sec:realmarch}

As for the periodic case, we can rewrite \eqref{eq:uhatAinteq} as a
recurrence in time.

\begin{lemma}[Continuous spectral evolution]
The evolution formula \eqref{eq:uhatAinteq}
can be written without explicit history dependence in the form
\begin{equation}\label{eq:uhatxirecur}
  \uhat(\xi,t) = e^{-i \| \xi \|^2 \Delta t + i \xi \cdot \paren{\varphi(t) -
  \varphi(t-\Delta t)}} \uhat(\xi,t-\Delta t) - i \int_{t - \Delta t}^t
  e^{-i \| \xi \|^2 (t-s) +
  i \xi \cdot (\varphi(t) - \varphi(s))} \Vuhat(\xi,s) \, ds.
\end{equation}
Using the trapezoidal rule for the update integral, we obtain the 
following recurrence:
\begin{equation} \label{eq:uhatxi}
  \uhat(\xi,t) \approx e^{-i \| \xi \|^2
\Delta t + i \xi \cdot \paren{\varphi(t) - \varphi(t-\Delta t)}} \paren{
  \uhat(\xi,t-\Delta t)  - i  \frac{\Delta t}{2} \Vuhat(\xi,t - \Delta
t)} - i  \frac{\Delta t}{2} \Vuhat(\xi,t).
\end{equation}
\end{lemma}
\begin{proof}
  The proof is identical to that in Lemma \ref{lem:discspecevol} for the periodic
  case. 
\end{proof}

Applying the inverse Fourier transform to
\eqref{eq:uhatxi}, we obtain the analogue of \eqref{eq:utrapk}, namely
\begin{equation} \label{eq:utrapxi}
  u(x,t) \approx \frac{1}{1 + i \frac{\Delta t}{2} V(x,t)}
  \int_{\RR^d} e^{i \xi \cdot x} e^{-i \| \xi \|^2 
  \Delta t + i \xi \cdot \paren{\varphi(t) - \varphi(t-\Delta t)}} \paren{
\uhat(\xi,t-\Delta t)  - i  \frac{\Delta t}{2} \Vuhat(\xi,t - \Delta
t)}.
\end{equation}
This suggests a semi-discrete marching scheme analogous to that for the
periodic case. Note that while the support of $u(x,t)$ in general extends beyond
$[-1,1]^d$, it need never be
evaluated outside the support of $V$. Indeed, given $\uhat(\xi,t-\Delta
t)$ and $\Vuhat(\xi,t - \Delta t)$, $\Vuhat(\xi,t)$ may be computed
by evaluating \eqref{eq:utrapxi} inside the support of $V$, multiplying
pointwise by $V(x,t)$, and applying the Fourier transform \eqref{eq:ftfree} to $(Vu)(x,t)$.
Then $\uhat(\xi,t)$ may be computed using \eqref{eq:uhatxi} instead of the
Fourier transform formula, which would require sampling $u(x,t)$ far outside
$[-1,1]^d$. This procedure describes a time step of a semi-discrete
$\OO{\Delta t^2}$ scheme. In particular, no artificial boundary conditions are needed.

Let us now consider the discretization of \eqref{eq:uhatxi} and \eqref{eq:utrapxi} 
in the physical and Fourier variables.
In the periodic
case, discretization in the Fourier domain amounted to truncating 
the rapidly converging Fourier series representations for $u(x,t)$ and $(Vu)(x,t)$. 
Here, again letting $d = 1$ for simplicity, we must discretize the inverse Fourier
transforms 
\begin{equation} \label{eq:iftureal}
  u(x,t) = \frac{1}{2\pi} \int_{-\infty}^\infty e^{i \xi x}
\uhat(\xi,t) \, d \xi
\end{equation}
and
\begin{equation} \label{eq:iftVureal}
(Vu)(x,t) = \frac{1}{2\pi} \int_{-\infty}^\infty e^{i \xi x}
\Vuhat(\xi,t) \, d \xi.
\end{equation}
Since $(Vu)(x,t)$ is smooth and compactly supported for each $t$, its Fourier transform is
rapidly decaying and non-oscillatory, and the discretization of
\eqref{eq:iftVureal} is straightforward. To understand the cost of discretizing
\eqref{eq:iftureal}, we can analyze the behavior of $\uhat(\xi,t)$ using
\eqref{eq:uhatAinteq}.
We assume for the moment that $A = 0$, in which case
\eqref{eq:uhatAinteq} takes the simpler form
\begin{equation} \label{eq:uhatinteq}
  \uhat(\xi,t) = e^{-i \xi^2 t} \uzhat(\xi) - i \int_0^t e^{-i
  \xi^2 (t-s)} \Vuhat(\xi,s) \, ds.
\end{equation}
$\uzhat$ is rapidly decaying like $\Vuhat$ because $u_0$ is smooth, so 
$\uhat$ is rapidly decaying as well, and \eqref{eq:iftureal} may be truncated at
some value $|\xi| = K_0$, i.e.\
\begin{equation}\label{eq:ifturealtrunc}
  u(x,t) \approx \frac{1}{2 \pi} \int_{-K_0}^{K_0} e^{i \xi x} \uhat(\xi,t)
\, d \xi,
\end{equation}
with superalgebraic convergence in the parameter $K_0$.
This implies, as discussed in Remark \ref{rem:ptrerr} for the periodic
case, that $u(x,t)$ may be resolved on $[-1,1]$ using a grid with $M =
\OO{K_0}$ points. However, unlike $\uzhat$ and $\Vuhat$, which are
non-oscillatory due to the compact support of $u_0$ and $Vu$, {\em
$\uhat(\xi,t)$ is highly oscillatory, requiring $\OO{K_0^2 T}$ grid
points to be resolved for all $t \in [0,T]$.}
Indeed, the behavior of $\uhat(\xi,t)$ is inherited from that of the
spectral Green's function $\wh{G}(\xi,t) = e^{-i \xi^2 t}$ according to
\eqref{eq:uhatinteq}, and $\wh{G}(\xi,t)$ has
$\OO{K_0^2 t}$ oscillations in $[-K_0,K_0]$ (see the top panels of
Figure~\ref{fig:specgfungamma} for an illustration).
Thus, we cannot
accurately discretize \eqref{eq:ifturealtrunc}, and therefore
\eqref{eq:iftureal}, for all $t \in [0,T]$
by a uniform quadrature grid of fewer than $\OO{K_0^2 T}$ nodes.
The best one can hope for 
using the classical Fourier transform is a marching scheme that requires
$\OO{M^2 T}$ work per time step for $M$ grid points in space---far greater 
than the $\OO{M \log M}$ cost of the periodic scheme.

The difference between the free space and periodic cases, of course,
is that the numerical support of the free space solution grows with
time, which causes oscillation in the frequency domain.
The challenge is to find a spectral representation
that is less oscillatory and can therefore be resolved with fewer degrees of freedom.

\begin{remark}
A closely related problem is that of developing a Fourier
  transform-based method for the heat equation in free space.  In
  \cite{greengard00}, it was shown that by exponentially clustering
  nodes toward $\xi = 0$, one can resolve the spectral Green's function
  by $\OO{M + \log T}$ nodes and obtain a quasi-optimal scheme.
  In that setting, the Fourier transform of
  the solution becomes sharply peaked near $\xi = 0$, but is otherwise
  smooth. Here, there is also a peak near $\xi=0$, but the
  oscillatory behavior at large $\xi$ renders this approach
  insufficient.
\end{remark}

\subsection{The complex-frequency representation} \label{sec:cfrep}

In order to cope with the oscillatory behavior of the spectral Green's function,
we will extend the variable $\xi$ to the complex space 
$\CC^d$ and define a suitable analytic extension of 
$\uhat(\xi,t)$ which will permit a contour deformation of the Fourier representation
\eqref{eq:uiftrep}. The contour will be chosen so
that the oscillations, which increase in frequency over time, are damped
to a specified precision,
yielding the same accuracy with a significantly coarser quadrature rule.

We first define the contour $\Gamma$, shown in Figure \ref{fig:zigzag}, by the parameterization $\gamma: \RR
\to \CC$,
\begin{equation}\label{eq:gammadef}
  \gamma(\tau) = 
  \begin{cases}
    \gamma_1(\tau) = \tau + iH, & -\infty < \tau < -H \\
    \gamma_2(\tau) = \tau - i\tau, & -H \le \tau \le H \\
    \gamma_3(\tau) = \tau - iH, & H < \tau < \infty.
  \end{cases}
\end{equation}
Here $H > 0$ is a parameter, the selection of which will be discussed later. We write $\Gamma = \Gamma_1 \cup \Gamma_2
\cup \Gamma_3$, where $\Gamma_i$ is the portion of the curve given by
the parameterization $\gamma_i$. 

\begin{figure}[t]
  \centering
    \includegraphics[width=0.8\linewidth]{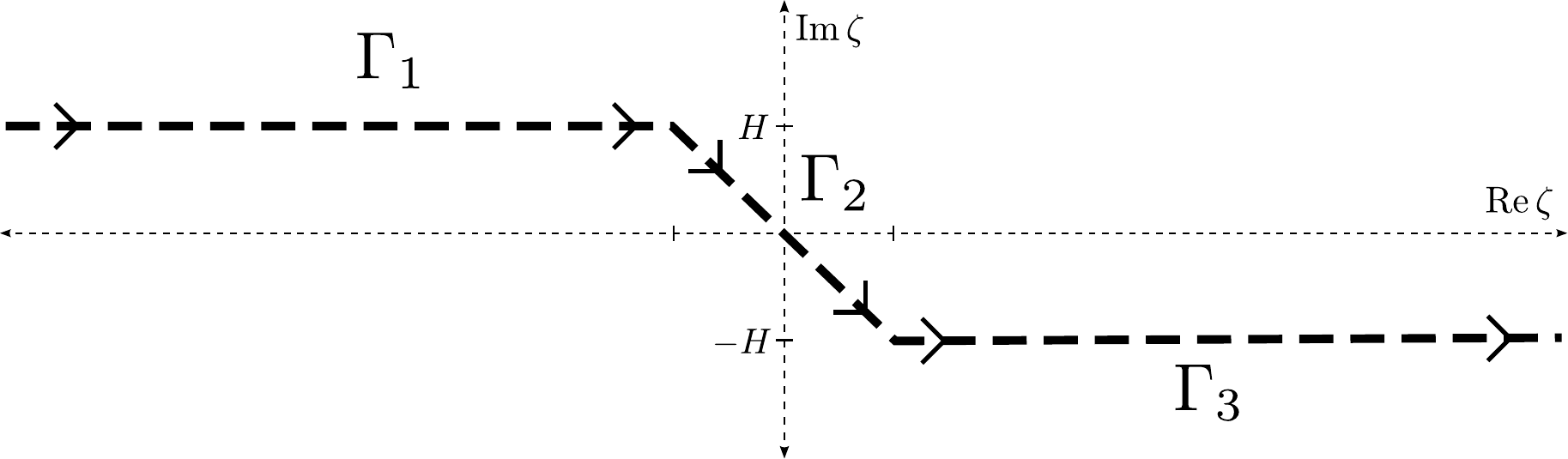}
    \caption{The contour $\Gamma = \Gamma_1 \cup \Gamma_2 \cup \Gamma_3$ is comprised of two horizontal
      segments with imaginary parts $H$ and $-H$, respectively, and a
      diagonal segment connecting them. It is given explicitly by the
      parameterization \eqref{eq:gammadef}.}
    \label{fig:zigzag}
\end{figure}

Since $u_0(x)$ and $(Vu)(x,t)$ for fixed $t \in [0,T]$ are smooth and
compactly supported in $x$, their
Fourier transforms define entire functions $\uzhat(\zeta)$ and
$\Vuhat(\zeta,t)$, respectively, with $\zeta \in \CC^d$ \cite[Thm. 7.2.2]{strichartz94}.
The following lemma asserts
that $\uhat(\zeta,t)$ is also an entire function on
$\CC^d$, and introduces the {\em complex Fourier representations}
of $u(x,t)$ and $(Vu)(x,t)$.
\begin{lemma}\label{lem:uentire}
  Let $u$ satisfy \eqref{eq:schrodfree} and the
  assumptions made above on $u_0$, $u$,
  $V$, and $A$ for the free space problem. Then for each $t \in [0,T]$, $\uhat(\xi,t)$
  may be extended as a function of $\xi$ to an entire function on
  $\CC^d$ by the formula
  \begin{equation} \label{eq:uhatzinteq}
    \uhat(\zeta,t) = e^{-i \zeta \cdot \zeta t + i \zeta \cdot \varphi(t)}
      \wh{u}_0(\zeta) - i \int_0^t e^{-i \zeta \cdot \zeta
    (t-s) + i \zeta \cdot (\varphi(t) - \varphi(s))} \wh{(Vu)}(\zeta,s)
    \, ds.
  \end{equation}
  For $\Gamma$ defined as in \eqref{eq:gammadef}, $u(x,t)$ and
  $(Vu)(x,t)$ may be recovered from their Fourier transforms
  on $\CC^d$ by the deformed inverse Fourier transforms
  \begin{equation}\label{eq:uciftrep}
    u(x,t) = \frac{1}{(2 \pi)^d}
  \int_{\Gamma^d} e^{i \zeta \cdot x} \uhat(\zeta,t) \, d \zeta
  \end{equation}
  and
  \begin{equation}\label{eq:Vuciftrep}
    (Vu)(x,t) = \frac{1}{(2 \pi)^d}
  \int_{\Gamma^d} e^{i \zeta \cdot x} \Vuhat(\zeta,t) \, d \zeta,
  \end{equation}
  respectively. Here, $\Gamma^d$ is the Cartesian
  product of $d$ copies of $\Gamma$, which is a $d$-dimensional surface in $\CC^d$.
\end{lemma}

A detailed proof for $d = 1$ is given in Appendix
\ref{sec:pflemuentire}, and for $d > 1$ the argument may be applied to
each dimension in turn. The analyticity of $\uhat(\zeta,t)$ defined by
\eqref{eq:uhatzinteq} follows from Morera's theorem, and the contour
deformations may be justified by Cauchy's theorem and an argument
involving the Riemann--Lebesgue lemma.

We give a brief explanation of the choice of $\Gamma$ here, with detailed
justification postponed until
Section \ref{sec:analysis}.  As before, we take $d = 1$ and $A = 0$, which will be sufficient to
illustrate the main ideas. We will show that
the complex Fourier representation \eqref{eq:uciftrep} can be
discretized with far fewer quadrature points than the real representation
\eqref{eq:uiftrep}. We assume here that $x \in [-1,1]$ in the
representation \eqref{eq:uciftrep}; indeed, as in Section
\ref{sec:realmarch}, our marching scheme will only require us to
evaluate $u(x,t)$ in this interval (see also Remark \ref{rem:fareval}). As before, we assume that $u(x,t)$ can be
resolved on $[-1,1]$ by a grid of $M = \OO{K_0}$ points, 
and show that the complex Fourier representation may be
discretized by a comparable number of points,
rather than the $\OO{M^2 T}$ points required for the real Fourier
representation. 
This leads directly to an efficient complex-frequency marching scheme.

Note first that 
$\uhat$ decays rapidly along $\Gamma$, as it does on the real line,
so that we can truncate the
complex Fourier representation \eqref{eq:uciftrep} at $\abs{\Re(\zeta)}
= K$ for some $K > 0$; that is, by analogy with
\eqref{eq:ifturealtrunc}, we have
\[u(x,t) \approx \frac{1}{2 \pi}
\int_{\Gamma_K} e^{i \zeta \cdot x} \uhat(\zeta,t) \, d \zeta,\]
where $\Gamma_K$ is the truncation of \eqref{eq:gammadef} to $\tau \in
[-K,K]$.
In Section \ref{sec:gammatrunc}, we show that we can take $K = K_0 + L$,
for a constant $L$,
so that $M = \OO{K}$.  The extension $L$ depends only on the 
desired precision and the parameter $H$, and not on $K_0$.

Since we have assumed $x \in [-1,1]$, the cost of discretizing this integral depends now on the 
behavior of $\uhat$ on $\Gamma_K$, which is described by
\eqref{eq:uhatzinteq}. For $A = 0$, \eqref{eq:uhatzinteq} becomes
\[\uhat(\zeta,t) = e^{-i \zeta^2 t} \wh{u}_0(\zeta) - i \int_0^t e^{-i \zeta^2
(t-s)} \wh{(Vu)}(\zeta,s) \, ds.\]
As before, $\uzhat$ and $\Vuhat$ are well-behaved, and in Figure
\ref{fig:specgfun} we give plots of the spectral Green's
function $\wh{G}(\zeta,t) = e^{-i \zeta^2 t}$ along $\Gamma$ and in the
complex plane, for several values of $t$.
While
$\wh{G}(\zeta,t)$ still oscillates along the horizontal
contours $\Gamma_1$ and $\Gamma_3$ at a
frequency which increases with $t$, it now decays exponentially at a rate
which also increases with $t$. As a result, $\wh{G}(\zeta,t)$, and
therefore $\uhat(\zeta,t)$, may be
resolved on $\Gamma_1 \cap \Gamma_K$ and $\Gamma_3 \cap \Gamma_K$ 
by a grid with $\OO{1}$ spacing with respect to $K$ for all $t
\in [0,T]$, or $\OO{K} = \OO{M}$ points in total, {\em for any fixed level of precision}.

On $\Gamma_2$, $\wh{G}(\zeta,t)$ takes the form of a Gaussian
of width $\frac{1}{2 \sqrt{t}}$, which motivates our choice of the angle
$-\pi/4$ for this segment.
To accurately
integrate all such Gaussians for $t \in [0,T]$ using a single quadrature rule, we
can cluster nodes exponentially towards the origin \cite{greengard00,qpsc}. This
requires a total of $\OO{\log T}$ quadrature nodes.

In short, we can discretize the complex Fourier representation
\eqref{eq:uciftrep} for all $t \in [0,T]$ using a quadrature rule with
$\OO{M + \log T}$ nodes on $\Gamma_K$.  In Section \ref{sec:gammares},
we will see that the same strategy may be used when $A \neq 0$, but more
grid points are required on $\Gamma_1$ and $\Gamma_3$; in this case,
given the proper choice of $H$, we will require $\OO{\phimax M + \log
T}$ nodes, where $\phimax$ is the quiver radius of $A$, defined by
\begin{equation} \label{eq:phimax}
  \phimax := \max_{t \in [0,T]} \abs{\varphi(t)}.
\end{equation}
We therefore write the estimate for the general case as $\OO{(1 + \phimax)
M + \log T}$, which reduces to the correct estimate for $A = 0$.

\begin{figure}[t]
  \centering

  \begin{subfigure}[t]{\textwidth}
    \includegraphics[width=\linewidth]{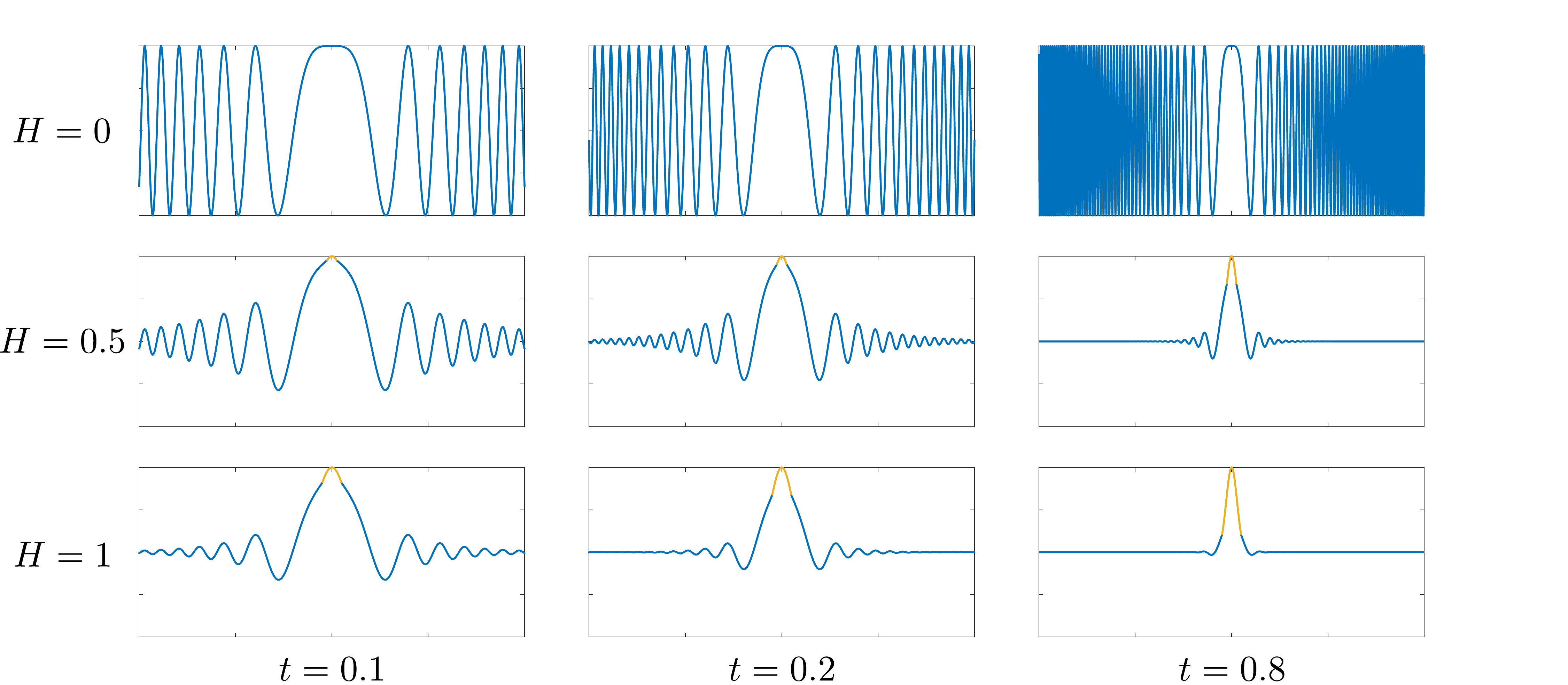}

    \caption{$\Re \wh{G}(\zeta,t)$ along $\Gamma$, $\zeta = \gamma(\tau)$. Yellow indicates the
    part of the graph of $\Re \wh{G}(\zeta,t)$ on $\Gamma_2$, and blue the part on 
    $\Gamma_1$ or $\Gamma_3$.}
    \label{fig:specgfungamma}
  \end{subfigure}
  \par\bigskip
  \begin{subfigure}[t]{\textwidth}
    \includegraphics[width=\linewidth]{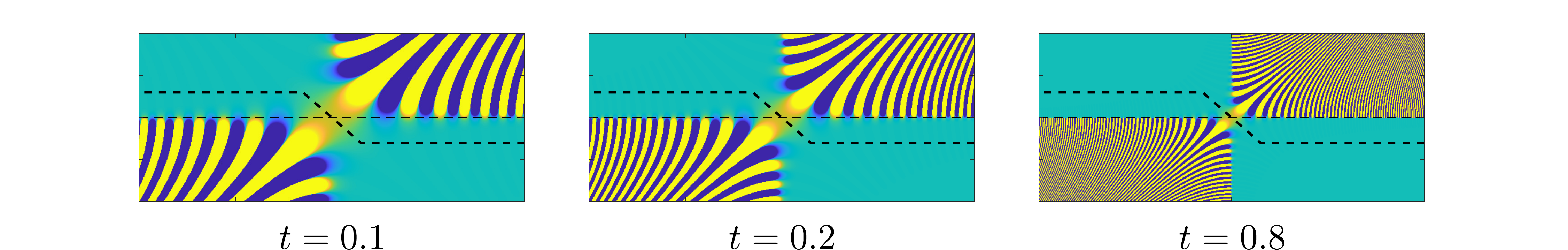}

    \caption{$\Re \wh{G}(\zeta,t)$ on $\CC$. The real line is indicated
    by the thin dashed line, and the contour $\Gamma$, for some choice
    of $H$, is
    indicated by the thick dashed line. Yellow corresponds to large
    positive values, blue to large negative values, and cyan to values
    near zero.}
    \label{fig:specgfuncontour}
  \end{subfigure}

  \caption{The real part of the spectral Green's function
  $\wh{G}(\zeta,t) = e^{-i \zeta^2 t}$, for several values of $t$,
  plotted (a) along a portion of the
  contour $\Gamma$, with several choices of $H$,
  and (b) and in the complex plane. Along the real line, which corresponds to
    $H = 0$, the spectral Green's function oscillates more and more
    rapidly with increasing $t$, and does not decay. For $H>0$, the
    oscillations remain, but they are accompanied by damping which also
    increases with $t$. As a result, the grid
    spacing required to resolve all oscillations with magnitude above a
    given threshold value remains constant with $t$. The damping rate
    increases with $H$. $\wh{G}(\zeta,t)$ also becomes narrower near the origin for larger $t$,
    requiring a logarithmic clustering of quadrature nodes for large $T$.}
  \label{fig:specgfun}
\end{figure}

\begin{remark} \label{rem:fareval}

Although we have assumed above that $x \in
  [-1,1]$, the complex-frequency
  representation \eqref{eq:uciftrep} may be evaluated at any
  $x \in \RR$ once $\uhat(\zeta,t)$ has been resolved on $\Gamma_K$.
  This may be done by interpolating $\uhat(\zeta,t)$ to a quadrature grid
  sufficiently fine to resolve $e^{i \zeta x}$ on $\Gamma_K$, or
  similarly by expanding $\uhat(\zeta,t)$ in a basis and precomputing
  the corresponding moments.

\end{remark}

\begin{remark} \label{rem:Hchoice}
It is evident from Figure \ref{fig:specgfun}
that the choice of $H$ is critical. Too small a value leads to insufficient
damping of high frequency oscillations.
  On the other hand, $e^{i \zeta x}$ and $\uhat(\zeta,t)$ grow
  exponentially in the imaginary direction, so too large a value places the path
of integration of the deformed inverse Fourier transform in a region
of large amplitude oscillations, leading to a loss of
accuracy in finite precision arithmetic from catastrophic cancellation. 
We will show in Section
\ref{sec:gammares} that the correct balance is achieved by taking $H = \frac{\log\paren{\varepsilon/\paren{\paren{1+\norm{V}_{2,\infty}}
  \epsm}}}{2d(1 + \phimax)}$, where $\varepsilon$
is the desired precision, $\epsm$ is the machine epsilon, and $\norm{V}_{2,\infty} = \max_{t \in
[0,T]} \norm{V(\cdot,t)}_2$.
\end{remark}

\subsection{The complex-frequency marching scheme}

We can now derive a complex-frequency marching scheme following exactly the
same procedure as for the real-frequency marching scheme in Section
\ref{sec:realmarch}. The formulas
\eqref{eq:uiftrep}-\eqref{eq:utrapxi} remain true, with integration over
$\RR^d$ replaced by integration over $\Gamma^d$, the real variable $\xi
\in \RR^d$ replaced by a complex variable $\zeta \in \Gamma^d$, and the
norm $\| \xi \|^2$ replaced by the sum of squares $\zeta \cdot \zeta = \zeta_1^2 + \cdots +
\zeta_d^2$. For completeness, we write out the fully discrete marching
scheme for the one-dimensional case; the higher-dimensional case is
analogous.

We introduce a set of
equispaced grid points on $[-1,1]$, $x_j = -1 + 2(j-1)/M$ with
$j=1,\ldots,M$, and assume for the moment that
there is a set of spectrally accurate quadrature nodes $\zeta_1,\ldots,\zeta_N \in
\Gamma$ and weights $w_1,\ldots,w_N$ so that
\begin{equation}\label{eq:uciftdisc}
  u(x,t) = \frac{1}{2 \pi} \int_\Gamma e^{i \zeta x} \uhat(\zeta,t) \, dt \approx
  \frac{1}{2 \pi} \sum_{k=1}^N e^{i \zeta_k x} \uhat(\zeta_k,t) w_k
\end{equation}
and
\begin{equation}\label{eq:Vuciftdisc}
  (Vu)(x,t) = \frac{1}{2 \pi} \int_\Gamma e^{i \zeta x} \Vuhat(\zeta,t) \, dt \approx
  \frac{1}{2 \pi} \sum_{k=1}^N e^{i \zeta_k x} \Vuhat(\zeta_k,t) w_k
\end{equation}
hold to high accuracy
for all $t \in [0,T]$. The specific form of this
rule will be discussed in Section \ref{sec:gammaquad}. As noted in the previous
section, it will have $N = \OO{(1+\phimax) M + \log T}$ nodes.
A complex-frequency DFT is given by the
equispaced trapezoidal rule, which is spectrally accurate for smooth,
compactly-supported functions:
\begin{equation} \label{eq:fzdft}
  \wh{f}(\zeta_k) = \int_{-1}^1 e^{-i \zeta_k x} f(x) \, dx
\approx \frac{2}{M} \sum_{j=1}^M e^{-i \zeta_k x_j} f(x_j).
\end{equation}
The fully-discretized, complex-frequency analogues of
\eqref{eq:uhatxi} and \eqref{eq:utrapxi} are, respectively,
\begin{equation} \label{eq:uhatzfd}
  \uhat(\zeta_k,t) \;\approx\;
  e^{-i
    \zeta_k^2 \Delta t + i \zeta_k \paren{\varphi(t) - \varphi(t-\Delta t)}} \paren{
  \uhat(\zeta_k,t-\Delta t)  - i  \frac{\Delta t}{2} \Vuhat(\zeta_k,t - \Delta
t)} - i \frac{\Delta t}{2} \Vuhat(\zeta_k,t)
\end{equation}
for each $k=1,\ldots,N$, and
\begin{equation} \label{eq:utrapzfd}
  u(x_j,t) \;\approx\; \frac{1}{1 + i \frac{\Delta t}{2} V(x_j,t)}
  \sum_{k=1}^N e^{i \zeta_k x_j} e^{-i \zeta_k^2
  \Delta t + i \zeta_k \paren{\varphi(t) - \varphi(t-\Delta t)}} \paren{
    \uhat(\zeta_k,t-\Delta t)  - i  \frac{\Delta t}{2} \Vuhat(\zeta_k,t - \Delta t)}
\end{equation}
for each $j=1,\ldots,M$.
They lead to the following
fully discrete second-order marching scheme:
\begin{enumerate}
  \item Given $\uhat(\zeta_k,t-\Delta t)$ and $\Vuhat(\zeta_k,t -
\Delta t)$ for $k=1,\ldots,N$, compute $u(x_j,t)$ for $j =
1,\ldots,M$ using \eqref{eq:utrapzfd}.
  \item Compute $\Vuhat(\zeta_k,t)$ by multiplication with $V(x_j,t)$ and the
complex-frequency DFT \eqref{eq:fzdft}.
  \item Compute $\uhat(\zeta_k,t)$ for $k = 1,\ldots,N$ using
    \eqref{eq:uhatzfd}. Update $t \leftarrow t + \Delta t$ and repeat
    from the first step.
\end{enumerate}
Since $u_0$ is supported on $[-1,1]$, the scheme is initialized by
directly computing $\uhat(\zeta_k,0)$ and $\Vuhat(\zeta_k,0)$
using the complex-frequency DFT \eqref{eq:fzdft}.
We note as before that the Fourier coefficients are updated without a direct
Fourier transform of $u$, which would require evaluating $u(x,t)$
outside of $[-1,1]$. 
The cost of this marching scheme is
dominated by that of computing one forward and one inverse
complex-frequency DFT per time step.

\begin{remark} \label{rem:highorderfs}

  Alternative time discretizations may be obtained as in
  Section \ref{sec:highorder}, by replacing the trapezoidal rule for the
  local update integral by some other approximation. In
  particular, we can obtain an $n$th-order implicit Adams scheme by
  copying over the formulas for the periodic case almost exactly,
  exchanging the periodic Fourier
  transforms for their free space, complex-frequency
  analogues. Thus, for the fully-discretized scheme,
  \eqref{eq:uhatzfd} and \eqref{eq:utrapzfd} are replaced by
\begin{multline*}
  \uhat(\zeta_k,t) \approx e^{-i \zeta_k^2
\Delta t + i \zeta_k \paren{\varphi(t) - \varphi(t-\Delta t)}}
  \uhat(\zeta_k,t-\Delta t)
  - i \Delta t \sum_{l=0}^{n-1} \mu_l e^{-i \zeta_k^2 l \Delta t + i \zeta_k
  \paren{\varphi(t) - \varphi(t-l\Delta t)}}  \Vuhat(\zeta_k,t - l \Delta
  t)
\end{multline*}
and
  \begin{multline*}
u(x_j,t) \approx \frac{1}{1 + i \mu_0 \Delta t V(x_j,t)}
  \sum_{k=1}^N e^{i \zeta_k x_j} \Bigg[ e^{-i \zeta_k^2
\Delta t + i \zeta_k \paren{\varphi(t) - \varphi(t-\Delta t)}}
  \uhat(\zeta_k,t-\Delta t)  \\
   - i \Delta t \sum_{l=1}^{n-1} \mu_l e^{-i \zeta_k^2 l \Delta t + i
   \zeta_k \paren{\varphi(t) - \varphi(t-l\Delta t)}}  \Vuhat(\zeta_k,t - l \Delta
  t)  \Bigg],
\end{multline*}
respectively. As discussed in Section \ref{sec:highorder},
the multistep method requires initialization, which can again be 
accomplished using iterated Richardson
  extrapolation. Note that here, we must perform
  Richardson extrapolation both on $u(x,t)$ and on $\uhat(\xi,t)$. 
\end{remark}

It remains to describe the quadrature used in \eqref{eq:uciftdisc} and
\eqref{eq:Vuciftdisc}, and to show that the non-standard DFTs arising in
the fully discrete marching scheme may be implemented by a fast,
FFT-based algorithm. These issues are discussed in Sections
\ref{sec:gammaquad} and \ref{sec:fftgamma}, respectively. The result
will be a free space marching scheme which does not require the use of
artificial boundary conditions, and shares the benefits of the periodic
scheme: it is spectrally accurate in space, admits inexpensive
high-order implicit time discretization, and has a near-optimal computational
cost and memory requirements.

\subsection{Quadrature rule on $\Gamma$} \label{sec:gammaquad}

Guided by the discussion in Section \ref{sec:cfrep}, we now describe a
spectrally accurate quadrature rule to use in 
\eqref{eq:uciftdisc} and \eqref{eq:Vuciftdisc}. We assume the integrals
have been truncated to $\Gamma_K$, with $K$ chosen based on the
decay of $\uhat$ and $\Vuhat$.  Here and throughout the rest of the
article, we will abuse notation and use the notation $\Gamma_1,
\Gamma_3$ for both the infinite rays and their truncated analogues; the
usage should be clear from the context.

We first require a quadrature for
a smooth function on the segments $\Gamma_1$ and $\Gamma_3$; that is, on $\gamma(\tau)$ with $\tau \in
[-K,-H]$ and $\tau \in [H,K]$. A simple and accurate 
choice would be Gauss-Legendre quadrature.
As will become clear in Section
\ref{sec:fftgamma}, this would lead to a fast algorithm, but one that requires
nonuniform FFTs \cite{nufft2,greengard04,barnett19}, which are slower
than ordinary FFTs. 
Instead, we will use 
Alpert's high-order hybrid Gauss-trapezoidal rule.
This rule modifies the equispaced trapezoidal rule to achieve convergence
of order $2p$ by adding $p$ auxilliary nodes, with carefully chosen
weights, near each endpoint.
On $\Gamma_2$, we will use a different rule that clusters points exponentially
near the origin. The resulting composite rule is accurate and robust,
and is compatible with a fast algorithm based on the ordinary FFT.

For any $p \in \ZZ^+$ and $n \in \ZZ^+$, Alpert's hybrid Gauss-trapezoidal 
rule for a smooth integrand $f$ on $[a,b]$ is given by
\[\int_a^b f(x) \, dx = h \sum_{k=1}^p w_k^{\text{alp}} f\paren{a + x_k^{\text{alp}}
  h} + h \sum_{k=0}^{n-1}
  f\paren{a + \kappa h + kh} + h \sum_{k=1}^p w_k^{\text{alp}} f\paren{b -
x_k^{\text{alp}} h},\]
  where $\kappa$ is the number of omitted regular nodes
  (a constant independent of $n$ determined by $p$),
  $h = (b-a)/(n+2\kappa-1)$ is the
trapezoidal rule grid spacing chosen so that $a + \kappa h + (n-1)h = b
- \kappa h$, and $x_1^{\text{alp}},\ldots,x_p^{\text{alp}}$,
$w_1^{\text{alp}},\ldots,w_p^{\text{alp}}$ are the nodes and weights providing endpoint corrections to the
trapezoidal rule. Values for $\kappa$, $x_k^{\text{alp}}$, and
$w_k^{\text{alp}}$ may be found in standard tables for several
choices of $p$ \cite{alpert99}. In our case, since the integrand already
decays at one of the endpoints, we only require corrections at the other.
For a fixed $p$ and some number $\NE$ of equispaced nodes, we obtain the following
quadrature of order $2p$ for a function $f$ on $\Gamma_3$:
\[
\int_{\Gamma_3} f(\zeta) \, d \zeta = \int_K^H f(\tau - iH) \, d \tau
\;\approx\;
\sum_{k=1}^p f\paren{\zeta_k^{(\calA_3)}} w_k^{(\calA_3)} + \sum_{k=1}^{\NE}
f\paren{\zeta_k^{(\calE_3)}} w_k^{(\calE_3)},
\]
where
\begin{align*}
  \zeta_k^{(\calA_3)} &= H + x_k^{\text{alp}} h - iH, & w_k^{(\calA_3)} &= h w_k^{\text{alp}}, \\
  \zeta_k^{(\calE_3)} &= H + \kappa h + kh - iH, & w_k^{(\calE_3)} &= h,
\end{align*}
and $h$ is defined as before with $a = H$, $b = K$.
Note that for simplicity we have assumed $\abs{f(\tau-iH)}$ has decayed below our required
accuracy by $\tau = K - \kappa h$ rather than $\tau = K$, and simply deleted the right
endpoint correction. 
The quadrature on $\Gamma_1$ may be defined
by symmetry:
\[\int_{\Gamma_1} f(\zeta) \, d \zeta = \int_{-K}^{-H} f(\tau + iH) \, d \tau
  \approx \sum_{k=1}^{\NE}
f\paren{\zeta_k^{(\calE_1)}} w_k^{(\calE_1)} + \sum_{k=1}^p f\paren{\zeta_k^{(\calA_1)}} w_k^{(\calA_1)},\]
with
\begin{align*}
  \zeta_k^{(\calA_1)} &= -\zeta_{p-k+1}^{(\calA_3)}, & w_k^{(\calA_1)} &= w_{p-k+1}^{(\calA_3)}, \\
  \zeta_k^{(\calE_1)} &= -\zeta_{\NE-k+1}^{(\calE_3)}, & w_k^{(\calE_1)} &= w_{\NE-k+1}^{(\calE_3)}.
\end{align*}

On $\Gamma_2$, or equivalently on $\gamma(\tau)$ with $\tau \in [-H,H]$,
we require a quadrature for a smooth function with nodes exponentially
clustered at the origin. Following \cite{greengard00}, we use a
dyadically-refined composite Gaussian
quadrature rule, defined as follows. Let $x_1^{\text{gau}},\ldots,x_q^{\text{gau}}$ and
$w_1^{\text{gau}},\ldots,w_q^{\text{gau}}$ be the standard Gaussian quadrature nodes
and weights, respectively, on $[-1,1]$, which define a rule of order $2q+1$.
Given a {\em refinement depth} $\nr \in \ZZ^+$, define a set of
panels for $\tau \in [0,H]$ denoted by $[a_0,a_1]$, $[a_1,a_2]$,$\ldots$,
$[a_{\nr-1},a_{\nr}]$, which are
dyadically refined towards the origin as follows:
\[a_k =
  \begin{cases}
    0 & k = 0 \\
    H/2^{\nr-k} & 1 \leq k \leq \nr.
  \end{cases}
\]
Then, supplement this with the reflected panels for $\tau \in [-H,0]$,
namely $[a_{-\nr},a_{-\nr+1}]$,
$[a_{-\nr+1},a_{-\nr+2}]$, $\ldots$, $[a_{-1},a_0]$, defined by
\[a_{-k} = -a_k.\]
On each such panel, we use a Gaussian quadrature rule, rescaled to the panel:
\[\int_{\Gamma_2} f(\zeta) \, d \zeta = (1-i) \int_{-H}^H f(\tau - i \tau)
  \, d \tau \approx \sum_{k=-\nr+1}^{\nr} \sum_{j=1}^q
f\paren{\zeta^{(\calC)}_{j,k}} w^{(\calC)}_{j,k}\]
where
\[\zeta^{(\calC)}_{j,k} = \frac{a_k - a_{k-1}}{2} x_j^{\text{gau}} + \frac{a_{k-1}
  + a_k}{2}, \qquad \qquad w^{(\calC)}_{j,k} = (1-i) \frac{a_k - a_{k-1}}{2}
w_j^{\text{gau}}.\]
For simplicity of notation, we re-index the double sum to a sum over
a single index,
\[\int_{\Gamma_2} f(\zeta) \, d \zeta = (1-i) \int_{-H}^H f(\tau - i \tau)
  \, d \tau \approx \sum_{k=1}^{\NC}
f\paren{\zeta^{(\calC)}_k} w^{(\calC)}_k,\]
where $\NC = 2 \nr q$ and $\zeta^{(\calC)}_k$, $w^{(\calC)}_k$ have been suitably defined
in terms of $\zeta^{(\calC)}_{j,k}$, $w^{(\calC)}_{j,k}$, respectively.
The notation $\NC$ is used to reflect the fact that this is a {\em clustered}
set of nodes.

We can now define the full set of quadrature nodes
$\zeta_1,\ldots,\zeta_N$ and weights $w_1,\ldots,w_N$ on $\Gamma_K$ by
combining the five quadrature rules described above: the equispaced
rules of $\NE$ nodes each on $\Gamma_1$ and $\Gamma_3$, the $p$ nodes
corresponding to Alpert's endpoint corrections on $\Gamma_1$ and
$\Gamma_3$, and the exponentially-clustered composite Gaussian rule of
$\NC$ nodes on $\Gamma_2$. In total, we have $N = 2 \NE + 2p + \NC$ nodes,
with $\NC = 2 \nr q$. From the discussion in Section \ref{sec:cfrep},
$p$ is a fixed constant and $q = \OO{H}$, while $\NE = \OO{(1+\phimax) M}$ and
$\nr = \OO{\log T}$ depend on the frequency content of the solution and
the overall simulation time, respectively. The locations of the quadrature nodes are
illustrated in Figure \ref{fig:nodes}.

\begin{figure}[t]
  \centering
    \includegraphics[width=.9\linewidth]{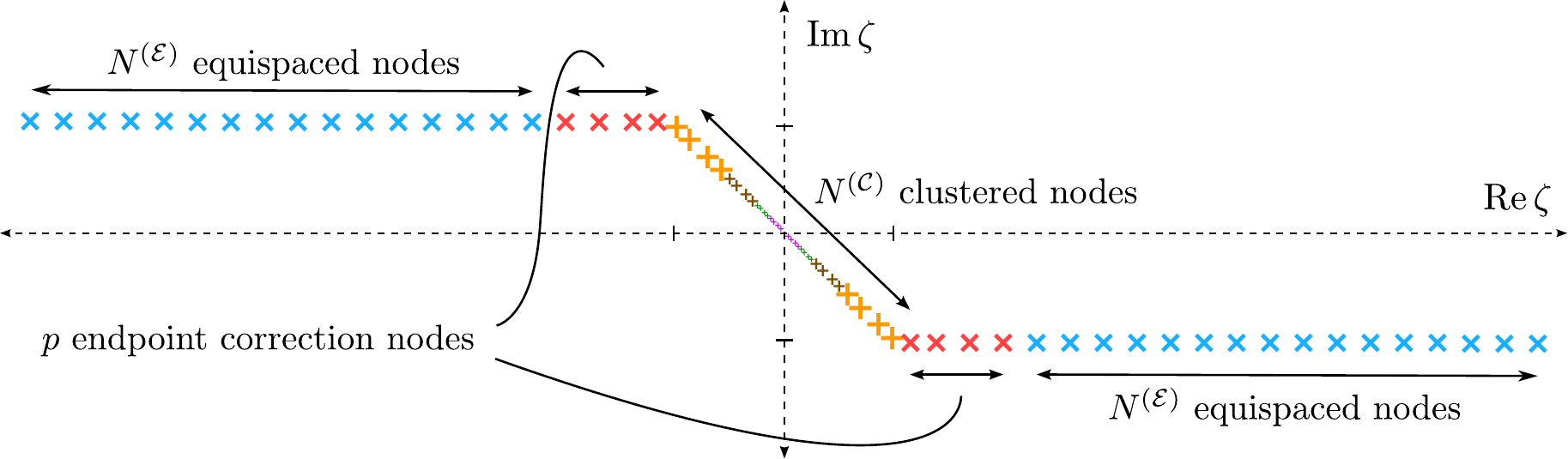}
    \caption{The quadrature nodes $\zeta_1,\ldots,\zeta_N$ for
    $p = 4$, $\NE = 16$, $q = 4$, and $\nr = 4$. There are
   $\NE$ equispaced nodes and $p$ endpoint corrections on each of 
   $\Gamma_1$ and  $\Gamma_3$. 
   On $\Gamma_2$, there are $\NC = 2 \nr q$ exponentially-clustered Gaussian
   nodes.}
    \label{fig:nodes}
\end{figure}

\section{Fast Fourier transforms on $\Gamma$} \label{sec:fftgamma}

We turn now to the fast computation of the complex-frequency
forward and inverse
DFTs, appearing for the one-dimensional case in \eqref{eq:fzdft} and
\eqref{eq:utrapzfd}, respectively. This will complete our description of the free space
method. Our algorithm uses a combination of rescaled, zero-padded FFTs,
Chebyshev interpolation, and
direct summation.

For compatability with the standard
FFT, it is convenient to place some restrictions on the grid spacing and
truncation in the frequency domain. The first is that we assume $K = H + \pi M/2$,
consistent with the principle that the grid spacing in the physical
domain is proportional to the truncation distance in the frequency
domain. The second is that $\NE > M/2$, which is also natural; if it
were not the case, the frequency domain grid would be too coarse to resolve the
highest-frequency planewaves in the complex Fourier representation.
These specific constraints will be derived below.

\subsection{The one-dimensional case} \label{sec:fft1d}

\begin{definition}
The {\em forward DFT} from $[-1,1]$ to $\Gamma$ is given by 
\begin{equation}\label{eq:dft1d}
  \wh{f}_k = \sum_{j=1}^M e^{-i \zeta_k x_j} f_j
\end{equation}
for $k = 1,\ldots,N$, where $x_j = -1 + \frac{2 (j-1)}{M}$ are
equispaced nodes on $[-1,1]$. Here $\zeta_1,\ldots,\zeta_M$ are the
quadrature nodes described in Section \ref{sec:gammaquad}.
  We note that the notation $\wh{f}_k$ no longer refers to the
  coefficients of integer Fourier modes, as in Section \ref{sec:periodic}.
\end{definition}

We can define
five subsets of the Fourier coefficients $\wh{f}_k$ corresponding to the
five subsets of the quadrature nodes. That is,
we associate $\wh{f}^{(\calE_1)}_k$ with
the quadrature node $\zeta^{(\calE_1)}_k$, $\wh{f}^{(\calE_3)}_k$ with the node
$\zeta^{(\calE_3)}_k$, $\wh{f}^{(\calA_1)}_k$ and $\wh{f}^{(\calA_3)}_k$ to the nodes
$\zeta^{(\calA_1)}_k$ and $\zeta^{(\calA_3)}_k$, respectively, and $\wh{f}^{(\calC)}_k$ to
the node $\zeta^{(\calC)}_k$. 
We separate the Fourier coefficients in this manner because the method of
computation is different for each subset. After transforming the five
subsets separately, the resulting coefficients can be concatenated into the
$N$-vector $(\wh{f}_1,\dots,\wh{f}_N)^T$.
There are three transform types: \emph{$\calA$-type}, \emph{$\calC$-type}, and
\emph{$\calE$-type}.

\begin{definition} \label{def:transformtypes}
The coefficients corresponding to Alpert's end-point correction nodes are
given by {\em $\calA$-type transforms}:
\begin{align*}
\wh{f}^{(\calA_1)}_k &= \sum_{j=1}^M e^{-i \zeta^{(\calA_1)}_k x_j} f_j, &
\wh{f}^{(\calA_3)}_k &= \sum_{j=1}^M e^{-i \zeta^{(\calA_3)}_k x_j} f_j,
\end{align*}
for $k = 1,\ldots,p$.
The coefficients corresponding to the
clustered composite Gauss nodes are given by a {\em $\calC$-type transform}:
\begin{equation} \label{eq:fhatCsum}
\wh{f}^{(\calC)}_k = \sum_{j=1}^M e^{-i \zeta^{(\calC)}_k x_j} f_j,
\end{equation}
for $k=1,\ldots,\NC$.
The coefficients corresponding to equispaced nodes are given by 
{\em $\calE$-type transforms}. Using the substitutions
\begin{equation}\label{eq:zetakr1}
  \zeta^{(\calE_1)}_k = \xi^{(\calE_1)}_k + iH,
\end{equation}
where $\xi^{(\calE_1)}_k$ are equispaced nodes on $[-K + \kappa h - h,-H - \kappa h]$,
and
\begin{equation} \label{eq:zetakr3}
  \zeta^{(\calE_3)}_k = \xi^{(\calE_3)}_k - iH,
\end{equation}
where $\xi^{(\calE_3)}_k$ are equispaced nodes on
$[H + \kappa h,K - \kappa h + h]$, these are given by
\begin{equation}\label{eq:ssfft1dr1}
  \wh{f}^{(\calE_1)}_k = \sum_{j=1}^M e^{-i \xi^{(\calE_1)}_k x_j} \paren{e^{H
  x_j} f_j}
\end{equation}
and
\begin{equation}\label{eq:ssfft1dr3}
  \wh{f}^{(\calE_3)}_k = \sum_{j=1}^M e^{-i \xi^{(\calE_3)}_k x_j} \paren{e^{-H
  x_j} f_j}
\end{equation}
for $k = 1,\ldots,\NE$.

\end{definition}

\subsubsection{Fast computation of one-dimensional forward transforms}

The $\calA$-type transforms may be computed by direct summation at a cost of $\OO{M}$,
since $p$ is a fixed constant.

The $\calC$-type transforms may also be computed by direct summation at a cost of $\OO{M
\NC}$. However, a simple interpolation scheme may be used to decrease
this cost if $\NC$ is large. Indeed, although our scheme requires us to sample the Fourier
transform at a clustered set of points $\zeta^{(\calC)}_k \in \Gamma_2$, the
restriction $x_j \in [-1,1]$ ensures that it is smooth in $\Gamma_2$, and in
particular well-resolved by a Chebyshev interpolant of order independent of $\NC$.
To see this, consider the function $e^{-i \zeta x_j}$ for $\zeta \in \Gamma_2$.
Substituting in 
the parameterization $\zeta = \gamma(\tau) = (1-i) \tau$ of $\Gamma_2$
gives
\[e^{-i \zeta x_j} = e^{-\tau x_j} e^{-i \tau x_j}\]
for $\tau \in [-H,H]$. A spectrally accurate approximation is given by
the Chebyshev interpolant 
\begin{equation} \label{eq:expintrpm}
e^{-\tau x_j} e^{- i \tau x_j} \approx \sum_{l=0}^{\NT-1} \lambda_{l,j}
T_l^H(\tau).
\end{equation}
Here $\NT-1$ is the degree of the interpolant, $T_l^H$ is the degree $l$ Chebyshev polynomial of the
first kind rescaled to $[-H,H]$, and $\lambda_{l,j} \in \CC$. Define $\tau^{(\calC)}_k \in [-H,H]$ so that $\gamma(\tau^{(\calC)}_k) =
\zeta^{(\calC)}_k$ for $k = 1,\ldots,\NC$. Then plugging the interpolant
into \eqref{eq:fhatCsum}, evaluating at the points $\tau^{(\calC)}_k$, and changing the order of summation gives
\[\wh{f}^{(\calC)}_k \approx \sum_{l=0}^{\NT-1} T_l^H(\tau^{(\calC)}_k) \sum_{j=1}^M
\lambda_{l,j} f_j.\]
This expression may be computed for
every $k = 1,\ldots,\NC$ directly in $\OO{M \NT + \NT \NC}$ operations.
Since $x_j \in [-1,1]$, we can estimate $\NT = \OO{H}$, and in
particular $\NT$ does not depend on $\NC$. In Section
\ref{sec:gammares} we show $H =
\OO{(1+\phimax)^{-1}}$, which may be estimated as $\OO{1}$ for
simplicity, so we can estimate $\NT = \OO{1}$. This scheme therefore reduces
the cost of computing the coefficients $\wh{f}^{(\calC)}_k$ from $\OO{M
\NC}$ to $\OO{M + \NC}$.

The $\calE$-type transforms may be thought of as {\em
shifted and scaled} versions of the standard DFT, applied to rescaled
inputs. Indeed, the standard DFT, given by
\begin{equation} \label{eq:dftstandard}
  \wh{c}_k = \sum_{j=1}^n e^{-2 \pi i (j-1)(k-1)/n} c_j
\end{equation}
for $k=1,\ldots,n$, maps values at $n$
equispaced nodes on $[0,2 \pi)$ to the coefficients of $n$ equispaced
frequencies on $[0,n)$. On the other hand, \eqref{eq:ssfft1dr1} and \eqref{eq:ssfft1dr3} are of the general form
\begin{equation} \label{eq:ssfftgen}
  \wh{c}_k = \sum_{j=1}^m e^{-i \xi_k x_j} c_j,
\end{equation}
where $\xi_k = \alpha + \frac{(\beta-\alpha)
(k-1)}{n}$, for $k=1,\ldots,n$. This transform maps
values at $m$ equispaced nodes on $[-1,1)$ to the coefficients of $n$
equispaced frequencies on
$[\alpha,\beta)$. 
Let us describe how to compute these transforms efficiently.

We first expand and rewrite
\eqref{eq:ssfftgen} as
\[\wh{c}_k = e^{i (\alpha + (\beta-\alpha) (k-1)/n)} \sum_{j=1}^m
e^{-i 2(\beta-\alpha)(j-1)(k-1)/mn} \paren{e^{-i 2 \alpha (j-1)/m} c_j}.\]

Let $\nu \geq \max(m,n)$ be an integer, and extend $c_j$ to $j =
1,\ldots,\nu$ by setting $c_j = 0$ for $j > m$.
Then we can take the above sum over $\nu$ terms:
\begin{equation} \label{eq:ssdftp}
  \wh{c}_k = e^{i (\alpha + (\beta-\alpha) (k-1)/n)} \sum_{j=1}^\nu
e^{-i 2(\beta-\alpha)(j-1)(k-1)/mn} \paren{e^{-i 2 \alpha (j-1)/m} c_j}.
\end{equation}
If $\alpha$, $\beta$, and $\nu$ are such that $(\beta-\alpha)/mn = \pi/\nu$,
then the sums in \eqref{eq:ssdftp}, for $k=1,\ldots,\nu$, are standard DFTs of
size $\nu$. We can therefore use this expression to compute
\eqref{eq:ssfftgen} in $\OO{\nu \log \nu}$ operations; we
pre-multiply and zero-pad the input values $c_j$, apply an FFT,
and post-multiply and truncate the output coefficients $\wh{c}_k$.

For the transforms \eqref{eq:ssfft1dr1} and \eqref{eq:ssfft1dr3}, we have $m
= M$, $n = \NE$, and $\beta - \alpha = K - H - (2 \kappa -
  1) h$. We must therefore choose $K$ and $\nu$
so that
\[\frac{K - H - (2 \kappa - 1) h}{M \NE} = \frac{\pi}{\nu}.\]
Recall from Section \ref{sec:gammaquad} that $h = (K-H)/(\NE + 2 \kappa - 1)$ is chosen so that
\[H + \kappa h + (\NE - 1) h = K - \kappa h.\]
After some manipulation, this expression becomes
\[\frac{K - H - (2 \kappa - 1) h}{\NE} = h = \frac{K-H}{\NE + 2 \kappa - 1}\]
so the condition on $\nu$ becomes
\[\frac{K-H}{M (\NE + 2 \kappa - 1)} = \frac{\pi}{\nu}.\]
We make the convenient---though not essential---choice $\nu = 2
(\NE + 2 \kappa - 1)$, so that $K = H + \pi M/2$.
If we assume $\NE > M/2$, we have $\nu \geq
\max(M,\NE)$, as required. Thus we obtain the restrictions mentioned
above. With this choice of $\nu$, we have an algorithm to compute \eqref{eq:ssfft1dr1} and
\eqref{eq:ssfft1dr3} in $\OO{\NE \log \NE}$
operations. We refer to it as a 
{\em shifted and scaled FFT}.

Thus, the total cost to compute all the Fourier coefficients
is $\OO{M + \NC + \NE \log \NE}$. Using $\NC = \OO{\log T}$, and $\NE = \OO{M}$, we obtain the cost estimate $\OO{M
\log M + \log T}$. To take into account the scaling with $\phimax$ in
the $A \neq 0$ case, we require $\NE = \OO{\phimax M}$, 
giving the estimate $\OO{\phimax M \log \paren{ \phimax M} + \log{T}}$.

\subsubsection{The one-dimensional inverse transform}

\begin{definition}
The {\em inverse DFT} from $\Gamma$ to $[-1,1]$ is defined by 
\[f_j = \sum_{k=1}^N e^{i \zeta_k x_j} \wh{f}_k\]
for $j = 1,\ldots,M$.

In this case, we split the transform into five components:
\begin{multline*}
f_j = f^{(\calE_1)}_j + f^{(\calA_1)}_j + f^{(\calC)}_j + f^{(\calA_3)}_j + f^{(\calE_3)}_j
\\ = \sum_{k=1}^{\NE} e^{i \zeta^{(\calE_1)}_k x_j}
\wh{f}^{(\calE_1)}_k + \sum_{k=1}^{p} e^{i \zeta^{(\calA_1)}_k x_j}
\wh{f}^{(\calA_1)}_k + \sum_{k=1}^{\NC} e^{i \zeta^{(\calC)}_k x_j}
\wh{f}^{(\calC)}_k + \sum_{k=1}^{p} e^{i \zeta^{(\calA_3)}_k x_j}
\wh{f}^{(\calA_3)}_k + \sum_{k=1}^{\NE} e^{i \zeta^{(\calE_3)}_k x_j}
\wh{f}^{(\calE_3)}_k.
\end{multline*}

We again distinguish three inverse transform types, which may defined in
a similar manner to their analogues for the forward transform in
Definition \ref{def:transformtypes}.
\end{definition}
  
The inverse $\calA$-type,
$\calC$-type, and $\calE$-type transforms may be computed
by techniques similar to those described above.

The values corresponding to the 
{\em $\calA$-type} coefficients, namely
$f^{(\calA_1)}_j$ and $f^{(\calA_3)}_j$ for $j=1,\ldots,M$, may be computed 
in $\OO{M}$ operations by direct summation.

The values corresponding to the {\em $\calC$-type} coefficients,
$f^{(\calC)}_j$, may be computed by 
direct summation for small $\NC$, or by a Chebyshev
interpolation scheme for large $\NC$. Using the interpolants
\begin{equation} \label{eq:expintrpp}
  e^{i \zeta x_j} = e^{\tau x_j} e^{i \tau x_j} \approx
\sum_{l=0}^{\NT-1} \rho_{j,l} T_l^H(\tau)
\end{equation}
gives
\[f^{(\calC)}_j \approx \sum_{l=0}^{\NT-1} \rho_{j,l} \sum_{k=1}^{\NC} T_l^H(\tau^{(\calC)}_k)
\wh{f}^{(\calC)}_k\]
which, as before, may be computed for every $j = 1,\ldots,M$ in $\OO{\NC + M}$ operations.

To compute the values corresponding to the {\em $\calE$-type} coefficients, $f^{(\calE_1)}_j$ and
$f^{(\calE_3)}_j$, we use
\eqref{eq:zetakr1} and \eqref{eq:zetakr3} to obtain
\begin{equation} \label{eq:ssifft1dr1}
  f^{(\calE_1)}_j = e^{-H x_j} \sum_{k=1}^{\NE} e^{i \xi^{(\calE_1)}_k x_j} \wh{f}^{(\calE_1)}_k.
\end{equation}
and
\begin{equation} \label{eq:ssifft1dr3}
  f^{(\calE_3)}_j = e^{H x_j} \sum_{k=1}^{\NE} e^{i \xi^{(\calE_3)}_k x_j} \wh{f}^{(\calE_3)}_k,
\end{equation}
respectively. These are shifted and scaled inverse DFTs, with rescaled
outputs, and may be computed in a similar manner to the shifted and
scaled DFTs. Now, our algorithm is built on the standard inverse FFT,
which computes
\[c_j = \sum_{k=1}^n e^{2 \pi i (j-1)(k-1)/n} \wh{c}_k\]
in $\OO{n \log n}$ operations.
The transforms in \eqref{eq:ssifft1dr1} and \eqref{eq:ssifft1dr3} are of
the form
\begin{equation} \label{eq:ssifftgen}
  c_j = \sum_{k=1}^n \wh{c}_k \, e^{i \xi_k x_j} 
\end{equation}
for $j = 1,\ldots,m$, with $\xi_k$ defined as before.
Writing \eqref{eq:ssifftgen} as
\[c_j = e^{i \alpha (-1 + 2(j-1)/m)} \sum_{k=1}^n
\paren{e^{-i (\beta-\alpha)(k-1)/n} \, 
\wh{c}_k} \, 
e^{i 2 (\beta-\alpha)(j-1)(k-1)/mn},
\]
we pre-multiply
and zero-pad the input
coefficients $\wh{c}_k$ to a set of $\nu$ values for properly chosen $\nu$, 
perform an inverse FFT of size $\nu$, and
post-multiply and truncate the outputs. Given the 
parameters corresponding to \eqref{eq:ssifft1dr1} and
\eqref{eq:ssifft1dr3}, the condition on $\nu$ is the same as before, and
we can make the same choice. The cost to compute \eqref{eq:ssifft1dr1}
and \eqref{eq:ssifft1dr3} is therefore again
$\OO{\NE \log \NE}$.

The cost to obtain all of
the values $f_j$ is therefore $\OO{M + \NC + \NE \log
\NE}$, as for the forward transform, and the estimates written with respect
to $M$, $T$, and $\phimax$ are identical.

\subsection{The two-dimensional case} \label{sec:fft2d}

\begin{definition}
The {\em forward DFT} from $[-1,1]^2$ to $\Gamma^2$ is given by
\begin{equation}\label{eq:dft2d}
  \wh{f}_{k_1,k_2} = \sum_{j_1=1}^{M_1} \sum_{j_2=1}^{M_2} e^{-i
    \paren{\zeta_{k_1}
  x_{j_1} + \omega_{k_2} y_{j_2}}} f_{j_1,j_2}
\end{equation}
for $k_1 = 1,\ldots,N_1$ and $k_2 = 1,\ldots,N_2$.
\end{definition}

The additional subscripts on the various indices refer to the spatial
dimension. The discretization nodes in the physical domain are
given by $(x_{j_1},y_{j_2}) \in [-1,1]^2$ for $j_1 = 1,\ldots,M_1$ and $j_2 =
1,\ldots,M_2$. Similarly, the quadrature nodes in the complex-frequency
domain are given by $(\zeta_{k_1},\omega_{k_2}) \in \Gamma^2$ for $k_1 =
1,\ldots,N_1$ and $k_2 = 1,\ldots,N_2$.
We have therefore allowed for the possibility that different
discretizations are used in the two coordinate directions. This may be useful, for example,
if the vector potential $A(t)$ has a larger amplitude in one dimension
than in the other, or if the support of the scalar potential $V$ is
anisotropic. We define $M = M_1 M_2$ to be the total number of spatial
grid points.

We can split the Fourier coefficients $\wh{f}_{k_1,k_2}$ into
subsets corresponding to pairs of subsets of quadrature nodes. For example,
the coefficient
\[\wh{f}^{(\calE_3,\calA_1)}_{k_1,k_2} =\sum_{j_1=1}^{M_1} \sum_{j_2=1}^{M_2} e^{-i
    \paren{\zeta^{(\calE_3)}_{k_1} x_{j_1} +
\omega^{(\calA_1)}_{k_2} y_{j_2}}} f_{j_1,j_2}\]
corresponds to the pair of nodes
$\paren{\zeta^{(\calE_3)}_{k_1},\omega^{(\calA_1)}_{k_2}}$.
Since there
are five types of subsets of nodes in one dimension, there are 25
types of node pairs and therefore of Fourier coefficients in two
dimensions.
The 25 transforms can be divided into six 
general types, which we will denote by 
$(\calA,\calA)$, $(\calA,\calE)$, $(\calA,\calC)$, $(\calC,\calE)$,
$(\calC,\calC)$, and $(\calE,\calE)$. These may be defined in
a straightforward manner. The different subsets of coefficients may
again be computed separately using their corresponding transforms and then
concatenated.

\subsubsection{Fast computation of two-dimensional forward transforms}

There are four $(\calA,\calA)$-type subsets of coefficients;
$\wh{f}^{(\calA_1,\calA_1)}_{k_1,k_2}$, $\wh{f}^{(\calA_1,\calA_3)}_{k_1,k_2}$,
$\wh{f}^{(\calA_3,\calA_1)}_{k_1,k_2}$, and $\wh{f}^{(\calA_3,\calA_3)}_{k_1,k_2}$.
For the first case, we write
\[\wh{f}^{(\calA_1,\calA_1)}_{k_1,k_2} = \sum_{j_1=1}^{M_1} e^{-i
    \zeta^{(\calA_1)}_{k_1} x_{j_1}} \sum_{j_2=1}^{M_2} e^{-i \omega^{(\calA_1)}_{k_2} y_{j_2}} f_{j_1,j_2},\]
where we have rearranged the sums to separate variables. The inner
sums may be computed by $M_1$ one-dimensional $\calA$-type transforms,
and the outer sums by $p$ one-dimensional $\calA$-type transforms, at a cost of $\OO{M}$. 
The other $(\calA,\calA)$-type transforms may be computed similarly.

There are eight $(\calA,\calE)$-type subsets; $\wh{f}^{(\calA_1,\calE_1)}_{k_1,k_2}$,
$\wh{f}^{(\calE_1,\calA_1)}_{k_1,k_2}$, $\wh{f}^{(\calA_1,\calE_3)}_{k_1,k_2}$,
$\wh{f}^{(\calE_3,\calA_1)}_{k_1,k_2}$, $\wh{f}^{(\calA_3,\calE_1)}_{k_1,k_2}$,
$\wh{f}^{(\calE_1,\calA_3)}_{k_1,k_2}$, $\wh{f}^{(\calA_3,\calE_3)}_{k_1,k_2}$, and
$\wh{f}^{(\calE_3,\calA_3)}_{k_1,k_2}$.
For the first case, after plugging in \eqref{eq:zetakr1} and
rearranging, we obtain
\[\wh{f}^{(\calA_1,\calE_1)}_{k_1,k_2} = \sum_{j_2=1}^{M_2} e^{-i
  \xi^{(\calE_1)}_{k_2} y_{j_2}} \paren{e^{H_2 y_{j_2}} \sum_{j_1=1}^{M_1}
e^{-i \zeta^{(\calA_1)}_{k_1} x_{j_1}} f_{j_1,j_2}}.\]
The inner sums may be computed by $M_2$ $\calA$-type transforms, and the outer sums 
by $p$ $\calE$-type transforms, at a cost of 
$\OO{M + \NE_1 \log \NE_1}$. Other $(\calA,\calE)$-type transforms are computed in the same manner,
and the total cost of computing them all
is of the order $\OO{M + \NE_1 \log \NE_1 + \NE_2 \log \NE_2}$. We note
that writing the sums in a different order would lead to an algorithm
with a greater computational cost; in all cases, the $\calA$-type transform should be taken as the
inner transform.

There are four $(\calA,\calC)$-type subsets;
$\wh{f}^{(\calA_1,\calC)}_{k_1,k_2}$,
$\wh{f}^{(\calC,\calA_1)}_{k_1,k_2}$, $\wh{f}^{(\calA_3,\calC)}_{k_1,k_2}$,
and $\wh{f}^{(\calC,\calA_3)}_{k_1,k_2}$.
Separating the sums in the first case gives
\[\wh{f}^{(\calA_1,\calC)}_{k_1,k_2} = \sum_{j_2=1}^{M_2} e^{-i
    \omega^{(\calC)}_{k_2} y_{j_2}} \sum_{j_1=1}^{M_1} e^{-i \zeta^{(\calA_1)}_{k_1} x_{j_1}} f_{j_1,j_2}.\]
The inner sums may be computed by $M_2$ $\calA$-type transforms, and the outer
sums by $p$ $\calC$-type transforms, at a cost of $\OO{M + \NC_2}$. The cost of computing
all $(\calA,\calC)$-type transforms is $\OO{M + \NC_1 + \NC_2}$. 
For efficiency, the $\calA$-type transform should be taken as the inner
transform.

There are four $(\calC,\calE)$-type subsets;
$\wh{f}^{(\calC,\calE_1)}_{k_1,k_2}$, $\wh{f}^{(\calE_1,\calC)}_{k_1,k_2}$,
$\wh{f}^{(\calC,\calE_3)}_{k_1,k_2}$, and $\wh{f}^{(\calE_3,\calC)}_{k_1,k_2}$. Unlike
the first three cases above, we do not simply
separate variables and repeatedly apply the
one-dimensional algorithms. Using
\eqref{eq:zetakr1} and rearranging the sums in the first case gives 
\[\wh{f}^{(\calC,\calE_1)}_{k_1,k_2} = \sum_{j_2=1}^{M_2} e^{-i
  \xi^{(\calE_1)}_{k_2} y_{j_2}} \paren{e^{H_2 y_{j_2}} \sum_{j_1=1}^{M_1}
e^{-i \zeta^{(\calC)}_{k_1} x_{j_1}} f_{j_1,j_2}}.\]
Using the interpolant \eqref{eq:expintrpm} in the $\calC$-type transform and
rearranging the sums again gives
\[\wh{f}^{(\calC,\calE_1)}_{k_1,k_2} = \sum_{l=0}^{\NT_1-1} T_l^{H_1}
  (\tau^{(\calC)}_{k_1}) \sum_{j_2=1}^{M_2} e^{-i
  \xi^{(\calE_1)}_{k_2} y_{j_2}} \paren{e^{H_2 y_{j_2}} \sum_{j_1=1}^{M_1}
\lambda_{l,j_1} f_{j_1,j_2}}.\]
The inner sums may be computed directly for each $j_2 =
1,\ldots,M_2$, the middle sum by $\NT_1$ $\calE$-type transforms, and the
outer sum directly 
for each $k_2 = 1,\ldots,\NE_2$. The cost of computing this transform is
therefore $\OO{M + \NE_2 \log \NE_2 + \NE_2 \NC_1}$, and the cost of
computing all $(\calC,\calE)$-type transforms is $\OO{M + \NE_1 \log \NE_1 + \NE_2
  \log \NE_2 + \NE_1 \NC_1 + \NE_2 \NC_2}$.

There is only one $(\calC,\calC)$-type subset: $\wh{f}^{(\calC,\calE_1)}_{k_1,k_2}$.
Plugging in the interpolant \eqref{eq:expintrpm} and rearranging gives
\[\wh{f}^{(\calC,\calC)}_{k_1,k_2} = \sum_{l_1=0}^{\NT_1-1} T_{l_1}^{H_1}
  (\tau^{(\calC)}_{k_1}) \sum_{l_2=0}^{\NT_2-1} T_{l_2}^{H_2}
  (\sigma^{(\calC)}_{k_2}) \sum_{j_1=1}^{M_1}
\lambda_{l_1,j_1} \sum_{j_2=1}^{M_2}
\lambda_{l_2,j_2} f_{j_1,j_2},\]
where we have used the nodes
$(\tau^{(\calC)}_{k_1},\sigma^{(\calC)}_{k_2}) \in [-H,H]^2$ as the
quadrature nodes in the two-dimensional parameter space.
Each sum may be computed directly at a total cost of $\OO{M
+ \NC_1 \NC_2}$.

There are four $(\calE,\calE)$-type subsets;
$\wh{f}^{(\calE_1,\calE_1)}_{k_1,k_2}$, $\wh{f}^{(\calE_1,\calE_3)}_{k_1,k_2}$,
$\wh{f}^{(\calE_3,\calE_1)}_{k_1,k_2}$, and $\wh{f}^{(\calE_3,\calE_3)}_{k_1,k_2}$.
After using the substitutions \eqref{eq:zetakr1} and \eqref{eq:zetakr3},
these may be written as shifted and scaled two-dimensional DFTs. The generalization of the
shifted and scaled FFT from one to two dimensions is straightforward,
and we omit the details. It uses a standard
two-dimensional FFT of size $\nu_1 \times \nu_2$, with $\nu_1$ and $\nu_2$ chosen as in the
one-dimensional case using the quadrature parameters corresponding to their
dimensions. We obtain an algorithm with a cost of
$\OO{\NE_1 \NE_2 \log \paren{\NE_1 \NE_2}}$. 

\bigskip

Combining all cases, we find that the total cost to compute the
two-dimensional forward transform is
\[
\OO{M + \NE_1 \NC_1 + \NE_2 \NC_2 + \NC_1 \NC_2 + \NE_1 \NE_2
  \log\paren{\NE_1 \NE_2}}~.
\]
If we take $A = 0$ and use the scalings with respect to $M$ and $T$, this expression becomes
\[\OO{M \log M + (M_1 + M_2) \log T + \log^2 T}.\]
If we take into account the scaling with respect to a field
$A(t) = (A_1(t),0)^T$ aligned with the first coordinate dimension, we
obtain the estimate
\[\OO{\phimax_1 M \log\paren{\phimax_1 M} + \paren{\phimax_1 M_1 + M_2}
\log T + \log^2 T}.\]
In the general
case $A(t) = (A_1(t),A_2(t))^T$, the estimate is
\[\OO{\phimax_1 \phimax_2 M \log\paren{\phimax_1 \phimax_2 M} +
    \paren{\phimax_1 M_1 + \phimax_2 M_2} \log T + \log^2 T}.\]

\subsubsection{The two-dimensional inverse transform}

\begin{definition}
The {\em inverse DFT} from $\Gamma^2$ to $[-1,1]^2$ is given by
\[f_{j_1,j_2} = \sum_{k_1=1}^{N_1} \sum_{k_2=1}^{N_2} e^{i
    \paren{\zeta_{k_1} x_{j_1} +
\omega_{k_2} y_{j_2}}} \wh{f}_{k_1,k_2}\]
for $j_1 = 1,\ldots,M_1$ and $j_2 = 1,\ldots,M_2$.
\end{definition}

The transform may be split into a sum of 25 terms corresponding to
different pairs of subsets of quadrature nodes. For example,
\[f^{(\calE_3,\calA_1)}_{j_1,j_2} = \sum_{k_1=1}^{\NE_1} \sum_{k_2=1}^p
  e^{i\paren{\zeta^{(\calE_3)}_{k_1} x_{j_1} + \omega^{(\calA_1)}_{k_2}
  y_{j_2}}} \wh{f}^{(\calE_3,\calA_1)}_{k_1,k_2}\]
corresponds to the pair of nodes
$\paren{\zeta^{(\calE_3)}_{k_1},\omega^{(\calA_1)}_{k_2}}$. As before, there
are six transform types. The algorithms used for each transform type are closely
related to their analogues in the forward transform and have the same algorithmic
complexity.

As for the forward transform, the $(\calA,\calA)$-type inverse transforms can be computed
by separation of variables and direct summation.
For the $(\calA,\calE)$-type transforms,
we use separation of variables and apply
the $\calA$ and $\calE$-type one-dimensional transforms, except in the reverse order:
the $\calE$-type transform must be taken as the inner transform to obtain
the same complexity as for the forward transform.

For the $(\calA,\calC)$-type transforms, as for the 
$(\calA,\calE)$-type, we
separate variables and apply the one-dimensional transforms in the
reverse order:  the $\calC$-type transform is taken as the inner transform.

For the $(\calC,\calE)$-type transforms, 
we use \eqref{eq:zetakr1} and the interpolant
\eqref{eq:expintrpp}, and rearrange in the form:
\[f^{(\calC,\calE_1)}_{j_1,j_2} = e^{-H_2 y_{j_2}} \sum_{l=0}^{\NT_1-1}
  \rho_{j_1,l} \sum_{k_2=1}^{\NE_2} e^{i \xi^{(\calE_1)}_{k_2} y_{j_2}}
  \sum_{k_1=1}^{\NC_1} T^{H_1}_l\paren{\tau^{(\calC)}_{k_1}}
\wh{f}^{(\calC,\calE_1)}_{k_1,k_2}.\]
The inner and outer transforms may be computed by direct summation, and
the middle as an $\calE$-type transform. The other $(\calC,\calE)$-type inverse transforms
are handled analogously.

The $(\calC,\calC)$-type inverse transform can be written,
using the interpolant \eqref{eq:expintrpp}, in the form
\[f^{(\calC,\calC)}_{j_1,j_2} = \sum_{l_1=0}^{\NT_1-1} \rho_{j_1,l_1}
  \sum_{l_2=0}^{\NT_2-1} \rho_{j_2,l_2} \sum_{k_1=1}^{\NC_1}
  T_{l_1}^{H_1}\paren{\tau^{(\calC)}_{k_1}} \sum_{k_2=1}^{\NC_2}
  T_{l_2}^{H_2}\paren{\sigma^{(\calC)}_{k_2}} \wh{f}^{(\calC,\calC)}_{j_1,j_2}.\]
Each transform may be computed by direct summation.
Finally, the $(\calE,\calE)$-type transforms
may be computed using a two-dimensional shifted and scaled inverse
FFT, which is again a simple generalization of the one-dimensional case.

\subsection{The three-dimensional case} \label{sec:fft3d}

The techniques we have
described may be used in the same manner to
design a fast algorithm for the three-dimensional case. There are
$5^3 = 125$ subsets of distinct types of quadrature node triplets, 
and 10 distinct transform types. If
$A = 0$, one can derive an algorithm with a cost of
\[\OO{M \log M +
  \paren{M_1 M_2 + M_1 M_3 + M_2 M_3} \log T + \paren{M_1 + M_2 + M_3}
\log^2 T + \log^3 T}.\]
The estimate for the general
case including a vector potential is more involved and is omitted.
A practical rule of thumb is that for each
non-zero component $A_i$ of $A$, the cost increases approximately by
a factor $\phimax_i$. 

\section{Analysis of the complex-frequency representation}
\label{sec:analysis}

In this section we expand on the discussion in Section \ref{sec:cfrep},
presenting analysis supporting our choice of
the contour $\Gamma$ and our quadrature estimates.
Our goal is to establish the accuracy of the discretizations
\eqref{eq:uciftdisc} and \eqref{eq:Vuciftdisc} of the complex Fourier representations
\begin{equation} \label{eq:ucift}
  u(x,t) = \frac{1}{2 \pi} \int_\Gamma e^{i \zeta x} \uhat(\zeta,t) \, d \zeta
\end{equation}
and
\begin{equation} \label{eq:Vucift}
  (Vu)(x,t) = \frac{1}{2 \pi} \int_\Gamma e^{i \zeta x}
  \Vuhat(\zeta,t) \, d \zeta,
\end{equation}
respectively, using $N = \OO{(1 + \phimax) K_0 + \log T} = \OO{(1 +
\phimax) M + \log T}$ quadrature
nodes. Here, $K_0$ denotes a truncation parameter for the
classical Fourier representation that guarantees a prescribed accuracy, 
as in \eqref{eq:ifturealtrunc}.
We will first show that these integrals may be truncated to contours $\Gamma_K$ with
$K = K_0 + \OO{1}$, thereby establishing $M = \OO{K_0}$, since $M = \OO{K}$
in our algorithm. We will
then show that the truncated integrals may be accurately resolved by the
stated number of quadrature nodes. It is sufficient to focus on the one-dimensional case, since the
$d$-dimensional quadrature rule is a tensor product of the
one-dimensional rules.

\subsection{Analysis of truncation}\label{sec:gammatrunc}

Here we demonstrate that our deformation of the inverse Fourier
transform from $\RR$ to $\Gamma$ does not significantly increase the
real-frequency truncation of the integral. In particular,
we show that we may choose a truncation $\abs{\Re(\zeta)} \leq K = K_0 +
\OO{1}$, with the $\OO{1}$ scaling depending only on $H$ and
$\varepsilon$.

We first show that the magnitude of the analytic
continuation of the Fourier transform of a function $f \in
C^\infty([-1,1])$ is controlled by its nearby values on the real line.
\begin{lemma} \label{lem:fhatcontinuation}
  For any imaginary shift $\eta > 0$, there is a constant $C > 0$ such
  that the following holds:
  for every $\varepsilon > 0$ there is an $L > 0$ such that
  for every $f \in C^\infty([-1,1])$,
  \[\abs{\wh{f}(\xi + i\eta)} \leq C \max_{-L \leq \nu \leq
    L} \abs{\wh{f}(\xi+\nu)} + \norm{f}_2 \varepsilon\]
  for all $\xi \in \RR$. The dependence of $C$ on $\eta$ is continuous,
  $C = C(\eta)$,
  and for fixed $\varepsilon$ the dependence of $L$ on $\eta$ is also
  continuous, $L = L(\eta)$.
\end{lemma}
\begin{proof}
Let $\psi \in C_c^\infty(\RR)$, the space of smooth functions of compact
  support, with $\psi \equiv 1$ on $[-1,1]$. Then since $f \in
  C^\infty([-1,1])$, we have, for any $\xi \in \RR$,
  \be
    \wh{f}(\xi + i \eta) = \int_{-\infty}^\infty e^{-i \xi x}
    \paren{e^{\eta x} f(x)} \, dx
    = \int_{-\infty}^\infty e^{-i \xi x}
    \paren{e^{\eta x} \psi(x) f(x)} \, dx
    = \frac{1}{2 \pi} \paren{\wh{f} \ast \wh{\phi}_\eta}(\xi),
\ee
  where $\phi_\eta(x) = e^{\eta x} \psi(x)$.
  Since $\psi \in C_c^\infty(\RR)$, so is $\phi_\eta$, and
  $\wh{\phi}_\eta$ is rapidly decaying.
  In particular, for each $n \in \ZZ^+$,
\[\wh{\phi}_\eta(\xi) = \int_{-\infty}^\infty e^{-i \xi x}
  \phi_\eta(x) \, dx = \frac{1}{(i \xi)^n} \int_{-\infty}^\infty e^{-i \xi x}
  \phi_\eta^{(n)} (x) \, dx\]
so
\[\abs{\wh{\phi}_\eta(\xi)} \leq
  \frac{\norm{\phi_\eta^{(n)}}_1}{\abs{\xi}^n}. \]
Therefore given $\varepsilon > 0$, there is an $L > 0$ depending
  continuously on $\eta$ so that
\begin{equation}\label{eq:phitrunc}
\sqrt{2 \pi \int_{\abs{\xi}>L}
\abs{\wh{\phi}_\eta(\xi)}^2 \, d \xi} < \varepsilon.
\end{equation}
We now split the frequency domain convolution into two terms,
\[\paren{\wh{f} \ast \wh{\phi}_\eta}(\xi) = \int_{-L}^L
  \wh{f}(\xi-\nu) \wh{\phi}_\eta(\nu) \, d \nu + \int_{|\nu|>L}
\wh{f}(\xi-\nu) \wh{\phi}_\eta(\nu) \, d \nu.\]
To bound the first term, we have
\[\abs{\int_{-L}^L
  \wh{f}(\xi-\nu) \wh{\phi}_\eta(\nu) \, d \nu} \leq 
\norm{\wh{\phi}_\eta}_1 \max_{-L \leq \nu \leq L} \abs{\wh{f}(\xi-\nu)}.\]
For the second term, we have
\[\abs{\int_{|\nu|>L}
  \wh{f}(\xi-\nu) \wh{\phi}_\eta(\nu) \, d \nu} \leq
  \sqrt{\int_{|\nu|>L} \abs{\wh{f}(\xi-\nu)}^2 \, d \nu} \cdot
  \sqrt{\int_{|\nu|>L} \abs{\wh{\phi}_\eta(\nu)}^2 \, d\nu} \leq
  \frac{\norm{\wh{f}}_2 \varepsilon}{\sqrt{2 \pi}} = \norm{f}_2 \varepsilon
\]
from \eqref{eq:phitrunc}.
Combining these bounds gives
the result with $C = \norm{\wh{\phi}_\eta}_1$.
\end{proof}

The next lemma relates the truncation of the classical
Fourier representation of a function $f \in C^\infty([-1,1])$ to that of
the complex Fourier representation modulated by an analytic weight
function $g$. The weight function is included for later convenience.

\begin{lemma}\label{lem:fifttrunc}
  Let $f \in C^\infty([-1,1])$,
  and $\Gamma$ defined
  by \eqref{eq:gammadef} as above with fixed $H > 0$. Let $g$ be
  analytic in an open set containing the strip $\Im(\zeta) \leq H$ with $|g| \leq B$ on
  $\Gamma$.
  Then there is a $C$ such that the following holds:
  for any $\varepsilon > 0$, and $K_0 > H$
  sufficiently large that
  \begin{equation}\label{eq:ftrunceps}
    \int_{\abs{\xi}>K_0} \abs{\wh{f}(\xi)} \, d \xi < \varepsilon
  \end{equation}
  and
  \begin{equation}\label{eq:fvaleps}
    \abs{\wh{f}(\xi)} < \varepsilon
  \end{equation}
  for $\abs{\xi}>K_0$, there is an $L>0$ so that if $K =
  K_0 + L$, then 
  \[ \abs{\int_{\Gamma \backslash \Gamma_{K}} e^{i \zeta x}
    g(\zeta) \wh{f}(\zeta) \, d \zeta} < B C \varepsilon
\qquad
\mbox{ for all } x \in [-1,1]~.
\]
Here, $L$ depends only on $H$ and $\varepsilon$, but not on $f$.
$C$ depends only on $H$ and $\norm{f}_2$,
and in particular not on $\varepsilon$.
\end{lemma}

\begin{proof}
  For any $K > K_0$, we have 
  \[\int_{\Gamma \backslash \Gamma_{K}} e^{i \zeta x} g(\zeta)
    \wh{f}(\zeta) \, d \zeta = \int_{K}^\infty e^{i
    (\tau-iH) x} g(\tau - iH) \wh{f}(\tau - iH) \, d \tau +
      \int_{-\infty}^{-K} e^{i
      (\tau - iH) x} g(\tau - iH) \wh{f}(\tau - iH) \, d \tau.
  \]
  We analyze the first integral; the analysis for the second is
  identical. The integrand is analytic in an open set containing the
  strip $\Im(\zeta) \leq H$, so Cauchy's theorem gives
  \[\int_{K}^\infty e^{i (\tau - iH) x} g(\tau - iH) \wh{f}(\tau - iH) \, d \tau = \int_{K}^\infty e^{i \xi x}
    g(\xi) \wh{f}(\xi) \, d \xi + i \int_0^H e^{i (K - i \eta) x}
    g(K - i \eta) \wh{f}(K - i \eta) \, d \eta\]
  and we have
  \begin{equation}\label{eq:tailtriangle}
    \abs{\int_{K}^\infty e^{i (\tau - iH) x} g(\tau - iH) \wh{f}(\tau - iH) \, d
    \tau} \leq B \paren{\int_{K}^\infty \abs{\wh{f}(\xi,t)} \, d \xi +
    e^H \int_0^H \abs{\wh{f}(K - i \eta)} \, d \eta}
  \end{equation}
  for every $x \in [-1,1]$. The first term in parentheses is bounded by $\varepsilon$, using
  \eqref{eq:ftrunceps}. To bound the second term, we apply Lemma
  \ref{lem:fhatcontinuation}, with our choice of $\varepsilon$.
  We obtain constants $C = \max_{0 \leq \eta \leq H} C(\eta)$ and $L =
  \max_{0 \leq \eta \leq H} L(\eta)$ so that
  \[
  \int_0^H \abs{\wh{f}(K - i \eta)} \, d \eta \leq \max_{-L \leq \nu \leq L}
  \abs{\wh{f}(K+\nu)} CH + \frac{H}{\sqrt{2\pi}} \norm{f}_2 \varepsilon
  ~.
  \]
  If we take $K = K_0 + L$, then \eqref{eq:fvaleps} implies
  \[\int_0^H \abs{\wh{f}(K - i \eta)} \, d \eta \leq \paren{CH + \frac{H}{\sqrt{2
  \pi}} \norm{f}_2} \varepsilon.\]
  Combining this with \eqref{eq:tailtriangle}, we obtain
  \[\abs{\int_{K}^\infty e^{i (\tau - iH) x} g(\tau - iH) \wh{f}(\tau - iH) \, d
    \tau} \leq B \paren{1 + e^H \paren{CH + \frac{H}{\sqrt{2
  \pi}} \norm{f}_2}} \varepsilon,\]
  which gives the result, with $C$ redefined as the expression in
  the outer parentheses.
\end{proof}

We can now state our main result on the truncation of the complex Fourier
representations \eqref{eq:ucift} and \eqref{eq:Vucift}.
\begin{theorem}\label{thm:uhattrunc}
  Let $u$ satisfy \eqref{eq:schrodfree} and the
  assumptions made above on $u_0$,
  $V$, and $A$ for the free space problem. Let $\Gamma$ be as described above with 
  fixed $H > 0$. Let $\varepsilon > 0$, and suppose $K_0$ is sufficiently
  large so that for all $t \in [0,T]$,
  \begin{equation}\label{eq:trunceps}
    \int_{\abs{\xi}>K_0} \abs{\uzhat(\xi)} \, d \xi < \varepsilon,
    \qquad\qquad
    \int_{\abs{\xi}>K_0} \abs{\Vuhat(\xi,t)} \, d \xi < \varepsilon,
  \end{equation}
  and
  \begin{equation}\label{eq:valeps}
    \abs{\uzhat(\xi)} < \varepsilon, \qquad\qquad \abs{\Vuhat(\xi,t)} <
    \varepsilon,
  \end{equation}
  for $\abs{\xi} > K_0$. Then, there are constants $L,C_1,C_2,C_3 > 0$ so that if $K
  = K_0+L$, then
  \begin{equation}\label{eq:Vuhattrunc}
    \abs{\int_{\Gamma \backslash \Gamma_{K}} e^{i \zeta x}
    \Vuhat(\zeta,t) \, d \zeta} < C_1 \varepsilon
  \end{equation}
  and
  \begin{equation}\label{eq:uhattrunc}
    \abs{\int_{\Gamma \backslash \Gamma_{K}} e^{i \zeta x}
    \uhat(\zeta,t) \, d \zeta} < e^{2 H \phimax} (C_2 + C_3 T)
    \varepsilon.
  \end{equation}
  $L$ depends only
  on $H$ and $\varepsilon$, and in particular not on $u_0$ nor on $V$.
  $C_1$ and $C_3$ depend only on $H$ and
  $\max_{0 \leq t \leq T} \norm{(Vu)(\cdot,t)}_2$, and $C_2$ depends
  only on $H$ and $\norm{u_0}_2$.
\end{theorem}
\begin{proof}
  \eqref{eq:Vuhattrunc} follows immediately from Lemma
  \ref{lem:fifttrunc} by taking $f(x) =
  (Vu)(x,t)$ for fixed $t$ and $g = 1$, and then maximizing the
  resulting bound over $t \in [0,T]$. This last step relies on the observation
  from the proof of Lemma \ref{lem:fifttrunc} that,
  if $f(x)$ is replaced by $f(x,t)$ with continuous dependence on $t$,
  then the dependence of the constant $C$ on $t$ is continuous.
  
  To prove
  \eqref{eq:uhattrunc}, we first assume $A = 0$ and use \eqref{eq:uhatzinteq} 
to obtain: 
  \[\int_{\Gamma \backslash \Gamma_{K}} e^{i \zeta x} \uhat(\zeta,t) \, d \zeta =
    \int_{\Gamma \backslash \Gamma_{K}} e^{i \zeta x} e^{-i \zeta^2 t}
      \wh{u}_0(\zeta) \, d \zeta - i \int_0^t \int_{\Gamma \backslash \Gamma_{K}} e^{i \zeta x} e^{-i \zeta^2
  (t-s)} \wh{(Vu)}(\zeta,s) \, d \zeta \, ds.\]
  To bound the first term on the right hand side, we fix $t$ and use Lemma
  \ref{lem:fifttrunc} with $f = u_0$ and $g(\zeta) = e^{-i \zeta^2 t}$,
  which satisfies $|g| \leq 1$ on $\Gamma$.
  We obtain
  \[\abs{\int_{\Gamma \backslash \Gamma_{K}} e^{i \zeta x} e^{-i \zeta^2 t}
  \wh{u}_0(\zeta) \, d \zeta} \leq C_2 \varepsilon\]
  where $C_2$ depends only on $H$ and $\norm{u_0}_2$. To bound the
  second term, we write
  \[\abs{\int_0^t \int_{\Gamma \backslash \Gamma_{K}} e^{i \zeta x} e^{-i \zeta^2
    (t-s)} \wh{(Vu)}(\zeta,s) \, d \zeta \, ds} \leq \int_0^t \abs{\int_{\Gamma \backslash \Gamma_{K}} e^{i \zeta x} e^{-i \zeta^2
    (t-s)} \wh{(Vu)}(\zeta,s) \, d \zeta} \, ds.\]
  Fixing $t$, 
  we may use Lemma \ref{lem:fifttrunc}
  with $f(x) = (Vu)(x,s)$ and $g(\zeta) = e^{-i \zeta^2 (t-s)}$ for 
  each $s$ in the inner integral to obtain
  \[\abs{\int_0^t \int_{\Gamma \backslash \Gamma_{K}} e^{i \zeta x} e^{-i \zeta^2
  (t-s)} \wh{(Vu)}(\zeta,s) \, d \zeta \, ds} \leq C_3 T \varepsilon\]
  where $C_3$ depends only on $H$ and $\max_{0 \leq t \leq T}
  \norm{(Vu)(\cdot,t)}_2$. Here we have performed the same maximization
  over $t$ as before.
  \eqref{eq:uhattrunc} follows for $\phimax = 0$ by combining these estimates in the
  triangle inequality.
  If $A \neq 0$, we again use \eqref{eq:uhatzinteq} and write
  \[\int_{\Gamma \backslash \Gamma_{K}} e^{i \zeta x} \uhat(\zeta,t) \, d \zeta =
    \int_{\Gamma \backslash \Gamma_{K}} e^{i \zeta x} e^{-i \zeta^2
    t + i \zeta \varphi(t)} \wh{u}_0(\zeta) \, d \zeta - i \int_0^t \int_{\Gamma \backslash \Gamma_{K}} e^{i \zeta x} e^{-i \zeta^2
  (t-s) + i \zeta (\varphi(t) - \varphi(s))} \wh{(Vu)}(\zeta,s) \, d \zeta \, ds.\]
  The rest of the argument is almost identical, except that we take $g(\zeta)
  = e^{-i \zeta^2 t + i \zeta \varphi(t)}$ for the first term and $g(\zeta)
  = e^{-i \zeta^2 (t-s) + i \zeta (\varphi(t)-\varphi(s))}$ for the
  second. These both satisfy the bound $|g| \leq e^{2 H \phimax}$. The
  final bounds therefore include this factor.
\end{proof}

We note that it is crucial that $L$ is independent of the data $u_0$ and
$V$ in the proofs above,
since this implies that
at fixed $\varepsilon$ and $H$, $L$ does not grow
with the frequency cutoff $K_0$.
At fixed $\varepsilon$ and $H$, we thus have $K = K_0 + \OO{1}$.
The growth of $L$ as $\varepsilon\to 0$ is weak, since
$\wh{\phi}_\eta$ in the proof of Lemma~\ref{lem:fhatcontinuation}
decays superalgebraically.

\subsection{Analysis of resolution}\label{sec:gammares}

Assuming that the complex Fourier representations \eqref{eq:ucift}
and \eqref{eq:Vucift} have been truncated as
\begin{equation}\label{eq:uciftK}
  u(x,t) \approx \frac{1}{2 \pi} \int_{\Gamma_K} e^{i \zeta x} \uhat(\zeta,t) \, d
\zeta
\end{equation}
and
\begin{equation}\label{eq:VuciftK}
  (Vu)(x,t) \approx \frac{1}{2 \pi} \int_{\Gamma_K} e^{i \zeta x} \Vuhat(\zeta,t)
\, d \zeta,
\end{equation}
we now determine the grid spacing required to resolve the
integrands for all $x \in [-1,1]$ and $t \in [0,T]$. We will provide an
argument analyzing the scaling of the quadrature parameters $\NE$, $q$,
and $\nr$ which demonstrates that 
{\em the number of quadrature nodes required on $\Gamma_K$ is of the order
\[\OO{(1+\phimax) K + \log T} = \OO{(1+\phimax) M + \log T}.\]}
As noted in Remark \ref{rem:Hchoice}, the required
grid spacing will depend on $H$, so we will have to choose this parameter carefully.
We focus on \eqref{eq:uciftK}, since it requires strictly stronger
accuracy constraints than \eqref{eq:VuciftK}.

The integrand may be understood by substituting
\eqref{eq:uhatzinteq} into \eqref{eq:uciftK},
\begin{equation} \label{eq:inteqHconstraint}
  \int_{\Gamma_K} e^{i \zeta x} \uhat(\zeta,t) \, d
\zeta = \int_{\Gamma_K} e^{-i \zeta^2 t + i
\zeta (x + \varphi(t))} \wh{u}_0(\zeta) \, d \zeta - i \int_0^t \int_{\Gamma_K} e^{-i \zeta^2
  (t-s) + i \zeta (x + \varphi(t) - \varphi(s))}
\wh{(Vu)}(\zeta,s) \, ds \, d \zeta,
\end{equation}
and analyzing the integrands of the two resulting terms. We focus on the
second since it requires slightly more stringent parameter choices, but
the analysis is similar for both. We abbreviate the integrand as
$g(\zeta,x,t,s) \Vuhat(\zeta,s)$, with
\[g(\zeta,x,t,s) = e^{-i \zeta^2 (t-s) + i \zeta (x + \varphi(t) -
\varphi(s))}.\]

We first derive a constraint on $H$ by examining the magnitude of
$g(\zeta,x,t,s) \Vuhat(\zeta,s)$. For $\zeta \in \Gamma$, we have
\begin{align*}
  \abs{g(\zeta,x,t,s) \Vuhat(\zeta,s)} &= e^{2 \Re(\zeta) \Im(\zeta) (t-s) - \Im(\zeta)(x +
\varphi(t) - \varphi(s))} \abs{\Vuhat}(\zeta,s) \\
&\leq e^{H (1+2\phimax)} \abs{\Vuhat}(\zeta,s) \leq e^{2 H (1+\phimax)}  \norm{V}_{2,\infty} \\
\end{align*}
where $\norm{V}_{2,\infty} = \max_{t \in
[0,T]} \norm{V(\cdot,t)}_2$. For the first inequality, we used that $\Re(\zeta) \Im(\zeta) \leq 0$ on $\Gamma$ and
$\abs{x} \leq 1$. For the second, we used the estimate
\[\abs{\Vuhat}(\zeta,s) = \abs{\int_{-1}^1 e^{-i \zeta x} (Vu)(x,s) \,
  dx} \leq e^H \int_{-1}^1 \abs{(Vu)(x,s)} \, dx \leq e^H \max_{t \in
[0,T]} \norm{V(\cdot,t)}_2\]
for $\zeta \in \Gamma$, which follows from the Cauchy-Schwarz inequality and the
fact that $\norm{u(\cdot,t)}_2 = \norm{u_0}_2 = 1$.
A large choice of $H$ may therefore lead to a loss of accuracy in
floating point arithmetic due to large-magnitude oscillations of the
integrand in \eqref{eq:uciftK}. To maintain a relative accuracy
$\varepsilon$, we require
\[e^{2 H (1+\phimax)}  \norm{V}_{2,\infty} \leq \varepsilon/\epsm,\]
where $\epsm$ is the machine epsilon. This implies the constraint
\[H \leq \frac{\log\paren{\varepsilon/\paren{\norm{V}_{2,\infty}
\epsm}}}{2(1 +
\phimax)}.\]
For dimension $d$, a similar argument gives
\[H \leq \frac{\log\paren{\varepsilon/\paren{\norm{V}_{2,\infty}
  \epsm}}}{2d(1 + \phimax)}\]
in each dimension. If $V = 0$, then we must analyze the first integral on 
the right hand side of
\eqref{eq:inteqHconstraint}, from which we obtain a similar but slightly weaker
constraint. The inequality
\[H \leq \frac{\log\paren{\varepsilon/\paren{\paren{1+\norm{V}_{2,\infty}}
  \epsm}}}{2d(1 + \phimax)}\]
covers both cases.

$\Vuhat(\zeta,s)$ is well-resolved by a grid with $\OO{1}$ spacing on
$\Gamma$, so we focus on the behavior of $g(\zeta,x,t,s)$. On $\Gamma_3$, we have $\zeta = \tau - iH$ with $\tau \in
[H,K]$, so
\begin{align*}
\begin{aligned}
g(\gamma(\tau),x,t,s) &= e^{-i (\tau - iH)^2 (t-s) + i (\tau - iH) (x +
  \varphi(t) - \varphi(s))} \\
&= e^{-i (\tau^2 - H^2) (t-s) + i \tau (x +
  \varphi(t) - \varphi(s))}
e^{- 2 \tau H (t-s) + H (x + \varphi(t) - \varphi(s))}.
\end{aligned}
\end{align*}
This function decays exponentially in $\tau$, and to achieve an 
accuracy of $\varepsilon$ in integration, we must resolve the oscillatory factor only for $\tau \in
\brak{H,\min\paren{K,\frac{\log(1/\epsm)}{2 H (t-s)}}}$.
For $t-s \leq
\frac{\log(1/\epsm)}{2 H K}$, this becomes $\tau \in
[H,K]$. We can estimate the required grid spacing by computing the
magnitude of the derivative of the
oscillatory factor:
\begin{align*}
\begin{aligned}
\abs{\frac{d}{d\tau} e^{-i (\tau^2 - H^2) (t-s) + i \tau (x +
  \varphi(t) - \varphi(s))}}
  &= \abs{2 \tau (t-s) - (x + \varphi(t) - \varphi(s))} \leq 2K (t-s) + 1
  + 2 \phimax \\
&\leq
\frac{\log(1/\epsm)}{H} + 1 + 2 \phimax.
\end{aligned}
\end{align*}
For $t-s > \frac{\log(1/\epsm)}{2 H K}$, we have $\tau \in
\brak{H,\frac{\log(1/\epsm)}{2 H (t-s)}}$, and obtain the same
estimate:
\[\abs{\frac{d}{d\tau} e^{-i (\tau^2 - H^2) (t-s) + i \tau (x +
  \varphi(t) - \varphi(s))}} = \abs{2 \tau (t-s) - (x + \varphi(t) -
  \varphi(s))} \leq
\frac{\log(1/\epsm)}{H} + 1 + 2 \phimax.\]
The grid spacing required to achieve minimal resolution may be estimated as the reciprocal of this
value. This suggests taking $H$ to be as large as possible, within the
constraints imposed by our floating point accuracy considerations, in order to obtain a
coarsest possible grid spacing. Thus, we set
\[H = \frac{\log\paren{\varepsilon/\paren{\paren{1+\norm{V}_{2,\infty}}
  \epsm}}}{2d(1 + \phimax)}.\]
Our estimate of the required grid spacing is then
\[\Delta \tau = \paren{2 (1 + \phimax) \brak{
      \frac{d
      \log(1/\epsm)}{\log\paren{\varepsilon/\paren{\paren{1+\norm{V}_{2,\infty}}
  \epsm}}} + 1} - 1}^{-1} =
\OO{(1+\phimax)^{-1}},\]
which notably does not scale with $K$. Taking uniformly spaced
nodes, we obtain $\NE = \OO{(1 + \phimax) K}$ points on $\Gamma_3$. The
analysis for $\Gamma_1$ is nearly identical.

On $\Gamma_2$, we have $\zeta = \tau - i \tau$ with $\tau \in [-H,H]$,
so
\[g(\gamma(\tau);x,t,s) = e^{-i \tau^2 (1-i)^2 (t-s) + i (1-i) \tau (x +
  \varphi(t) - \varphi(s))} = e^{- 2 \tau^2 (t-s)} e^{(1+i) \tau (x +
\varphi(t) - \varphi(s))}.\]
Since $\abs{\tau (x + \varphi(t) - \varphi(s))} \leq H (1 + 2 \phimax)
\leq \log\paren{\varepsilon/\paren{\paren{1+\norm{V}_{2,\infty}}
  \epsm}}/d$ for $\tau \in [-H,H]$, the second factor
may be resolved by a grid with spacing independent of $K$, $\phimax$,
and $T$. The first factor is a Gaussian of width $\frac{1}{2 \sqrt{t-s}}$,
and may be resolved for all $s
\in [0,T]$ by a composite Gauss quadrature rule with $\nr = \OO{\log T}$ panels
of uniform order $q$, dyadically refined toward the origin.

\section{Numerical results} \label{sec:results}

We illustrate the performance of the periodic and free space methods on 
a collection of model problems. In addition, for the free space method, we
carry out several experiments which demonstrate the convergence behavior of the
quadrature rule on $\Gamma$ with respect to the relevant quadrature
parameters. All codes were written in MATLAB, which invokes the FFTW library
\cite{frigo05}. Experiments were performed on a laptop with an
Intel Xeon E-2176M 2.70GHz processor.

We define the time-dependent $L^2$ error over the computational domain,
measured against a reference solution $u_{\text{ref}}$, as
\begin{equation} \label{eq:l2err}
  \err(t) = \sqrt{\int_{-L}^L \abs{u(x,t)-u_{\text{ref}}(x,t)}^2 \, dx},
\end{equation}
and the maximum $L^2$ error as
\begin{equation} \label{eq:maxl2err}
  \err_{\max} = \max_{t \in [0,T]} \err(t).
\end{equation}
Here $L = \pi$ for the periodic case and $L = 1$ for the free space
case. The reference solution $u_{\text{ref}}$ will be specified in each
experiment. We approximate 
\eqref{eq:l2err} using the left endpoint rule on the
computational grid.

We define a pulse vector potential $A(t)$, 
given in one dimension by 
\begin{equation} \label{eq:apulse}
  A(t) = A_0 \sin^2(t \pi/T) \cos(\omega t),
\end{equation}
where $A_0$ is an amplitude parameter and $\omega$ is a frequency
parameter.
In two dimensions, we will take $A(t) = (A_1(t),0)^T,$ where $A_1$ has the
form \eqref{eq:apulse}. This form of the vector potential will be used
in several of our experiments.

\subsection{Example 1: moving periodic Gaussian well potential in 1D}

Our first numerical example takes $V(x,t)$ to be 
the periodic extension of a one-dimensional Gaussian
well moving with constant speed $c$:
\[V(x,t) = \sum_{k=-\infty}^\infty -V_0 e^{-\frac{(x - 2 \pi k - ct)^2}{2 \beta^2}}.\]
We take $V_0 = 300$ and $\beta = 0.2$. For simplicity, we set $A = 0$,
and take $u_0$ to be the $L^2$-normalized ground state of the time-independent
\Schrod equation with potential $V$, computed to approximately 11 digits
of accuracy using the \texttt{eigs}
function of the Chebfun software package \cite{driscoll08}. 
The ground state eigenvalue is approximately $-243$. We use three
different values of the speed, $c = 15$, $c = 30$, and $c
= 45$, and a final time $T = 2 \pi / 15$. Plots of the three solutions are given in Figure
\ref{fig:persolnplots}. At the slowest speed, the solution remains
largely bound by the potential, although it oscillates somewhat within the
potential well. For the fastest speed, most of the mass of the
wavefunction falls out of the well, and quickly spreads
out over the domain.

\begin{figure}[t]
  \centering
    \includegraphics[width=\linewidth]{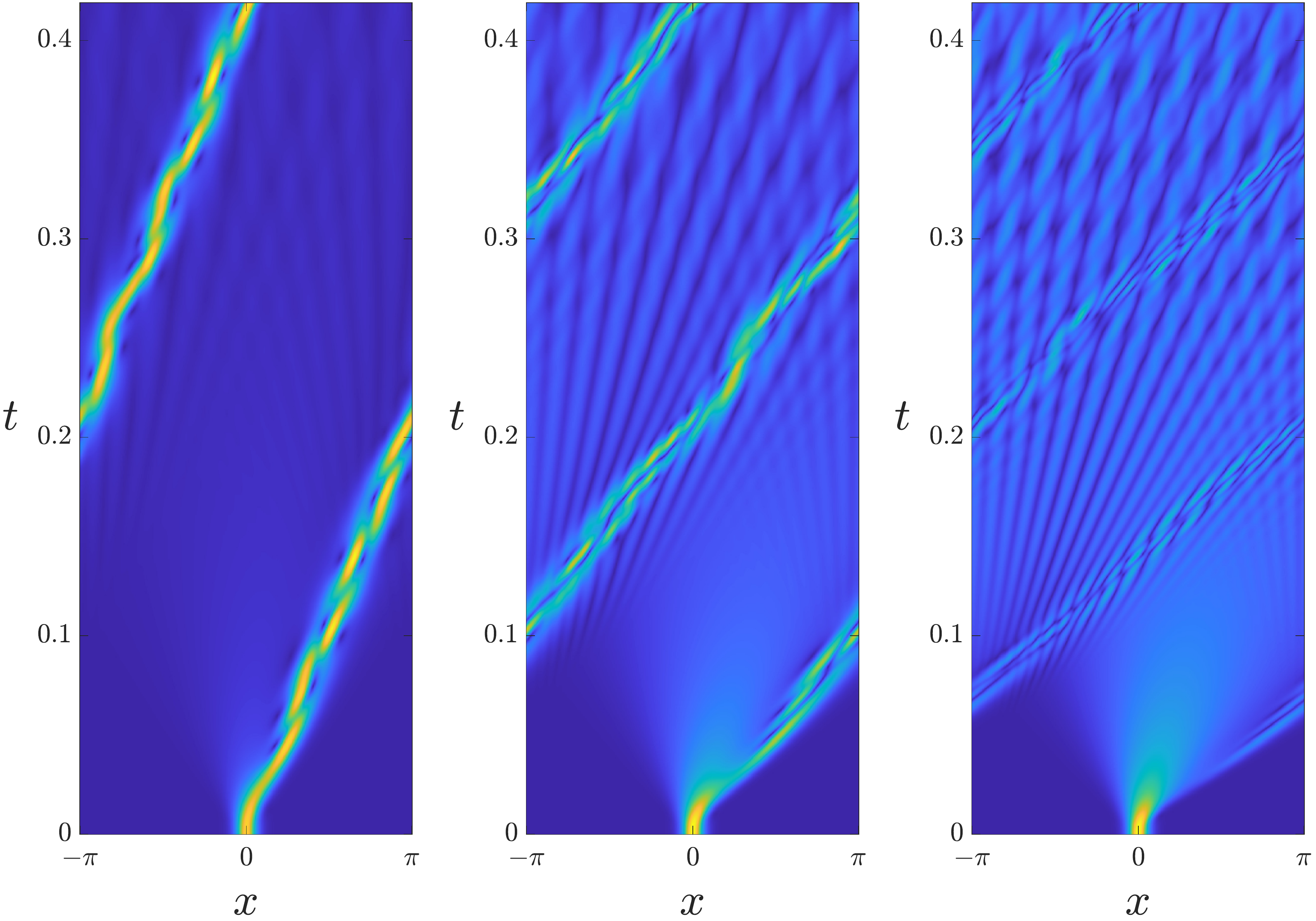}
    \caption{In Example 1, a periodic Gaussian well potential moves with
    constant speed $c$, carrying along a solution $u$ which is initialized
    in the ground state of the stationary potential. Plots of
    $\abs{u(x,t)}$ are given in the unit cell $[-\pi,\pi]$ for $c = 15$
    (left), $c = 30$ (middle), and $c = 45$ (right).}
    \label{fig:persolnplots}
\end{figure}

We solve the equations for various choices of $M$ and values of $\Delta
t$ corresponding to 200, 400, 800, \ldots, 25600 time steps, using the
eighth-order version of the
implicit multistep scheme described in
Section \ref{sec:highorder}. We measure the final time errors $\err(T)$
using a reference solution $u_{\text{ref}}$ obtained by increasing 
$M$ and decreasing $\Delta t$ to self-consistent convergence beyond 12
digits of accuracy. The results
are presented in Figure \ref{fig:pererrvdt}. We
observe the expected eighth-order convergence with respect to $\Delta t$
and spectral convergence with respect to $M$. 

\begin{figure}[t]
  \centering
    \includegraphics[width=\linewidth]{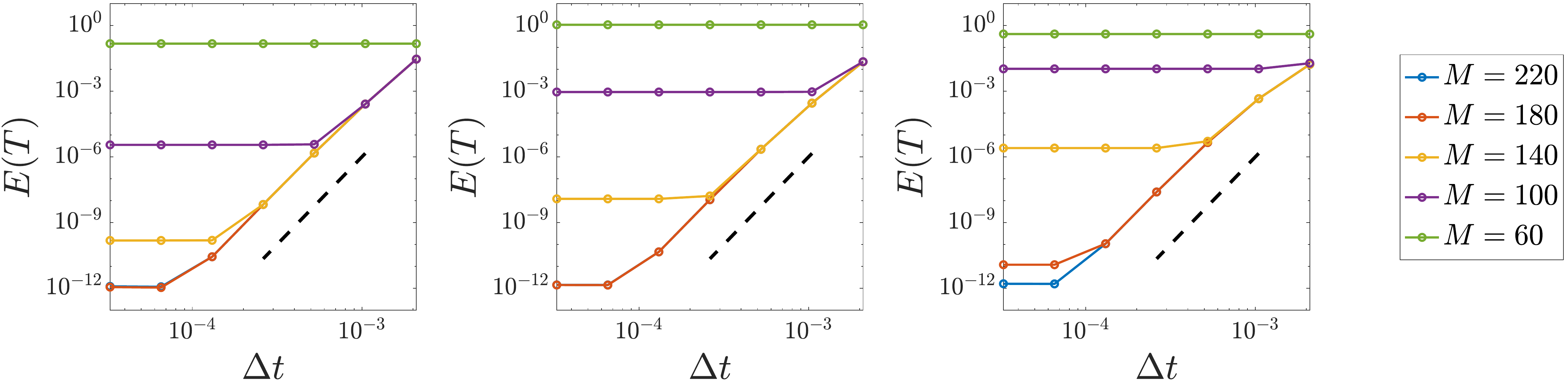}
    \caption{Final time $L^2$ error of $u(x,t)$ against $\Delta t$ for several values of
      $M$ and $c =$ $15$ (left), $30$ (middle), and $45$ (right) for
      Example 1. Eighth-order convergence is indicated by the black
  dashed lines.}
    \label{fig:pererrvdt}
\end{figure}

\subsection{Example 2: convergence of the $\Gamma$ quadrature}

For the free space problem, the accuracy parameters at our disposal in the
one-dimensional case are:
\begin{itemize}
  \item the numerical tolerance $\varepsilon$
  \item the number $M$ of grid points on $[-1,1]$
  \item the Alpert quadrature order parameter $p$
  \item the number $\NE$ of equispaced points in the Alpert quadrature,
    which sets the regular grid spacing $h$
  \item the Gaussian quadrature order parameter $q$
  \item the dyadic refinement depth $\nr$
\end{itemize}
In the $d$-dimensional case, except for $\varepsilon$, there is one such parameter for each
dimension. We fix $p = 8$ in every dimension, so that the Alpert quadrature rule is
$16$th-order accurate. $K = \frac{\pi}{2} M + H$ and $H =
\frac{\log\paren{\varepsilon/\paren{\paren{1+\norm{V}_{2,\infty}}
\textbf{u}}}}{2d(1 + \phimax)}$ are also fixed in every dimension.

We examine the convergence of the quadrature on $\Gamma$ with
respect to $M$, $\NE$, and $\nr$. We demonstrate numerically the claim
that a fixed accuracy is achieved by taking $\NE = \OO{M (1 + \phimax)}$
and $\nr = \OO{\log T}$. For all experiments we fix $\varepsilon = 10^{-14}$. Since
the $d$-dimensional quadratures are tensor products of the
one-dimensional quadratures, it is sufficient to work in one dimension.

We test the following Gaussian wavepacket solution of
\eqref{eq:schrodfree} for $d = 1$, $V = 0$, and $A = 0$:
\begin{equation} \label{eq:wavepacket}
  u_{\text{wp}}(x,t) = \frac{\sigma \sqrt{\sigma}}{\pi^{1/4}\sqrt{\sigma^2 + 2 i t}}
   \exp\paren{-\frac{\paren{x/\sqrt{2} - i \sigma k_0 /
     2}^2}{\sigma^2 + 2 i t } - k_0^2/4}.
\end{equation}
Here $\sigma$ is a width parameter and $k_0$ is a frequency parameter. We fix $k_0 = 0$ for all the
experiments in this section.

When $V = 0$, our
method simply amounts to applying the propagator in the frequency 
domain to complex-frequency modes and then transforming back to physical
space. In particular, there is no time discretization error, only
truncation and quadrature errors. We can therefore measure these errors with respect to 
the various quadrature
parameters by taking $V = 0$, $u_0 = u_{\text{wp}}(x,0)$, and computing the maximum $L^2$ error \eqref{eq:maxl2err} with $u_{\text{ref}} =
u_{\text{wp}}$. In all experiments, each quadrature parameter aside from
the one being varied is refined until convergence to about fifteen digits of accuracy.

The truncation error is determined by $M$, which sets the
truncation radius on $\Gamma$ according to the formula $K =
\frac{\pi}{2} M + H$. The quadrature error is determined by $q$, $\nr$,
and $h$, the last of which
is related to $\NE$ by the formula
\[h = \frac{K-H}{\NE + 2 \kappa - 1} = \frac{\pi M}{2(\NE + 13)}.\]
Here we have used that $\kappa = 7$ for $p = 8$.

In addition to showing typical convergence rates with respect to $M$ and
$h$, our first two experiments show that, consistent with our analysis, the quiver radius $\phimax$
does not significantly affect the choice of $M$ required to achieve a
given error, but does affect $h$ approximately as $h \sim 1/(1 +
\phimax)$. We fix $T = 0.1$ and $\sigma = 0.1$ in
\eqref{eq:wavepacket}. We take $A(t)$ given by
\eqref{eq:apulse} with $\omega = 500$, yielding pulses of a few cycles, and use four different field
amplitudes: $A_0 = 0$, $500$, $1500$, and $3500$. These correspond to
the quiver radii $\phimax = 0$, $\approx 1$, $\approx 3$, and $\approx
5$, respectively. Figure \ref{fig:Mconva} shows $\err_{\max}$ as $M$ is
varied for each choice of $A_0$. The
convergence of the quadrature with respect to $M$ is superexponential,
as expected since $\uhat$ is entire. The truncation radius required to
achieve a given error is not significantly affected by $A_0$.
Next, Figure \ref{fig:hconva} shows $\err_{\max}$ as $h$ is varied for
each choice of $A_0$.
The convergence with respect to $h$ is approximately
$16^{\text{th}}$-order. Furthermore, as $1+\phimax$ doubles from $2$
to $4$ and from $4$ to $8$, the grid spacing required to achieve a given error
approximately halves, consistent with the expectation $h \sim 1/(1 +
\phimax)$.

\begin{figure}[t]
  \centering
  \begin{subfigure}[t]{0.47\textwidth}
    \includegraphics[width=\textwidth]{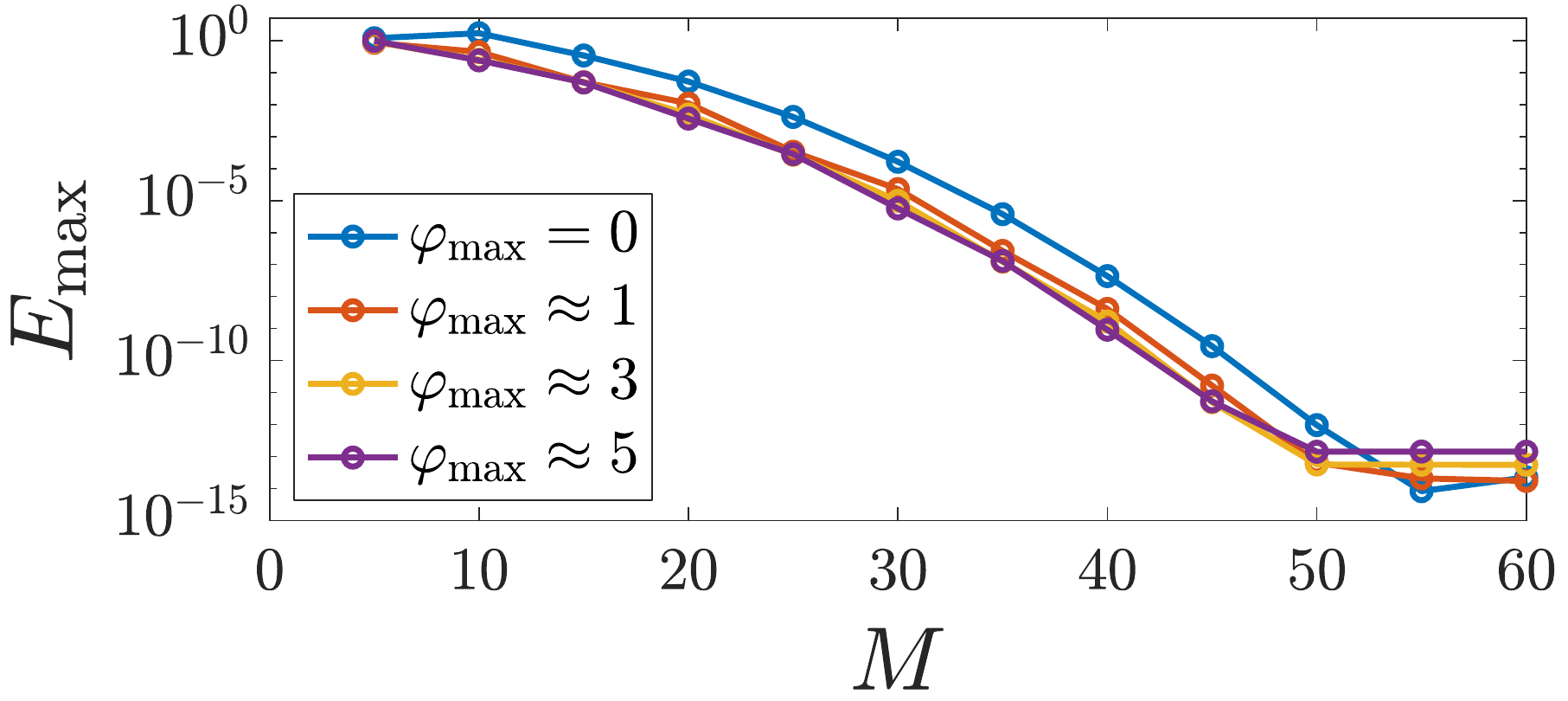}
    \caption{Convergence with respect to $M$ for different field
    strengths}
    \label{fig:Mconva}
  \end{subfigure}
  \hspace{0.5cm}
  \begin{subfigure}[t]{0.47\textwidth}
    \includegraphics[width=\textwidth]{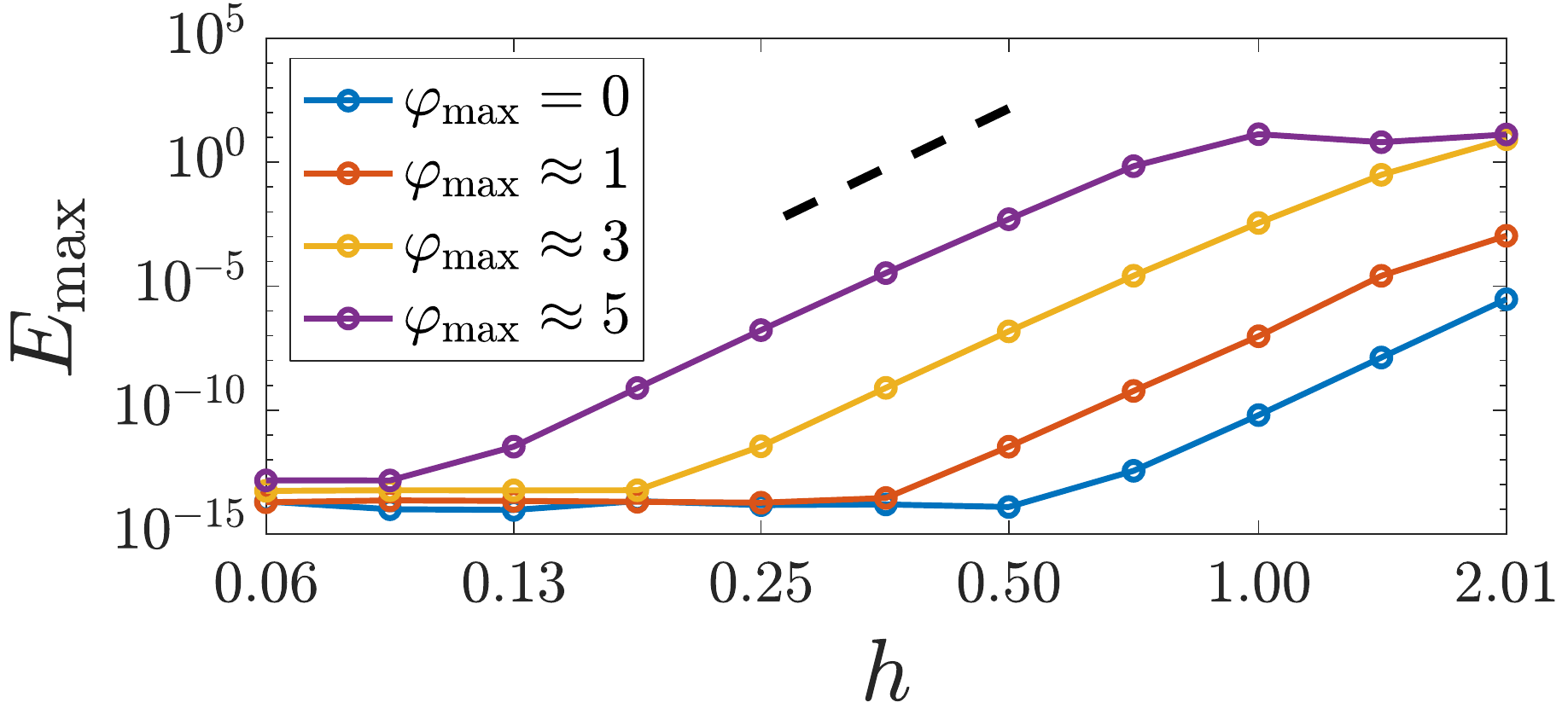}
    \caption{Convergence with respect to $h$ for different field
      strengths}
    \label{fig:hconva}
  \end{subfigure}

  \vspace{0.3cm}
  
  \begin{subfigure}[t]{0.47\textwidth}
    \includegraphics[width=\textwidth]{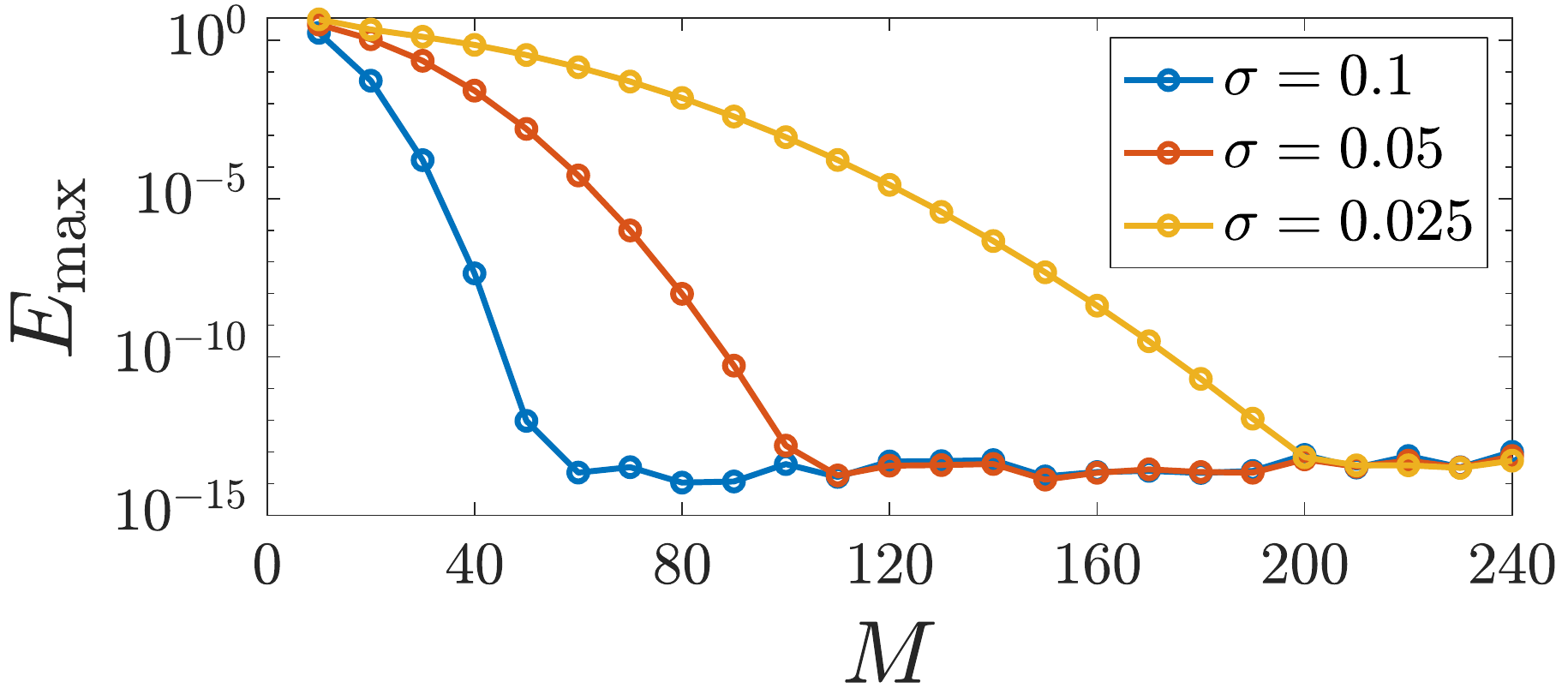}
    \caption{Convergence with respect to $M$ for different initial
    conditions}
    \label{fig:mconvsig}
  \end{subfigure}
  \hspace{0.5cm}
  \begin{subfigure}[t]{0.47\textwidth}
    \includegraphics[width=\textwidth]{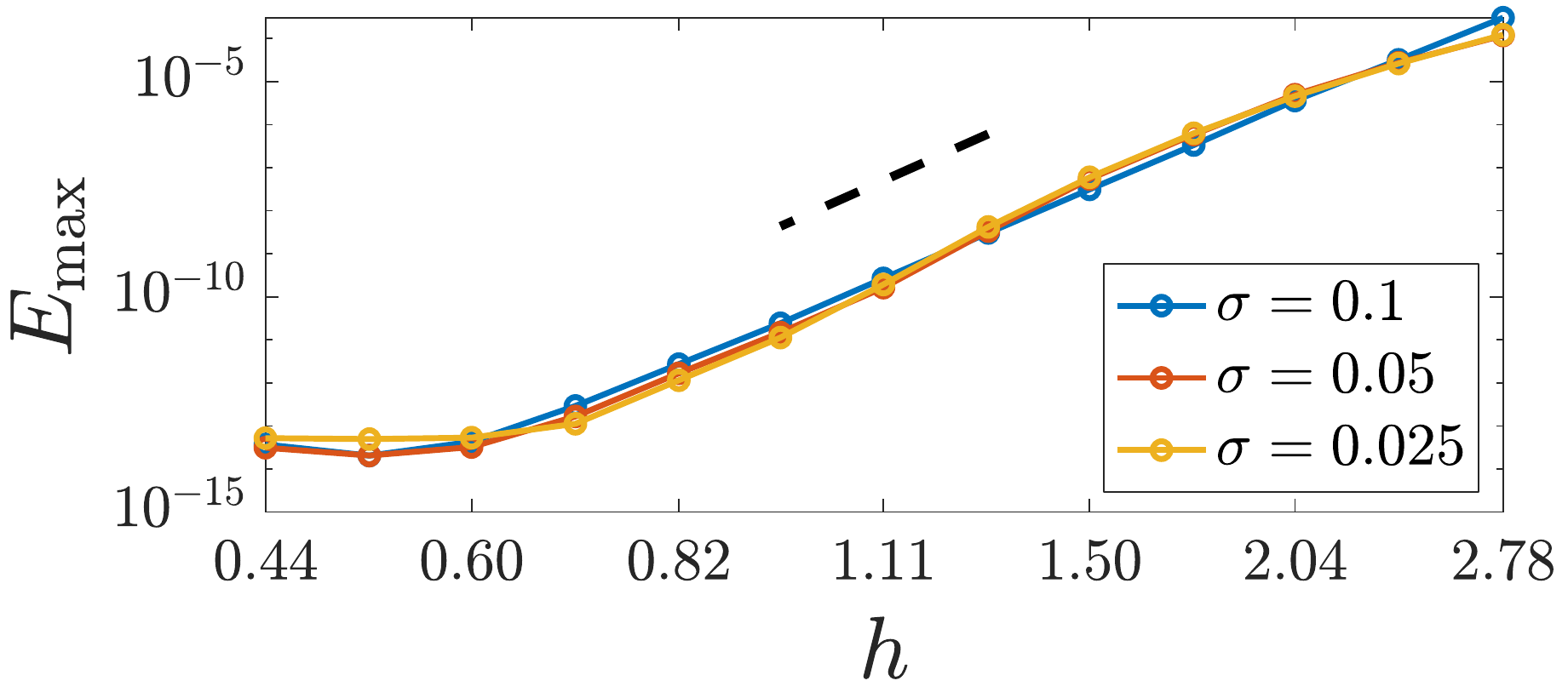}
    \caption{Convergence with respect to $h$ for different initial
    conditions}
    \label{fig:hconvsig}
  \end{subfigure}

    \caption{Convergence of the $\Gamma$ truncation and quadrature error with respect
      to $M$ and $h$, respectively, for the Gaussian wavepacket
      solution of Example 2. $M$ scales with the
    frequency cutoff of the solution but not with the field strength,
  and $h$ scales with the field strength but not with the frequency
  cutoff. For (b) and (d), the black dashed line indicates
$16^{\text{th}}$-order convergence.}
    \label{fig:convergence}
\end{figure}

In the next two experiments, we let $A = 0$, and adjust the numerical
support of the solution in the frequency
domain by taking three different values of $\sigma$:
$\sigma = 0.1$, $0.05$, and $0.025$. We expect that this should not
significantly affect the regular grid spacing $h$ required to achieve a
given error, but it should affect $M$.
Figure \ref{fig:mconvsig} shows $\err_{\max}$ as $M$ is varied for each
choice of $\sigma$. When
$\sigma$ is halved the numerical support of the solution in the frequency domain
increases by a factor of two, so a given error is maintained by
approximately doubling $M$. Figure \ref{fig:hconvsig} shows that the
choice of $h$ required to achieve a given error $\err_{\max}$ is
insensitive to $\sigma$.

In the final convergence experiment, we examine the error of a
long-time simulation as $\nr$ is increased. We take $T = 1000$, $\sigma = 0.1$, and
$A_0 = \omega = 1$, yielding a pulse of many cycles over the large time
interval. We fix $q = 16$ and plot $\err(t)$ for $\nr = 1,2,3,4,5$ on a log-log scale. The results are
shown in Figure \ref{fig:longtime}. As expected, for any fixed choice
of $\nr$, at some point in time the quadrature begins to lose accuracy. However,
incrementing $\nr$ increases this time by a fixed order of magnitude,
so that the scaling $\nr = \OO{\log T}$ preserves a uniform accuracy.

\begin{figure}[t]
  \centering
    \includegraphics[width=\linewidth]{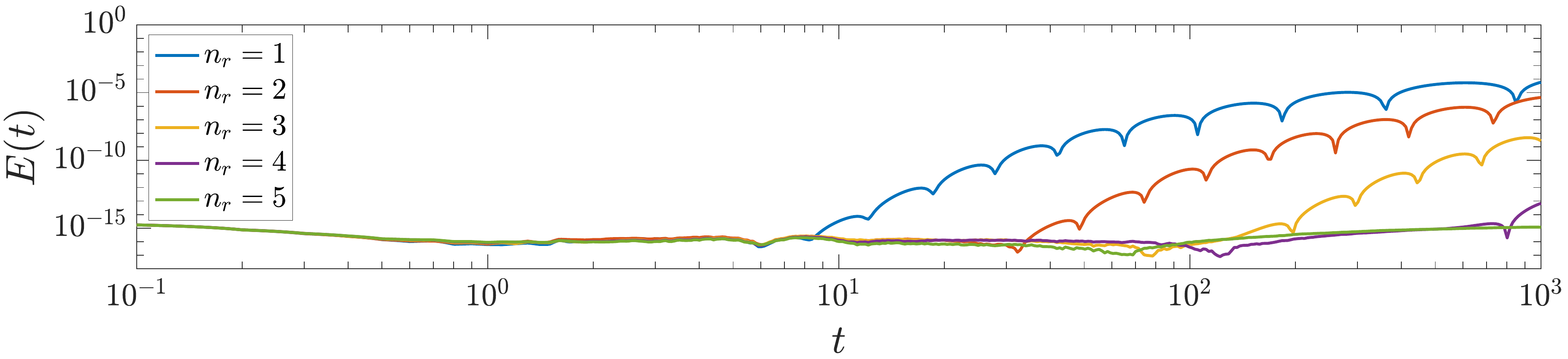}
    \caption{The $\Gamma$ quadrature error over time as $\nr$ is varied, 
    for the Gaussian wavepacket solution of Example 2.
      Incrementing $\nr$ preserves a given quadrature accuracy for an
    additional fixed order of magnitude of time.}
    \label{fig:longtime}
\end{figure}

\subsection{Example 3: ionization from a Gaussian well in 1D}

In our next example, we take the scalar potential to be a
Gaussian well,
\[V(x) = -V_0 e^{-\frac{x^2}{2 \beta^2}},\]
$u_0$ to be the $L^2$-normalized ground state of the time-independent
\Schrod equation with potential $V$, and $A$ to be a pulse
\eqref{eq:apulse}. We set $V_0 = -1400$ and $\beta = 0.1$. The
ground state $u_0$ is again computed using Chebfun's \texttt{eigs}
routine. 
The ground state eigenvalue is approximately $-1154$. $V$ is less
than $10^{-18}$ and $u_0$ is less than
$10^{-12}$ outside $[-1,1]$. We take $T =
0.5$, $A_0 = 100$ and $\omega = 50$, $100,$ and $200$, yielding quiver radii of 
$\phimax \approx 2$, $1$, and $1/2$, respectively. Plots of the three solutions and the
corresponding fields $A(t)$ are given in Figure \ref{fig:1dsolnanda}.

\begin{figure}[!htb]
  \centering

    \includegraphics[width=\linewidth]{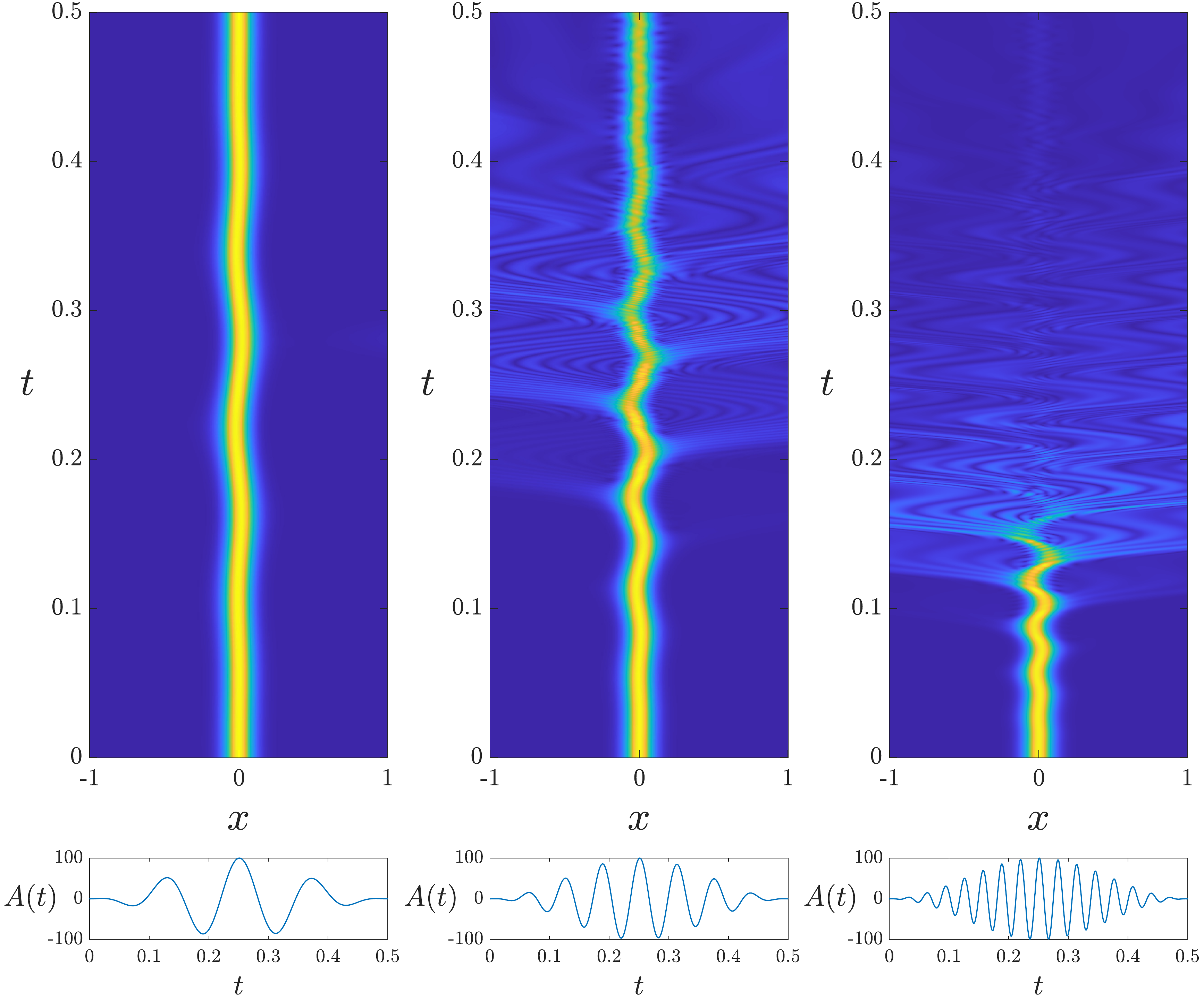}

    \caption{In Example 3, a solution $u$ initialized in the ground
    state of a Gaussian well potential is perturbed by an applied field
    $A(t)$. Plots are given of $\abs{u(x,t)}$ (above) and the corresponding potential
      $A(t)$ (below), with $\omega = 50$
    (left), $\omega = 100$ (middle), and $\omega = 200$ (right). The ionization fractions, estimated as
  $1 - \int_{-1}^1 \abs{u(x,T)}^2$, are approximately $0\%$, $40.72\%$, and
  $99.86\%$, for $\omega = 50$, $100$, and $200$, respectively.}
    \label{fig:1dsolnanda}
\end{figure}

We use the eighth-order version of the implicit multistep scheme described in
Remark \ref{rem:highorderfs}, with several approximately
logarithmically-spaced values of $h$, and values of 
$\Delta t$ corresponding to 1000, 2000, 4000, \ldots, 64000 time steps.
$\varepsilon = 10^{-10}$, $M = 100$, $q = 10$, and $\nr = 0$ are fixed. 
In the complex-frequency Fourier
transform algorithm, the $\calC$-type transforms are computed by direct
matrix multiplication rather than the Chebyshev interpolation scheme,
since the latter does not offer a speed improvement for small $\nr$. The final-time errors $\err(T)$ are plotted against
$\Delta t$ in Figure \ref{fig:1derrvdt}. The reference solution is
obtained by converging the solver to high accuracy with respect to
all parameters. We observe the expected eighth-order convergence with
$\Delta t$, and that the value of $h$ required to achieve a given accuracy
decreases as $\phimax$ increases. Timings associated with
these experiments for each choice of $h \sim 1/\NE$
are given in Table \ref{tab:1dtimes}. The scaling with $\NE$ appears to
be sublinear for these values, but this is simply because the asymptotic
regime has not yet been reached with the relatively small FFT sizes.

\begin{figure}[!htb]
  \centering

    \includegraphics[width=\linewidth]{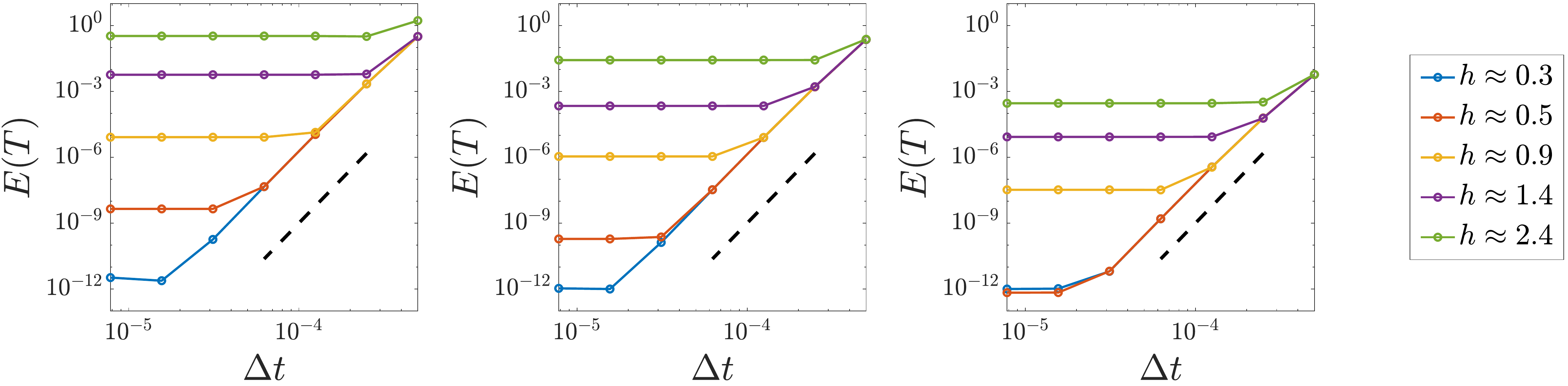}

    \caption{Final time $L^2$ error of $u(x,t)$ against $\Delta t$ for several values of
      $h$ and $\omega =$ $50$ (left), $100$ (middle), and $200$ (right) in Example 3. Eighth-order convergence is indicated by the black
  dashed lines. The minimum achievable error decreases with $h$, and the
value of $h$ required to achieve a given error decreases as $\phimax$
increases.}
    \label{fig:1derrvdt}
\end{figure}

\begin{table}
  \centering
  \begin{tabular}{|c|r|r|r|r|r|}
    \hline
    $h \approx$ & 0.3 & 0.5 & 0.9 & 1.4 & 2.4 \\ \hline
    Time steps per second & $3063$ & $3789$ & $4872$ & $6759$ & $8441$ \\
    \hline
  \end{tabular}
  \caption{Number of time steps per second for the experiments in Example 3.}
  \label{tab:1dtimes}
\end{table}

\subsection{Example 4: ionization from a Gaussian well in 2D}

We next examine the two-dimensional analogue of the previous example. We
use the scalar potential
\[V(x,y) = -V_0 e^{-\frac{x^2+y^2}{2 \beta^2}}\]
with $V_0 = 1400$ and $\beta = 0.1$, and again take $u_0$ to
be the normalized ground state of the corresponding time-independent
\Schrod equation. The ground state may be computed by working in polar
coordinates and solving the resulting one-dimensional eigenvalue problem
using Chebfun's \texttt{eigs} routine. It
is less than $10^{-11}$ outside of $[-1,1]^2$. The ground state eigenvalue is approximately $-922$. We take $T = 0.5$ as before,
and $A(t) = (A_1(t),0)^T$ with $A_1(t)$ as in the previous experiment
with the same choices of $A_0$ and $\omega$. Plots of the
solution with $\omega = 100$ at various time steps are given in Figure \ref{fig:2dsoln}.

\begin{figure}[!htb]
  \centering
  \begin{subfigure}[t]{0.45\textwidth}
    \includegraphics[width=\textwidth]{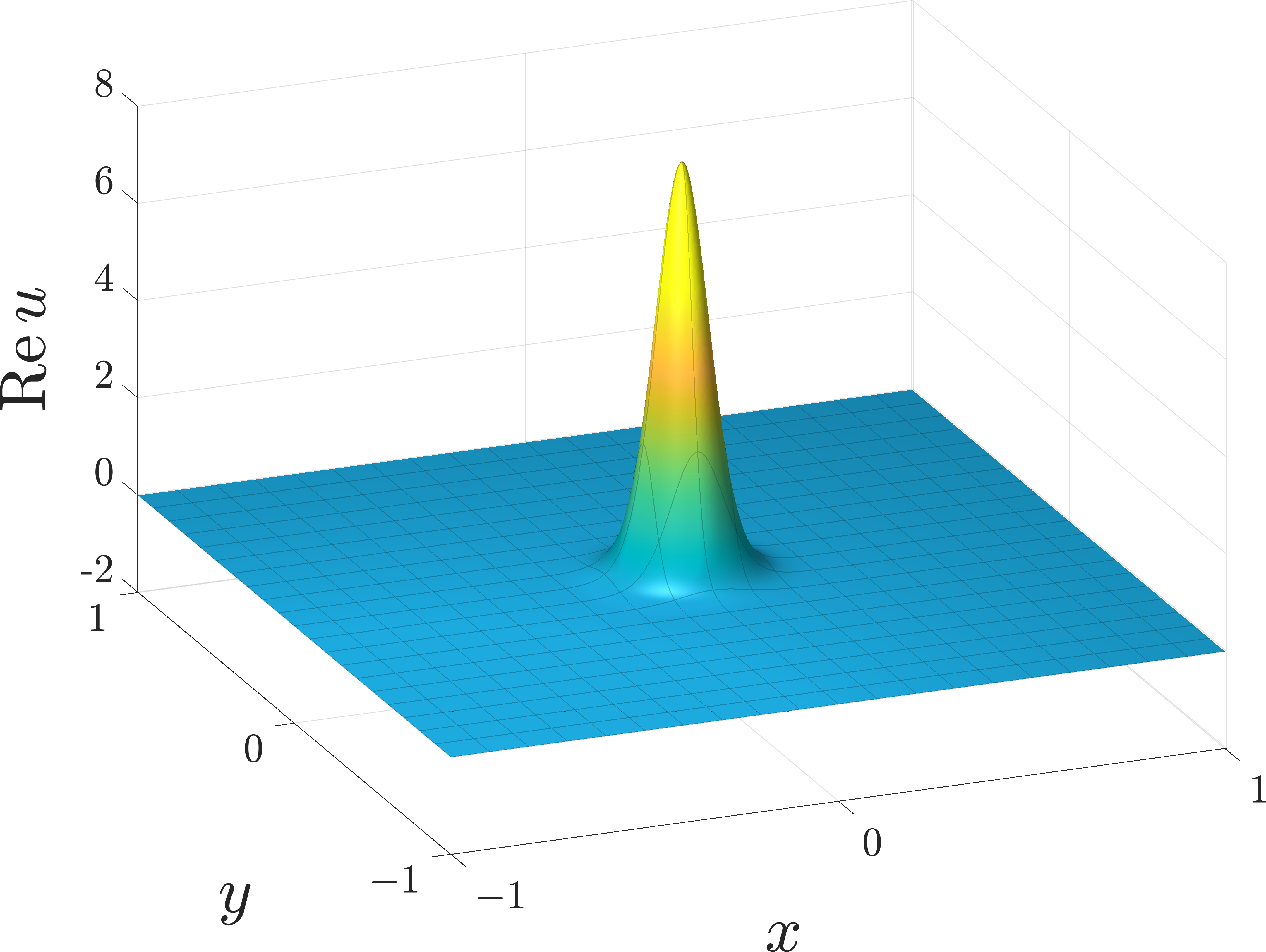}
    \caption{$t = 0$: ground state of the Gaussian potential}
    \label{fig:soln2d1}
  \end{subfigure}
  \hspace{0.5cm}
  \begin{subfigure}[t]{0.45\textwidth}
    \includegraphics[width=\textwidth]{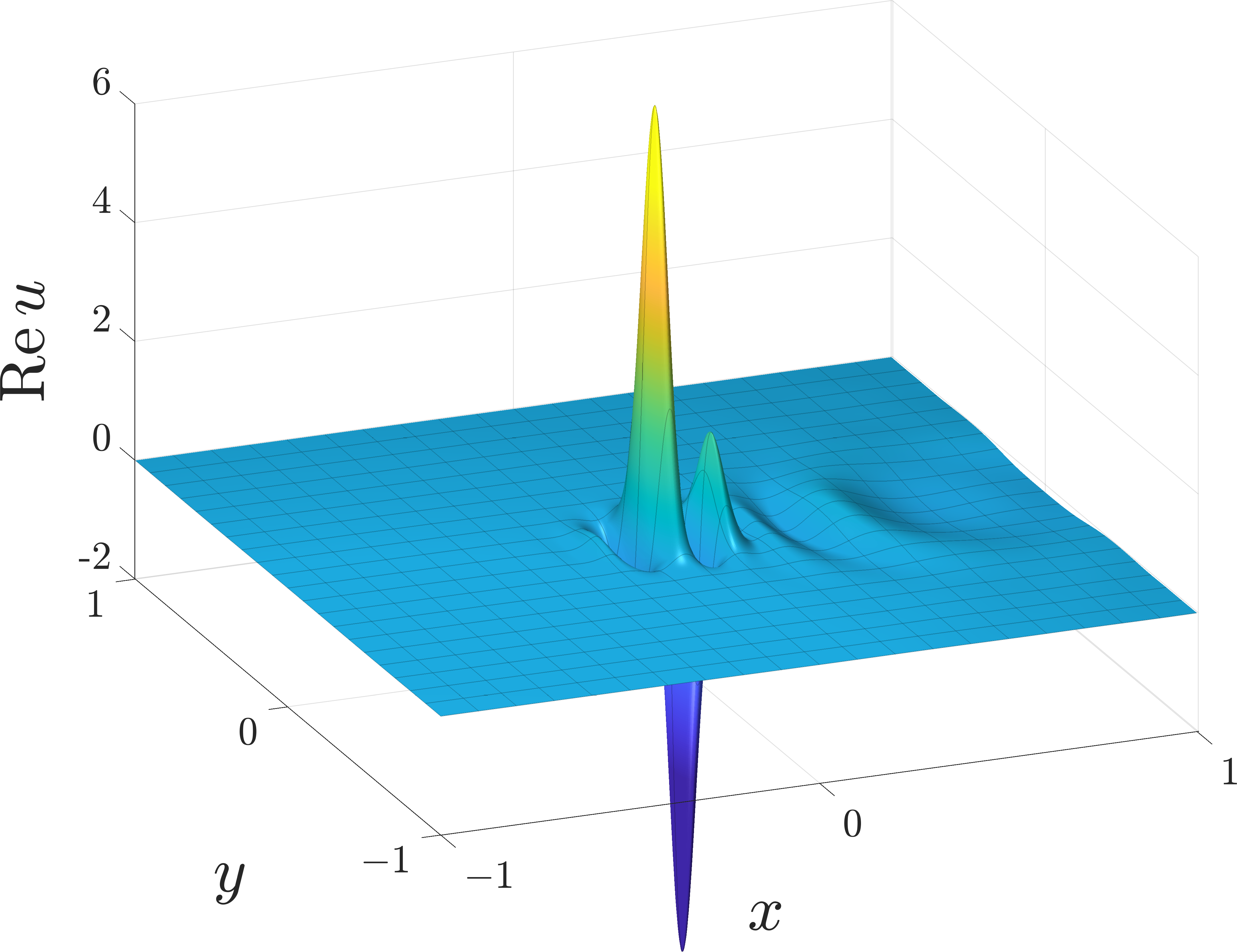}
    \caption{$t \approx 0.16$: $u$ is pushed to the right by the applied
    field}
    \label{fig:soln2d2}
  \end{subfigure}

  \vspace{0.3cm}
  
  \begin{subfigure}[t]{0.45\textwidth}
    \includegraphics[width=\textwidth]{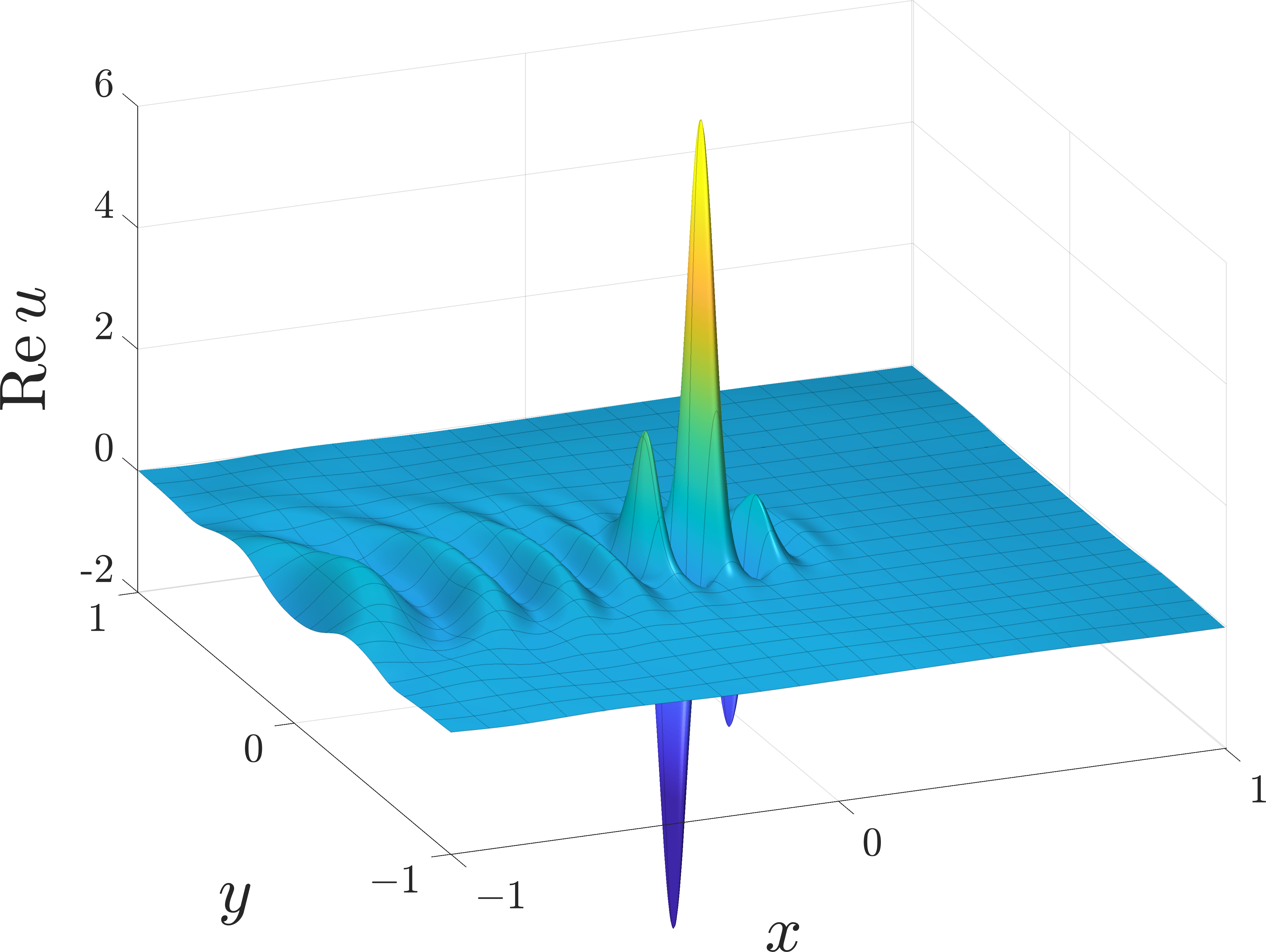}
    \caption{$t \approx 0.19$: $u$ is pushed to the left by the applied
    field}
    \label{fig:soln2d3}
  \end{subfigure}
  \hspace{0.5cm}
  \begin{subfigure}[t]{0.45\textwidth}
    \includegraphics[width=\textwidth]{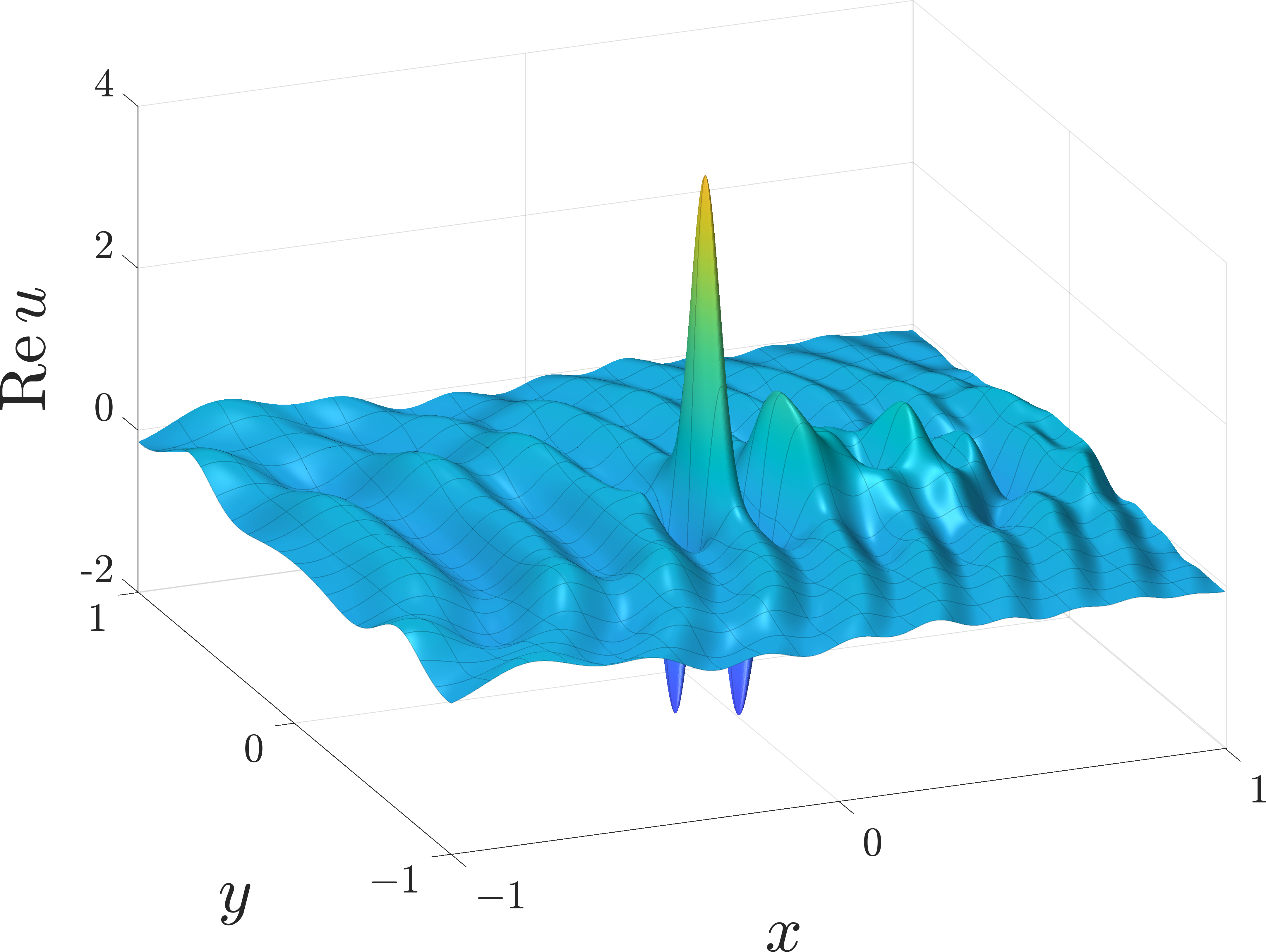}
    \caption{$t \approx 0.27$: $u$ is again pushed to the right and is dispersed throughout the domain}
    \label{fig:soln2d4}
  \end{subfigure}

    \caption{Example 4 is the two-dimensional analogue of Example 3,
    with the applied field $A(t)$ aligned with the $x$ axis. Plots are given of $\Re u(x,t)$ with $\omega =
    100$ at four time steps. The form of $A_1(t)$ is shown in the middle panel of
    Figure \ref{fig:1dsolnanda}.}
    \label{fig:2dsoln}
\end{figure}

We again use the eighth-order implicit multistep scheme and fix $M =
100$, $q = 10$, and $\nr = 0$. As in Example 3, in the complex-frequency
Fourier transform algorithm, $\calC$-type transforms and transforms
involving $\calC$-type nodes are applied using direct multiplication rather
than the Chebyshev interpolation scheme, since $\nr = 0$. We carry out
higher accuracy calculations with $\varepsilon = 10^{-10}$, $h_2 \approx
0.5$, and values of $\Delta t$ corresponding to 1000, 2000, 4000, \ldots 32000
time steps, and lower accuracy calculations with $\varepsilon =
10^{-5}$, $h_2 \approx 1.6$, and values of $\Delta t$ corresponding to
1000, 2000, 4000, \ldots 16000 time steps. 
In Figure \ref{fig:2derrvdt}, the final-time errors $\err(T)$, measured against a well-converged
reference solution, are plotted against
$\Delta t$ for several approximately logarithmically-spaced values 
of $h_1$. Timings for each choice of $h_1$ and both choices of $h_2$ are given in Table
\ref{tab:2dtimes}.

\begin{figure}[!htb]
  \centering

  \begin{subfigure}[t]{\textwidth}
    \includegraphics[width=\textwidth]{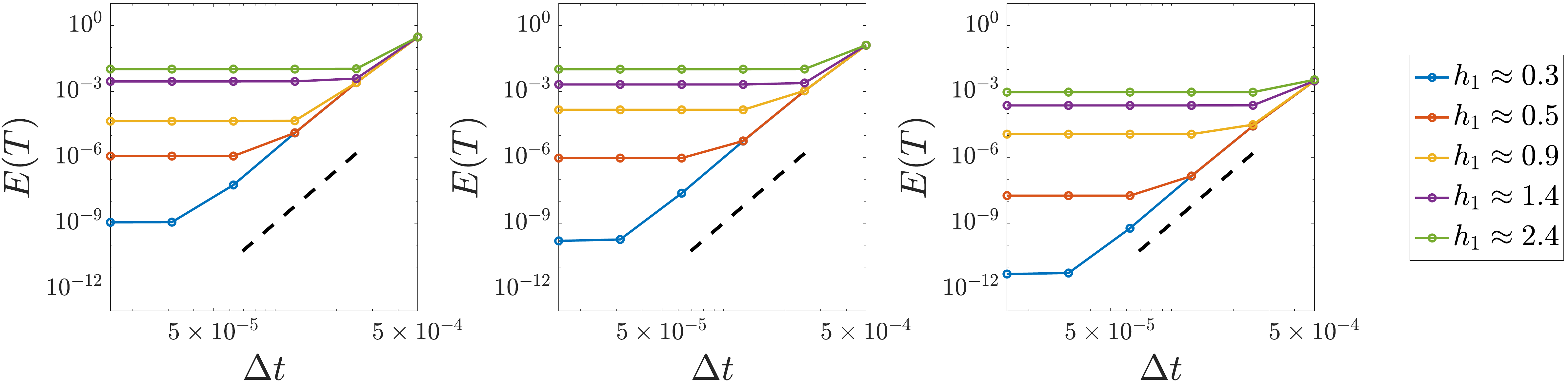}
    \caption{Higher accuracy experiments: $\varepsilon = 10^{-10}$ and
    $h_2 \approx 0.5$.}
    \label{fig:2derrvdteps10}
  \end{subfigure}
  \par\bigskip
  \begin{subfigure}[t]{\textwidth}
    \includegraphics[width=\textwidth]{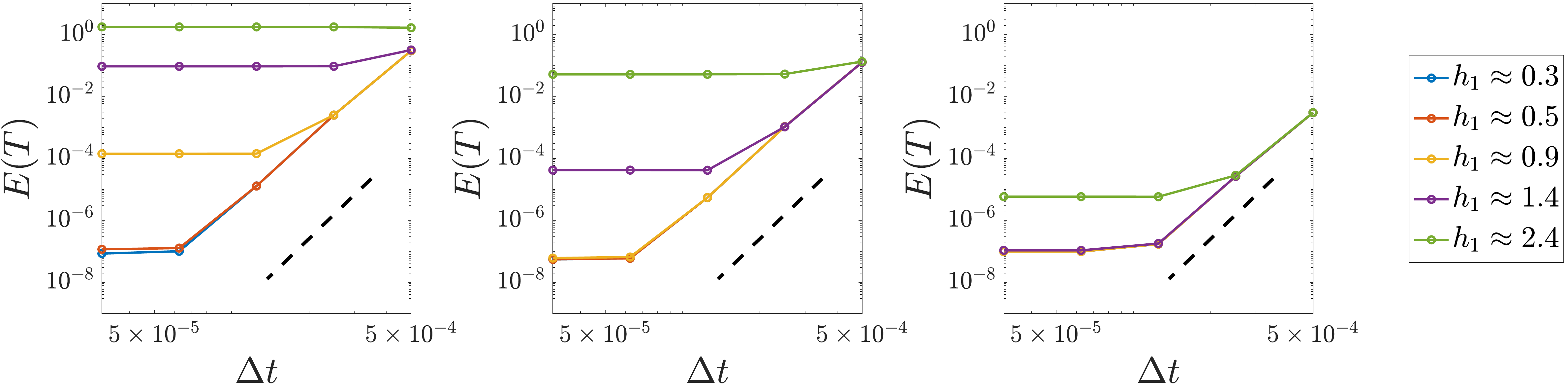}
    \caption{Lower accuracy experiments: $\varepsilon = 10^{-5}$ and
    $h_2 \approx 1.6$.}
    \label{fig:2derrvdteps5}
  \end{subfigure}
  \caption{Final time $L^2$ error of $u(x,t)$ against $\Delta t$ for several values of
      $h_1$ and $\omega =$ $50$ (left), $100$ (middle), and $200$ (right) in Example 4. Eighth-order convergence is indicated by the black
  dashed lines.}
  \label{fig:2derrvdt}
\end{figure}

\begin{table}[ht]
  \centering
  \begin{tabular}{|c|r|r|r|r|r|}
    \hline
    $h_1 \approx $ & 0.3 & 0.5 & 0.9 & 1.4 & 2.4 \\ \hline
    Time steps per second, $h_2 \approx 0.5$ & 14 & 24 & 40 & 66 & 96 \\
    \hline
    Time steps per second, $h_2 \approx 1.6$ & 44 & 70 & 92 & 145 & 193 \\
    \hline
  \end{tabular}
  \caption{Number of time steps per second for the experiments in Example 4.}
  \label{tab:2dtimes}
\end{table}

We remind the reader that increasing $\varepsilon$ also increases
$H_1$ and $H_2$, so that the spectral Green's function is less oscillatory
along $\Gamma$ (see Figure
\ref{fig:specgfun}). Thus $h_1$ and
$h_2$ should be increased with $\varepsilon$ to achieve the
fastest computation for a given accuracy. In the experiment with $\omega
= 100$, for example, to obtain approximately 10 digits of accuracy we
set $\varepsilon = 10^{-10}$, $h_1 \approx 0.3$, $h_2 \approx 0.5$ and take $8000$ time steps at 14 time steps per
second, whereas to obtain approximately 5 digits of accuracy, we can
set $\varepsilon = 10^{-5}$, $h_1 \approx 1.4$, $h_2 \approx 1.6$ and
take $4000$ time steps at 145 times steps per second.

\section{Conclusion} \label{sec:conclusion}

We have introduced a Volterra integral equation-based numerical method
for the periodic and free space TDSE with a spatially-uniform vector potential.
The method offers several notable advantages compared with finite difference
methods and methods based on applying the unitary single time
step propagator. Namely, it permits inexpensive high-order implicit time
stepping, naturally includes the case of time-dependent scalar
potentials, and obviates the need for artificial boundary conditions
in the free space case.

The Volterra integral equation involves 
a spacetime history-dependent volume integral,
and we have used a Fourier method to avoid the computational cost and
memory associated with its naive evaluation. This leads to a fast
and memory-efficient FFT-based method, but requires the solution to be
resolvable on a uniform grid in the physical domain. A new strategy
will be required to make the integral equation formulation compatible
with spatially-adaptive discretizations.

We note lastly that in practical applications, the scalar potential $V$
may be replaced by a somewhat more general object. In time-dependent
density functional theory, for example, the potential is nonlinear and
may be nonlocal. The integral equation approach enjoys
several advantages over PDE-based methods in these cases, which will be
explored in future work.

\section*{Acknowledgements}

We thank Angel Rubio and Umberto de Giovannini for many useful
discussions. J.K. was supported in part by the Research Training Group
in Modeling and Simulation funded by the National Science Foundation via
grant RTG/DMS-1646339.
The Flatiron Institute is a division of the Simons Foundation.

\begin{appendices}

  \section{Proof of Lemma \ref{lem:uentire} for $d = 1$}\label{sec:pflemuentire}

  \begin{proof}
  Fix $t \in [0,T]$ and let $\uhat(\zeta,t)$ be defined by the
  formula \eqref{eq:uhatzinteq}.
  It is well-defined and continuous in $\zeta$ because $\uzhat(\zeta)$ and
  $\Vuhat(\zeta,t)$ are entire functions of $\zeta$, and
  provides a proper extension of $\uhat(\xi,t)$ into the complex plane.
  The integral of $\uhat(\zeta,t)$ around any closed contour in
  $\CC$ is zero---we can interchange the order of
  integration using Fubini's theorem, and apply
  Cauchy's theorem to the analytic integrand---so it follows from Morera's theorem that
  $\uhat(\zeta,t)$ is entire in $\zeta$.

  To obtain \eqref{eq:uciftrep} and \eqref{eq:Vuciftrep}, we write
  \[\int_\Gamma e^{i \zeta x} \wh{f}(\zeta) \, d \zeta = \lim_{K \to
  \infty} \int_{\Gamma_K} e^{i \zeta x} \wh{f}(\zeta) \, d \zeta\]
  where $\Gamma_K$ is the truncation of \eqref{eq:gammadef} to $\tau \in [-K,K]$. We fix $x\in\RR$ and choose $f(x)$ to be either
  $u(x,t)$ or $(Vu)(x,t)$.
  By Cauchy's theorem, the classical inverse Fourier
  transforms \eqref{eq:iftureal} and \eqref{eq:iftVureal} are equal to
  those taken along the
  deformed contour $(-\infty,-K) \cup (-K,-K+iH) \cup \Gamma_K
\cup (K-iH,K) \cup (K,+\infty)$ for any $K$.
  The contributions from $(-\infty,-K)$ and $(K,\infty)$ vanish as
  $K\to\infty$ because $u$ and $Vu$, and therefore $\uhat$ and $\Vuhat$,
  are in the Schwartz space. Thus to
prove \eqref{eq:uciftrep} and \eqref{eq:Vuciftrep}, we only need to show
  that the contributions from the two vertical
  segments $(-K,-K+iH)$ and $(K-iH,K)$ vanish in that limit, i.e.\
  \[\lim_{K \to \infty} \int_0^H e^{i (K - i \eta) x} \wh{f}(K - i
    \eta) \, d \eta = \lim_{K \to \infty} \int_0^H e^{i (-K + i \eta)
  x} \wh{f}(-K + i \eta) \, d \eta = 0\]
  for $f(x) = u(x,t)$ and $f(x) = (Vu)(x,t)$. For the latter, we write
  \begin{align*}
    \int_0^H e^{i (K - i \eta) x} \Vuhat(K - i \eta,t) \, d \eta &=
  \int_0^H e^{i (K - i \eta) x} \int_{-\infty}^\infty e^{-i (K-i\eta) y}
    (Vu)(y,t) \, dy \, d \eta \\ 
    &= \int_{-\infty}^\infty e^{i K (x-y)} (V u)(y,t) \int_0^H e^{\eta
    (x-y)} d \eta \, dy.
  \end{align*}
  Here, noting that $V$ is smooth and compactly supported, we have used
  Fubini's theorem to switch the order of integration. The inner
  integral is a smooth function, so the outer integral is the Fourier
  transform of a smooth, compactly supported function, evaluated at $K$. The desired
  result then follows from the Riemann--Lebesgue lemma.

  For $f(x) = u(x,t)$, we instead use \eqref{eq:uhatAinteq} to write
  \begin{multline*}
    \int_0^H e^{i (K - i \eta) x} \uhat(K - i \eta,t) \, d \eta =
    \int_0^H e^{i (K - i \eta) x} \wh{G}(K-i\eta,t,0) \wh{u}_0(K-i\eta) \, d \eta \\
    -  i \int_0^H e^{i (K - i \eta) x} \int_0^t \wh{G}(K-i\eta,t,s)
    \wh{(Vu)}(K-i\eta,s) \, ds \, d\eta
  \end{multline*}
  Let us consider the first term on the right hand side; the second may
  be dealt with by a similar approach. We again use Fubini's theorem to
  obtain
\be
  \int_0^H e^{i (K - i \eta) x} \wh{G}(K-i\eta,t,0) \wh{u}_0(K-i\eta) \,
  d \eta
  =
  e^{i K x} \int_{-\infty}^\infty e^{-i K y} u_0(y) \int_0^H
  e^{\eta (x-y)} \wh{G}(K-i\eta,t,0) \, d \eta \, dy~.
  \label{fubini}
\ee
We have $|e^{\eta (x-y)} \wh{G}(K-i\eta,t,0)| \le e^{H(|x|+1+\phimax)}$
for all $y\in[-1,1]$ and $\eta\in[0,H]$, where
$\phimax$ is given by \eqref{eq:phimax}. Therefore the inner integral
defines a bounded, continuous function of $y \in [-1,1]$. Since $u_0$ is
a smooth function supported on $[-1,1]$, the outer integral is the
Fourier transform of an integrable function evaluated at $K$, and the
result again follows from the Riemann--Lebesgue lemma.
\end{proof}

\end{appendices}

\bibliographystyle{ieeetr}
{\footnotesize \bibliography{kbg_tdse}}

\end{document}